\documentclass[11pt, a4paper,twoside]{amsart}

\usepackage[french,english]{babel}
\usepackage{amsfonts}
\usepackage{amsmath,amssymb,amscd}
\usepackage{mathrsfs}  
\usepackage{amsthm}
\usepackage{a4wide}   
\usepackage[T1]{fontenc} 
\usepackage{newlfont} 
\usepackage{color}
\usepackage{stackrel}

\newcommand{\cache}[1]{}

\newcommand{\modif}[1]{\marginpar{modif}\red{#1}}
\newcommand{\red}[1]{{\color{red}#1}}

\usepackage[all]{xy}  

\def\choixcompteur{subsection}
\newtheorem{theo}[\choixcompteur]{Theorem}

\newtheorem{prop}[\choixcompteur]{Proposition}
\newtheorem{lemm}[\choixcompteur]{Lemma}
\newtheorem{coro}[\choixcompteur]{Corollary}

\theoremstyle{definition}

\newtheorem{defi}[\choixcompteur]{Definition}

\newtheorem{remas}[\choixcompteur]{Remarks}

\newtheorem*{exem*}{Example}
\newtheorem*{exems*}{Examples}
\newtheorem*{exam*}{Example}
\newtheorem*{exams*}{Examples}
\newtheorem*{rema*}{Remark}
\newtheorem*{remas*}{Remarks}
\newtheorem*{NB}{N.B}

\theoremstyle{definition}
\newtheorem*{defi*}{Definition}
\newtheorem*{defiprop*}{Definition-Proposition}

\theoremstyle{plain}
\newtheorem*{prop*}{Proposition}
\newtheorem*{lemm*}{Lemma}
\newtheorem*{coro*}{Corollary}
\newtheorem*{theo*}{Theorem}

 \def\cdr@enoncedef{%
 \newenvironment{enonce*}[2][plain]%
 {\let\cdrenonce\relax \theoremstyle{##1}%
 \newtheorem*{cdrenonce}{##2}%
 \begin{cdrenonce}}%
 {\end{cdrenonce}}   }%

 \AtBeginDocument{%
    \cdr@enoncedef}

\def\cf{{\it cf.\/}\ }
\def\ie{{\it i.e.\/}\ }
\def\eg{{\it e.g.\/}\ }
\def\etc{{\it etc.\/}\ }
\def\lc{{\it l.c.\/}\ }

\def\truc{\unskip\kern 3pt\penalty 500
\hbox{\vrule\vbox to 5pt{\hrule width 4pt\vfill\hrule}\vrule}\kern
3pt}

\def\qed{\nobreak\hfill $\truc$\par\goodbreak}

\def\vect{\overrightarrow}

\def\parni{\par\noindent}

\def\eds{ editors}

\def\N{{\mathbb N}}    
\def\Z{{\mathbb Z}}
\def\Q{{\mathbb Q}}
\def\R{{\mathbb R}}
\def\C{{\mathbb C}}

\def\F{{\mathbb F}}
\def\A{{\mathbb A}}
\def\M{{\mathbb M}}

\newcommand{\g}[1]{\mathfrak{#1}} 

\def\qa{\alpha}     
\def\qb{\beta}
\def\qd{\delta}
\def\qe{\varepsilon}
\def\qf{\varphi}

 \def\ql{\lambda}
\def\qm{\mu}
\def\qn{\nu}
\def\qo{\omega}
\def\qp{\pi}
\def\qr{\rho}
\def\qs {\sigma}
\def\qt{\tau}

\def\qx{\xi}
 \def\qz{\zeta}

\def\QD{\Delta}
\def\QF{\Phi}
\def\QG{\Gamma}
\def\QL{\Lambda}
\def\QO{\Omega}
\def\QP{\Pi}

\def\sha{{\mathcal A}}   

\def\shc{{\mathcal C}}

\def\shh{{\mathcal H}}

\def\shk{{\mathcal K}}

\def\shm{{\mathcal M}}

\def\shq{{\mathcal Q}}

\def\sht{{\mathcal T}}

\def\SHC{{\mathscr C}}

\def\SHI{{\mathscr I}}

\begin{document}

\title{Iwahori-Hecke algebras for Kac-Moody groups \goodbreak over local fields}
\author{Nicole Bardy-Panse, St\'ephane Gaussent and Guy Rousseau}

\date{May 10, 2016}

%




\maketitle



\begin{abstract}
We define the Iwahori-Hecke algebra $^I\shh$ for an almost split Kac-Moody group $G$ over a local non-archimedean field.
We use the hovel $\SHI$ associated to this situation, which is the analogue of the Bruhat-Tits building for a reductive group.
 The fixer $K_I$ of some chamber in the standard apartment plays the role of the Iwahori subgroup.
 We can define $^I\shh$ as the algebra of some $K_I-$bi-invariant functions on $G$ with support consisting of a finite union of double classes.
 As two chambers in the hovel are not always in a same apartment, this support has to be in some large subsemigroup $G^+$ of $G$.
 In the split case, we prove that the structure constants of $^I\shh$ are polynomials in the cardinality of the residue field, with integer coefficients depending on the geometry of the standard apartment.
We give a presentation of this algebra $^I\shh$, similar to the Bernstein-Lusztig presentation in the reductive case, and embed it in a greater algebra $^{BL}\shh$, algebraically defined by the Bernstein-Lusztig presentation.
In the affine case, this algebra contains the Cherednik's double affine Hecke algebra. Actually, our results apply to abstract ``locally finite'' hovels, so that we can define the Iwahori-Hecke algebra with unequal parameters.

\end{abstract}


\setcounter{tocdepth}{1}    
\tableofcontents

\section*{Introduction}
\label{seIntro}

\subsection*{A bit of history}

Iwahori-Hecke algebras were first introduced in arithmetics by Erich Hecke in the '30s \cite{He37}. He defined an algebra, now called the Hecke algebra, generated by some operators on modular forms. Then in the late '50s, based on an idea of AndrÈ Weil, Goro Shimura \cite{Shi59} defined an algebra attached to a group containing a subgroup (under some hypotheses) as the algebra spanned by some double cosets and recovered Hecke's algebra. In 1964, Nagayoshi Iwahori \cite{Iwa64} showed that, in the case of a Chevalley group over a finite field containing a Borel subgroup, Shimura's algebra can be defined in terms of bi-invariant functions on the group. He further gave a presentation by generators and relations of this algebra. Examples of such groups containing a suitable subgroup are given by BN-pairs and the theory of buildings. Nagayoshi Iwahori and Hideya Matsumoto \cite{IM65} found a famous instance in a Chevalley group over a $p$-adic field corresponding to the Bruhat-Tits building associated to the situation. In fact, it is possible to define these algebras only in terms of building theory, see \eg \cite{P06}, for a contemporary treatment. 

In a previous article \cite{GR08}, the last two authors introduced the analogue of the Bruhat-Tits building in Kac-Moody theory, and called it, a hovel. Guy Rousseau developed the notion further and gave in \cite{R11} an axiomatic definition allowing to deal with a broader context.

In this paper, we first define, in terms of the hovel, the Iwahori-Hecke algebra associated to a Kac-Moody group over a local field containing the equivalent of the Iwahori subgroup. Then, we study the properties of this algebra, like the structure constants of the product, some presentations by generators and relations, and an interesting example where we recover the Double Affine Hecke Algebras.

In the remaining of the introduction, we give a more detailed account of our work.

\subsection*{The case of simple algebraic groups}

To begin, we recall the situation in the finite dimensional case. Let $\shk$ be a local non-archimedean field, with residue field $\F_q$.
Suppose $G$ is a split, simple and simply connected algebraic group over $\shk$ and  $K$  an open compact subgroup.
The space $\shh_{K}$ of complex functions on $G$, bi-invariant by $K$ and with compact support, is an algebra for the natural convolution product.
 Ichiro Satake \cite{Sa63} studied such algebras to define the spherical functions and proved, in particular, that $\shh_{K}$ is commutative for a good choice $K_s$ of $K$, maximal compact. The corresponding convolution algebra $\shh_{K_s}={^s\shh}(G)$ is now called the spherical Hecke algebra.
 From the work of Nagayoshi Iwahori and Hideya Matsumoto \cite{IM65}, we know that there exists an interesting open subgroup $K_I$, so called the Iwahori subgroup, of $K_s$ with a Bruhat decomposition $G=K_I.W.K_I$, where $W$ is an infinite Coxeter group.
 The corresponding convolution algebra $\shh_{K_I}={^I\shh}(G)$, called the Iwahori-Hecke algebra,  may be described as the abstract Hecke algebra associated to this Coxeter group and the parameter $q$.
 There is another presentation of this Hecke algebra, stated by Joseph  Bernstein and proved in the most general case by George Lusztig \cite{Lu89}. This presentation emphasizes the role of the translations in $W$ and uses  new relations, now often called the Bernstein-Lusztig relations.
In the building-like definition of these algebras, the group $K_s$ (resp. $K_I$) is the fixer of a special vertex (resp. a chamber) for the action of $G$ on the Bruhat-Tits building $\SHI$, \cite{BrT72}. 
 
\subsection*{The Kac-Moody setting}

Kac-Moody groups are interesting generalizations of semisimple groups, hence it is natural to define the Iwahori-Hecke algebras also in the Kac-Moody setting.


 So, from now on, let $G$ be a Kac-Moody group over $\shk$, assumed minimal or ``algebraic'', \ie as studied by Jacques Tits \cite{T87} in the split case and by Bertrand R\'emy \cite{Re02} in the almost split case.
  Unfortunately there is, up to now, no good topology on $G$ and no good compact subgroup, so the ``convolution product'' has to be defined by other means.
  Alexander Braverman and David Kazhdan \cite{BrK11} succeeded in defining geometrically such a spherical Hecke algebra, when $G$ is split and untwisted affine, see also the survey \cite{BrK14} by the same authors.
  We were able, in \cite{GR13}, to generalize their construction to any Kac-Moody group over $\shk$.
  In \cite{BrKP14}, using  results of \cite{Ga95} and \cite{BrGKP14}, Alexander Braverman, David Kazhdan and Manish Patnaik  construct the spherical Hecke algebra and the Iwahori-Hecke algebra by algebraic computations in the Kac-Moody group, still assumed split and untwisted affine (and even simply laced for some statements).
  Those algebras are convolution algebras of functions on $G$ bi-invariant under some analogue group $K_s$ or $K_I$ ($\subset K_s$), but there are two new features: the support has to be in a subsemigroup $G^+$ of $G$ and the description of the Iwahori-Hecke algebra has to use Bernstein-Lusztig type relations since $W$ is no longer a Coxeter group.

\subsection*{Iwahori-Hecke algebras in the Kac-Moody setting}

 As in \cite{GR13}, our idea is to define the Iwahori-Hecke algebra using the hovel associated to the almost split Kac-Moody group $G$ that we built in \cite{GR08}, \cite{R11} and \cite{R12}.
  This hovel $\SHI$ is a set with an action of $G$ and a covering by subsets called apartments. They are in one-to-one correspondence with the maximal split subtori, hence permuted transitively by $G$.
  Each apartment $A$ is a finite dimensional real affine space. Its stabilizer $N$ in $G$ acts on $A$ via a generalized affine Weyl group $W=W^v\ltimes Y$, where $Y\subset\vect A$ is a discrete subgroup of translations. The group $W$ stabilizes a set $\shm$ of affine hyperplanes called walls.
  So, $\SHI$ looks much like the Bruhat-Tits building of a reductive group. But as the root system $\QF$ is infinite, the set of walls $\shm$ is not locally finite. Further, two points in $\SHI$ are not always in a same apartment. This is why $\SHI$ is called a hovel. 
  However, there exists on $\SHI$ a $G-$invariant preorder $\leq$ which induces on each apartment $A$ the preorder given by the Tits cone $\sht\subset\vect A$.

  \par Now, we consider the fixer $K_I$ in $G$ of some (local) chamber $C_0^+$ in a chosen standard apartment $\A$; it is our Iwahori subgroup. Fix a ring $R$.
  The Iwahori-Hecke algebra $^I\shh_R$ will be defined as the space of some $K_I$-bi-invariant functions on $G$ with values in $R$.
  In other words, it will be the space $^I\shh_R^\SHI$ of some $G$-invariant functions on $\shc^+_0\times\shc^+_0$, where $\shc^+_0=G/K_I$ is the orbit of $C_0^+$ in the set $\shc$ of chambers of $\SHI$.
  The convolution product is easy to guess from this point of view:
$$
(\qf*\psi)(C_x,C_y)=\sum_{C_z\in\shc^+_0}\,\qf(C_x,C_z)\psi(C_z,C_y)
$$ (if this sum means something).
As for points two chambers in $\SHI$ are not always in a same apartment, \ie the Bruhat-Iwahori decomposition fails: $G\neq K_I.N.K_I$. So, we have to consider pairs of chambers $(C_x,C_y)\in\shc^+_0\times_\leq\shc^+_0$, \ie $C_x$ (resp. $C_y$) $\in \shc^+_0$ has $x$ (resp. $y$) for vertex and  $x\leq y$. This implies that $C_x,C_y$ are in a same apartment.
  For $^I\shh_R$, this means that the support of $\qf\in\shh_R$ has to be in $K_I\backslash G^+/K_I$ where $G^+=\{g\in G\mid 0\leq g.0\}$ is a semigroup. We suppose moreover this support to be finite.
  In addition, $K_I\backslash G^+/K_I$ is in bijective correspondence with the subsemigroup $W^+=W^v\ltimes Y^+$ of $W$, where $Y^+=Y\cap \sht$.

  \par With this definition we are able to prove that $^I\shh_R$ is really an algebra, which generalizes the known Iwahori-Hecke algebras in the semi-simple case (see \S \ref{s2}).

\subsection*{The structure constants}

The structure constants of $^I\shh_R$ are the non-negative integers $a_{\mathbf w, \mathbf v}^{\mathbf u}$, for $\mathbf w, \mathbf v,\mathbf u\in W^+$, such that
$$
T_{\mathbf w} * T_{\mathbf v} = \sum_{\mathbf u\in W^+} a_{\mathbf w, \mathbf v}^{\mathbf u} T_{\mathbf u},
$$ where $T_{\mathbf w} $ is the characteristic function of $K_I.{\mathbf w}.K_I$ and the sum is finite.
  Each chamber in $\SHI$ has only a finite number of adjacent chambers along a given panel. These numbers are called the parameters of $\SHI$ and form a finite set $\shq$.
  In the split case, there is only one parameter $q$: the number of elements of the residue field of $\shk$.
  We conjecture that each $a_{\mathbf w, \mathbf v}^{\mathbf u}$ is a polynomial in these parameters with integral coefficients depending only on the geometry of the model apartment $\A$ and on $W$.
  We prove this only partially: this is true if $G$ is split or if we replace ``polynomial'' by ``Laurent polynomial'' (\cf \ref{5.7}); this is also true for $\mathbf w, \mathbf v$ ``generic'' (\cf \ref{sc8}).
  Actually in the generic case, we give, in section \ref{sc}, an explicit formula for $a_{\mathbf w, \mathbf v}^{\mathbf u}$.

\subsection*{Generators and relations}

If the parameters in $\shq$ are invertible in the ring $R$, we are able, in section \ref{s3},
 to deduce from the geometry of $\SHI$ a set of generators and some relations in $^I\shh_R$.
 The family $(T_\ql* T_w)_{\ql\in Y^+, w\in W^v} $ is an $R$-basis of $^I\shh_R$. And the subalgebra $\sum_{w\in W^v}\,R.T_w$ is the abstract Hecke algebra $\shh_R(W^v)$ associated to the Coxeter group $W^v$, generated by the $T_i=T_{r_i}$, where the $r_i$ are the fundamental reflections in $W^v$.
 So, $^I\shh_R$ is a free right $\shh_R(W^v)$-module.
 We get also some commuting relations between the $T_\ql$ and the $T_w$, including some relations of Bernstein-Lusztig type (see Theorem \ref{3.8}).

 \par From all these relations, we deduce algebraically in section \ref{s4} that there exists a new basis $(X^\ql* T_w)_{\ql\in Y^+, w\in W^v} $  of $^I\shh_R$, associated to some new elements $X^\ql\in{^I\shh_R}$.
 These elements satisfy $X^\ql=T_\ql$ for $\ql\in Y^{++}=Y\cap \overline{C^v_f}$, where $C^v_f$ is the fundamental Weyl chamber, and $X^\ql*X^\qm=X^{\ql+\qm}=X^\qm*X^\ql$ for $\ql,\qm\in Y^+$.
 As, for any $\ql\in Y^+$, there is $\qm\in Y^{++}$ with $\ql+\qm\in Y^{++}$, these $X^\ql$ are some quotients of some elements $T_\qm$.
 The Bernstein-Lusztig type relations may be translated to this new basis.
 When $R$ contains sufficiently high roots of the parameters in $\shq$ (\eg if $R\supset\R$), we may replace the $T_w$ and $X^\ql$ by some $R^\times$-multiples $H_w$ and $Z^\ql$.
 We get a new basis $(Z^\ql* H_w)_{\ql\in Y^+, w\in W^v} $  of $^I\shh_R$, satisfying a set of relations very close to the Bernstein-Lusztig presentation in the semi-simple case (\cf \ref{4.7}).

 \par In section \ref{s5}, we define algebraically the Bernstein-Lusztig-Hecke algebra $^{BL}\shh_{R_1}$: it is the free module with basis written $(Z^\ql H_w)_{\ql\in Y^+, w\in W^v} $ over the algebra $R_1=\Z[({\qs_i}^{\pm 1},{\qs_i'}^{\pm 1})_{i\in I}]$, where $\qs_i,\qs_i'$ are indeterminates (with some identifications). The product $*$ is given by the same relations as above for the $Z^\ql* H_w$, one just extends $\ql\in Y^+$ to $\ql\in Y$ and replace $\sqrt{q_i},\sqrt{q_i'}$ by $\qs_i,\qs_i'$.
 We prove then that, up to a change of scalars, $^I\shh_R$ may be identified to a subalgebra of $^{BL}\shh_{R_1}$.
 This Bernstein-Lusztig algebra may be considered as a ring of quotients of the Iwahori-Hecke algebra.

\subsection*{Ordered affine hovel}

Actually, this article is written in a more general framework (explained in \S \ref{s1}): we work with $\SHI$ an abstract ordered affine hovel (as defined in \cite{R11}), and we take $G$ to be a strongly transitive group of (positive, ``vectorially Weyl'') automorphisms.
 In section \ref{s7}, we drop the assumption that $G$ is vectorially Weyl to define extended versions $^I{\widetilde\shh}$ and $^{BL}{\widetilde\shh}$ of $^I{\shh}$ and $^{BL}{\shh}$.
In the affine case, we prove that they are graded algebras and that the summand of degree $0$ of $^{BL}{\widetilde\shh}$ is very close to Cherednik's double affine Hecke algebra.

\section{General framework}\label{s1}

\subsection{Vectorial data}\label{1.1}  We consider a quadruple $(V,W^v,(\qa_i)_{i\in I}, (\qa^\vee_i)_{i\in I})$ where $V$ is a finite dimensional real vector space, $W^v$ a subgroup of $GL(V)$ (the vectorial Weyl group), $I$ a finite set, $(\qa^\vee_i)_{i\in I}$ a family in $V$ and $(\qa_i)_{i\in I}$ a free family in the dual $V^*$.
 We ask these data to satisfy the conditions of \cite[1.1]{R11}.
  In particular, the formula $r_i(v)=v-\qa_i(v)\qa_i^\vee$ defines a linear involution in $V$ which is an element in $W^v$ and $(W^v,\{r_i\mid i\in I\})$ is a Coxeter system.

  \par To be more concrete, we consider the Kac-Moody case of [\lc; 1.2]: the matrix $\M=(\qa_j(\qa_i^\vee))_{i,j\in I}$ is a generalized Cartan matrix.
  Then $W^v$ is the Weyl group of the corresponding Kac-Moody Lie algebra $\g g_\M$ and the associated real root system is
$$
\QF=\{w(\qa_i)\mid w\in W^v,i\in I\}\subset Q=\bigoplus_{i\in I}\,\Z.\qa_i.
$$ We set $\QF^\pm{}=\QF\cap Q^\pm{}$ where $Q^\pm{}=\pm{}(\bigoplus_{i\in I}\,(\Z_{\geq 0}).\qa_i)$ and $Q^\vee=(\bigoplus_{i\in I}\,\Z.\qa_i^\vee)$, $Q^\vee_\pm{}=\pm{}(\bigoplus_{i\in I}\,(\Z_{\geq 0}).\qa_i^\vee)$.
   We have  $\QF=\QF^+\cup\QF^-$ and, for $\qa=w(\qa_i)\in\QF$, $r_\qa=w.r_i.w^{-1}$ and $r_\qa(v)=v-\qa(v)\qa^\vee$, where the coroot $\qa^\vee=w(\qa_i^\vee)$ depends only on $\qa$.

\par The set $\QF$ is an (abstract, reduced) real root system in the sense of \cite{MP89}, \cite{MP95} or \cite{Ba96}.
We shall sometimes also use the set $\QD=\QF\cup\QD^+_{im}\cup\QD^-_{im}$ of all roots (with $-\QD^-_{im}=\QD^+_{im}\subset Q^+$,  $W^v-$stable) defined in \cite{K90}.
 It is an (abstract, reduced) root system in the sense of \cite{Ba96}.

  \par The {\it fundamental positive chamber} is $C^v_f=\{v\in V\mid\qa_i(v)>0,\forall i\in I\}$.
   Its closure $\overline{C^v_f}$ is the disjoint union of the vectorial faces $F^v(J)=\{v\in V\mid\qa_i(v)=0,\forall i\in J,\qa_i(v)>0,\forall i\in I\setminus J\}$ for $J\subset I$. We set $V_0 = F^v(I)$.
    The positive (resp. negative) vectorial faces are the sets $w.F^v(J)$ (resp. $-w.F^v(J)$) for $w\in W^v$ and $J\subset I$.
    The support of such a face is the vector space it generates.
    The set $J$ or the face $w.F^v(J)$ or an element of this face is called {\it spherical} if the group $W^v(J)$ generated by $\{r_i\mid i\in J\}$ is finite.
    An element of a vectorial chamber $\pm w.C^v_f$ is called {\it regular}.

    \par The {\it Tits cone}  $\sht$ (resp. its interior $\sht^\circ$) is the (disjoint) union of the positive (resp. and spherical) vectorial faces. It is a $W^v-$stable convex cone in $V$.

  \par
  We say that $\A^v=(V,W^v)$ is a {\it vectorial apartment}.
  A {\it positive automorphism} of $\A^v$ is a linear bijection $\qf:\A^v\to\A^v$ stabilizing $\sht$ and permuting the roots and corresponding coroots; so it normalizes $W^v$ and permutes the vectorial walls $M^v(\qa)=Ker(\qa)$.
 As $W^v$ acts simply transitively on the positive (resp. negative) vectorial chambers, any subgroup $\widetilde W^v$ of the group $Aut^+(\A^v)$ (of positive automorphisms of $\A^v$) containing $W^v$ may be written $\widetilde W^v=\QO\ltimes W^v$, where $\QO$ is the stabilizer in $\widetilde W^v$ of $C^v_f$ (and $-C^v_f$).
 This group $\QO$ induces a group of permutations of $I$ (by $\qo(\qa_i)=\qa_{\qo(i)}$, $\qo(\qa_i^\vee)=\qa_{\qo(i)}^\vee$); but it may be greater than the whole group of permutations in general (even infinite if $(\cap Ker\qa_i)\not=\{0\}$).

\subsection{The model apartment}\label{1.2} As in \cite[1.4]{R11} the model apartment $\A$ is $V$ considered as an affine space and endowed with a family $\shm$ of walls. 
 These walls  are affine hyperplanes directed by Ker$(\qa)$ for $\qa\in\QF$.

 \par We ask this apartment to be {\bf semi-discrete} and the origin $0$ to be {\bf special}.
  This means that these walls are the hyperplanes defined as follows:
$$M(\qa,k)=\{v\in V\mid\qa(v)+k=0\}\qquad\text{for }\qa\in\QF\text{ and } k\in\QL_\qa,$$ with $\QL_\qa=k_\qa.\Z$ a non trivial discrete subgroup of $\R$. 
Using Lemma 1.3 in \cite{GR13} (\ie replacing $\QF$ by another system $\QF_1$) we may (and shall) assume that $\QL_\qa=\Z, \forall\qa\in\QF$.


  \par For $\qa=w(\qa_i)\in\QF$, $k\in\Z$ and $M=M(\qa,k)$, the reflection $r_{\qa,k}=r_M$ with respect to $M$ is the affine involution of $\A$ with fixed points the wall $M$ and associated linear involution $r_\qa$.
   The affine Weyl group $W^a$ is the group generated by the reflections $r_M$ for $M\in \shm$; we assume that $W^a$ stabilizes $\shm$.
We know that $W^a=W^v\ltimes Q^\vee$ and we write $W^a_\R=W^v\ltimes V$; here $Q^\vee$ and $V$ have to be understood as groups of translations.

   \par An automorphism of $\A$ is an affine bijection $\qf:\A\to\A$ stabilizing the set of pairs $(M,\qa^\vee)$ of a wall $M$ and the coroot associated with $\qa\in\QF$ such that $M=M(\qa,k)$, $k\in\Z$. The group $Aut(\A)$ of these automorphisms contains $W^a$ and normalizes it.
We consider also the group $Aut^W_\R(\A)=\{\qf\in Aut(\A)\mid\vect{\qf}\in W^v\}=Aut(\A)\cap W^a_\R$.

   \par For $\qa\in\QF$ and $k\in\R$, $D(\qa,k)=\{v\in V\mid\qa(v)+k\geq 0\}$ is an half-space, it is called an {\it half-apartment} if $k\in\Z$. We write  $D(\alpha,\infty) = \mathbb A$.


The Tits cone $\mathcal T$ 
and its interior $\mathcal T^o$ are convex and $W^v-$stable cones, therefore, we can define two $W^v-$invariant preorder relations  on $\mathbb A$: 
$$
x\leq y\;\Leftrightarrow\; y-x\in\mathcal T
; \quad x\stackrel{o}{<} y\;\Leftrightarrow\; y-x\in\mathcal T^o.
$$
 If $W^v$ has no  fixed point in $V\setminus\{0\}$ and no finite factor, then they are orders; but, in general, they are not.


\subsection{Faces, sectors, chimneys...}
\label{suse:Faces}

 The faces in $\mathbb A$ are associated to the above systems of walls
and half-apartments
. As in \cite{BrT72}, they
are no longer subsets of $\mathbb A$, but filters of subsets of $\mathbb A$. For the definition of that notion and its properties, we refer to \cite{BrT72} or \cite{GR08}.

If $F$ is a subset of $\mathbb A$ containing an element $x$ in its closure,
the germ of $F$ in $x$ is the filter $\mathrm{germ}_x(F)$ consisting of all subsets of $\mathbb A$ which contain intersections of $F$ and neighbourhoods of $x$. In particular, if $x\neq y\in \mathbb A$, we denote the germ in $x$ of the segment $[x,y]$ (resp. of the interval $]x,y]$) by $[x,y)$ (resp. $]x,y)$).

Given $F$ a filter of subsets of $\mathbb A$, its {\it enclosure} $cl_{\mathbb A}(F)$ (resp. {\it closure} $\overline F$) is the filter made of  the subsets of $\mathbb A$ containing an element of $F$ of the shape $\cap_{\alpha\in\Delta}D(\alpha,k_\alpha)$, where $k_\alpha\in\mathbb Z\cup\{\infty\}$ (resp. containing the closure $\overline S$ of some $S\in F$).

\medskip

A {\it local face} $F$ in the apartment $\mathbb A$ is associated
 to a point $x\in \mathbb A$, its vertex, and a  vectorial face $F^v$ in $V$, its direction. It is defined as $F=germ_x(x+F^v)$ and we denote it by $F=F^\ell(x,F^v)$.
 Its closure is $\overline{F^\ell}(x,F^v)=germ_x(x+\overline{F^v})$

There is an order on the local faces: the assertions ``$F$ is a face of $F'$ '',
``$F'$ covers $F$ '' and ``$F\leq F'$ '' are by definition equivalent to
$F\subset\overline{F'}$.
 The dimension of a local face $F$ is the smallest dimension of an affine space generated by some $S\in F$.
  The (unique) such affine space $E$ of minimal dimension is the support of $F$; if $F=F^\ell(x,F^v)$, $supp(F)=x+supp(F^v)$.
 A local face $F=F^\ell(x,F^v)$ is spherical if the direction of its support meets the open Tits cone (\ie when $F^v$ is spherical), then its  pointwise stabilizer $W_F$ in $W^a$ is finite.

 \par We shall actually here speak only of local faces, and sometimes forget the word local.

\medskip
 Any point $x\in \mathbb A$ is contained in a unique face $F(x,V_0)\subset cl_\A(\{x\})$ which is minimal of positive and negative direction (but seldom spherical). 
  For any local face $F^\ell=F^\ell(x,F^v)$, there is a unique face $F$ (as defined in \cite{R11}) containing $F^\ell$.
 Then $\overline{F^\ell}\subset\overline F=cl_\A(F^\ell)=cl_\A(F)$ is also the enclosure of any interval-germ $]x,y)=germ_x(]x,y])$ included in $F^\ell$.

 A {\it local chamber}  is a maximal local face, \ie a local face $F^\ell(x,\pm w.C^v_f)$ for $x\in\A$ and $w\in W^v$.
 The {\it fundamental local chamber} is $C_0^+=germ_0(C^v_f)$.

A {\it (local) panel} is a spherical local face maximal among local faces which are not chambers, or, equivalently, a spherical face of dimension $n-1$. Its support is a wall.

\medskip
 A {\it sector} in $\mathbb A$ is a $V-$translate $\mathfrak s=x+C^v$ of a vectorial chamber
$C^v=\pm w.C^v_f$, $w \in W^v$. The point $x$ is its {\it base point} and $C^v$ its  {\it direction}.  Two sectors have the same direction if, and only if, they are conjugate
by $V-$translation,
 and if, and only if, their intersection contains another sector.

 The {\it sector-germ} of a sector $\mathfrak s=x+C^v$ in $\mathbb A$ is the filter $\mathfrak S$ of
subsets of~$\mathbb A$ consisting of the sets containing a $V-$translate of $\mathfrak s$, it is well
determined by the direction $C^v$. So, the set of
translation classes of sectors in $\mathbb A$, the set of vectorial chambers in $V$ and
 the set of sector-germs in $\mathbb A$ are in canonical bijection.
  We denote the sector-germ associated to the negative fundamental vectorial chamber $-C^v_f$ by $\g S_{-\infty}$.

 A {\it sector-face} in $\mathbb A$ is a $V-$translate $\mathfrak f=x+F^v$ of a vectorial face
$F^v=\pm w.F^v(J)$. The sector-face-germ of $\mathfrak f$ is the filter $\mathfrak F$ of
subsets containing a translate $\mathfrak f'$ of $\mathfrak f$ by an element of $F^v$ ({\it i.e.} $\mathfrak
f'\subset \mathfrak f$). If $F^v$ is spherical, then $\mathfrak f$ and $\mathfrak F$ are also called
spherical. The sign of $\mathfrak f$ and $\mathfrak F$ is the sign of $F^v$.

\medskip
A {\it chimney} in $\mathbb A$ is associated to a face $F=F(x, F_0^v)$, called its basis, and to a vectorial face $F^v$, its direction, it is the filter
$$
\mathfrak r(F,F^v) = cl_{\mathbb A}(F+F^v).
$$ A chimney $\mathfrak r = \mathfrak r(F,F^v)$ is {\it splayed} if $F^v$ is spherical, it is {\it solid} if its support (as a filter, i.e. the smallest affine subspace containing $\mathfrak r$) has a finite  pointwise stabilizer in $W^v$. A splayed chimney is therefore solid. The enclosure of a sector-face $\mathfrak f=x+F^v$ is a chimney.

 \par A  ray $\delta$ with origin in $x$ and containing $y\not=x$ (or the interval $]x,y]$, the segment $[x,y]$) is called {\it preordered} if $x\leq y$ or $y\leq x$ and {\it generic} if $x\stackrel{o}{<} y$ or $y\stackrel{o}{<} x$.
 With these new notions, a chimney can be defined as the enclosure of a preordered  ray and a preordered segment-germ sharing the same origin. The chimney is splayed if, and only if, the  ray is generic.

 \subsection{The hovel}\label{1.3}

 In this section, we recall the definition and some properties of an ordered affine hovel given by Guy Rousseau in \cite{R11}.

\parni{\bf 1)} An apartment of type $\mathbb A$ is a set $A$ endowed with a set $Isom^W\!(\mathbb A,A)$ of bijections (called Weyl-isomorphisms) such that, if $f_0\in Isom^W\!(\mathbb A,A)$, then $f\in Isom^W\!(\mathbb A,A)$ if, and only if, there exists $w\in W^a$ satisfying $f = f_0\circ w$.
An isomorphism (resp. a Weyl-isomorphism, a vectorially-Weyl isomorphism) between two apartments $\varphi :A\to A'$ is a bijection such that
, for any $f\in Isom^W\!(\mathbb A,A)$, $f'\in Isom^W\!(\mathbb A,A')$, $f'^{-1}\circ\qf\circ f\in Aut(\A)$ (resp. $\in W^a$, $\in Aut^W_\R(\A)$); the group of these isomorphisms is written $Isom(A,A')$ (resp. $Isom^W(A,A')$, $Isom^W_\R(A,A')$).
 As the filters in $\A$ defined in \ref{suse:Faces} above (\eg local faces, sectors, walls,..) are permuted by $Aut(\A)$, they are well defined in any apartment of type $\A$ and exchanged by any isomorphism.

\begin{defi*}
\label{de:AffineHovel}
An ordered affine hovel of type $\mathbb A$ is a set $\SHI$ endowed with a covering $\mathcal A$ of subsets  called apartments such that:
\begin{enumerate}
\item[{\bf (MA1)}] any $A\in \mathcal A$ admits a structure of an apartment of type $\mathbb A$;

\item[{\bf (MA2)}] if $F$ is a point, a germ of a preordered interval, a generic  ray or a solid chimney in an apartment $A$ and if $A'$ is another apartment containing $F$, then $A\cap A'$ contains the
enclosure $cl_A(F)$ of $F$ and there exists a Weyl-isomorphism from $A$ onto $A'$ fixing $cl_A(F)$;

\item[{\bf (MA3)}] if $\mathfrak R$ is  the germ of a splayed chimney and if $F$ is a face or a germ of a solid chimney, then there exists an apartment that contains $\mathfrak R$ and $F$;

\item[{\bf (MA4)}] if two apartments $A,A'$ contain $\mathfrak R$ and $F$ as in {\bf (MA3)}, then their intersection contains $cl_A(\mathfrak R\cup F)$ and there exists a Weyl-isomorphism from $A$ onto $A'$ fixing $cl_A(\mathfrak R\cup F)$;

\item[{\bf (MAO)}] if $x,y$ are two points contained in two apartments $A$ and $A'$, and if $x\leq_A y$ then the two line segments $[x,y]_A$ and $[x,y]_{A'}$ are equal.
\end{enumerate}
\end{defi*}

  \par We ask here $\SHI$ to be thick of {\bf finite thickness}: the number of local chambers  containing a given (local) panel has to be finite $\geq 3$.
     This number is the same for any panel in a given wall $M$ \cite[2.9]{R11}; we denote it by $1+q_M$.

  \par An automorphism (resp. a Weyl-automorphism, a vectorially-Weyl automorphism) of $\SHI$ is a bijection $\qf:\SHI\to\SHI$ such that $A\in\sha\iff \qf(A)\in\sha$ and then $\qf\vert_A:A\to\qf(A)$ is an isomorphism (resp. a Weyl-isomorphism, a vectorially-Weyl isomorphism).
\medskip
\parni{\bf 2)} For $x\in\SHI$, the set $\sht^+_x\SHI$ (resp. $\sht^-_x\SHI$) of segment germs $[x,y)$ for $y>x$ (resp. $y<x$) may be considered as a building, the positive (resp. negative) tangent building. The corresponding faces are the local faces of positive (resp. negative) direction and vertex $x$. The associated Weyl group is $W^v$.
 If the $W-$distance (calculated in $\sht^\pm_x\SHI$) of two local chambers is $d^W(C_x,C'_x)=w\in W^v$, to any reduced decomposition $w=r_{i_1}\cdots r_{i_n}$ corresponds a unique minimal gallery from $C_x$ to $C'_x$ of type $(i_1,\cdots,i_n)$. We shall say, abusively, that this gallery is of type $w$.

 \par The buildings $\sht^+_x\SHI$ and $\sht^-_x\SHI$ are actually twinned. The codistance $d^{*W}(C_x,D_x)$ of two opposite sign chambers $C_x$ and $D_x$ is the $W-$distance $d^W(C_x, op D_x)$, where $op D_x$ denotes the opposite chamber to $D_x$ in an apartment containing $C_x$ and $D_x$.

 \begin{lemm*} \cite[2.9]{R11} Let $D$ be an half-apartment in $\SHI$ and $M=\partial D$ its wall (\ie its boundary).
 One considers a panel $F$ in $M$ and a local chamber $C$ in $\SHI$ covering $F$.
 Then there is an apartment containing $D$ and $C$.
 \end{lemm*}

\parni{\bf 3)}  We assume that $\SHI$ has a strongly transitive group of automorphisms $G$, \ie all isomorphisms involved in the above axioms are induced by elements of $G$, \cf \cite[4.10]{R13} and \cite{CiR15}.
  We choose in $\SHI$ a fundamental apartment which we identify with $\A$.
   As $G$ is strongly transitive, the apartments of $\SHI$ are the sets $g.\A$ for $g\in G$. The stabilizer $N$ of $\A$ in $G$ induces a group $W=\qn(N)\subset Aut(\A)$ of affine automorphisms of $\A$ which permutes the walls, local faces, sectors, sector-faces... and contains the affine Weyl group $W^a=W^v\ltimes Q^\vee$ \cite[4.13.1]{R13}.

\par   We denote the  stabilizer of $0\in\A$ in $G$ by $K$ and the pointwise stabilizer (or fixer) of $C_0^+$ by $K_I$; this group $K_I$ is called the {\it Iwahori subgroup}.

\medskip
\parni{\bf 4)}  We ask  $W=\qn(N)$ to be {\bf positive} and {\bf vectorially-Weyl} for its action on the vectorial faces.
      This means that the associated linear map $\vect w$ of any $w\in\qn(N)$ is in $W^v$.
      As $\qn(N)$ contains $W^a$ and stabilizes $\shm$, we have $W=\qn(N)=W^v\ltimes Y$, where $W^v$ fixes the origin $0$ of $\A$ and $Y$ is a group of translations such that:
  \quad    $Q^\vee\subset Y\subset P^\vee=\{v\in V\mid\qa(v)\in\Z,\forall\qa\in\QF\}$.
 An element $\mathbf{w}\in W$ will often be written $\mathbf{w}=\ql.w$, with $\ql\in Y$ and $w\in W^v$.


  \par We ask $Y$ to be {\bf discrete} in $V$. This is clearly satisfied if $\QF$ generates $V^*$ \ie $(\qa_i)_{i\in I}$ is a basis of $V^*$.

\medskip
\parni{\bf 5)} Note that there is only a finite number of constants $q_M$ as in the definition of thickness. Indeed, we must have $q_{wM}=q_M$, $\forall w\in\qn(N)$ and $w.M(\qa,k)=M(w(\qa),k),\forall w\in W^v$. So now, fix $i\in I$, as $\qa_i(\qa_i^\vee)=2$ the translation by $\qa_i^\vee$ permutes the walls $M=M(\qa_i,k)$ (for  $k\in \Z$) with two orbits.
  So, $Q^\vee\subset W^a$ has at most two orbits in the set of the constants $q_{M(\qa_i,k)}$ : one containing the $q_i=q_{M(\qa_i,0)}$ and the other containing the $q_i'=q_{M(\qa_i,\pm{}1)}$.
  Hence, the number of (possibly) different $q_M$ is at most $2.\vert I\vert$. We denote this set of parameters by $\shq=\{q_i,q'_i\mid i\in I\}$.

\par If $\qa_i(\qa_j^\vee)$ is odd for some $j\in I$, the translation by $\qa_j^\vee$ exchanges the two walls $M(\qa_i,0)$ and $M(\qa_i,-\qa_i(\qa_j^\vee))$; so $q_i=q'_i$.
More generally, we see that $q_i=q'_i$ when $\qa_i(Y)=\Z$, \ie $\qa_i(Y)$ contains an odd integer.
If $\qa_i(\qa_j^\vee)=\qa_j(\qa_i^\vee)=-1$, one knows that the element $r_ir_jr_i$ of $W^v(\{i,j\})$ exchanges $\qa_i$ and $-\qa_j$, so $q_i=q'_i=q_j=q'_j$.

\par Actually many of the following results (in sections 2, 3) are true without assuming the existence of $G$: we have only to assume that the parameters $q_M$ satisfy the above conditions.

 \medskip
  \par\noindent{\bf 6) Examples.} The main examples of all the above situation are provided by the hovels of almost split Kac-Moody groups over fields complete for a discrete valuation and with a finite residue field, see \ref{7.2} below.

 \cache{
\par These Kac-Moody examples are type-preserving. We get a non trivial group $D$ of diagram automorphisms when considering a loop group $\widetilde G$ associated to a non simply connected semi simple group over a local field; then $G$ is a Kac-Moody group.
} 

\medskip
  \par\noindent{\bf 7) Remarks.}
  a) In the following, we sometimes use results of \cite{GR08} even though, in this paper we deal with split Kac-Moody groups and residue fields containing $\C$. But the cited results are easily generalizable to our present framework, using the above references.

  \par b)
  All isomorphisms in \cite{R11} are Weyl-isomorphisms, and, when $G$ is strongly transitive, all isomorphisms constructed in \lc are induced by an element of $G$.


  \subsection{Type $0$ vertices}\label{1.4}

  The elements of $Y$, through the identification $Y=N.0$, are called {\it vertices of type $0$} in $\A$; they are special vertices. We note $Y^+=Y\cap\sht$ and $Y^{++}=Y\cap \overline{C^v_f}$.
   The type $0$ vertices in $\SHI$ are the points on the orbit $\SHI_0$ of $0$ by $G$. This set $\SHI_0$ is often called the affine Grassmannian as it is equal to $G/K$, where $K =$ Stab$_G(\{0\})$. But in general, $G$ is not equal to $KYK=KNK$ \cite[6.10]{GR08} \ie $\SHI_0\not=K.Y$.

   \par We know that $\SHI$ is endowed with a $G-$invariant preorder $\leq $ which induces the known one on $\A$ and satisfies $x\leq y\implies\exists A\in\sha$ with $x,y\in A$ \cite[5.9]{R11}.
   We set $\SHI^+=\{x\in\SHI\mid0\leq x\}$ , $\SHI^+_0=\SHI_0\cap\SHI^+$ and $G^+=\{g\in G\mid0\leq g.0\}$; so $\SHI^+_0=G^+.0=G^+/K$.
   As $\leq $ is a $G-$invariant preorder, $G^+$ is a semigroup.

 \par  If $x\in\SHI^+_0$ there is an apartment $A$ containing $0$ and $x$ (by definition of $\leq$) and all apartments containing $0$ are conjugated to $\A$ by $K$ (axiom (MA2)); so $x\in K.Y^+$ as $\SHI^+_0\cap\A=Y^+$.
    But $\qn(N\cap K)=W^v$ and $Y^+=W^v.Y^{++}$, with uniqueness of the element in $Y^{++}$. 
     So $\SHI^+_0=K.Y^{++}$, more precisely $\SHI^+_0=G^+/K$ is the union of the $KyK/K$ for $y\in Y^{++}$.
 This union is disjoint, for the above construction does not depend on the choice of $A$ (\cf \ref{1.11}.a).

    \par Hence, we have proved that the map $Y^{++}\to K\backslash G^+/K$ is one-to-one and onto.

\subsection{Vectorial distance and $Q^\vee-$order}\label{1.5}

    For $x$ in the Tits cone $\sht$, we denote by $x^{++}$ the unique element in $\overline{C^v_f}$ conjugated by $W^v$ to $x$.

    \par  Let $\SHI\times_\leq \SHI=\{(x,y)\in\SHI\times\SHI\mid x\leq y\}$ be the set of increasing pairs in $\SHI$.
    Such a pair $(x,y)$ is always in a same apartment $g.\A$; so $(g^{-1}).y-(g^{-1}).x\in\sht$ and we define the {\it vectorial distance} $d^v(x,y)\in  \overline{C^v_f}$ by $d^v(x,y)=((g^{-1}).y-(g^{-1}).x)^{++}$.
    It does not depend on the choices we made (by \ref{1.11}.a below).

    \par For $(x,y)\in \SHI_0\times_\leq \SHI_0=\{(x,y)\in\SHI_0\times\SHI_0\mid x\leq y\}$, the vectorial distance $d^v(x,y)$ takes values in $Y^{++}$.
     Actually, as $\SHI_0=G.0$, $K$ is the  stabilizer of $0$ and $\SHI^+_0=K.Y^{++}$ (with uniqueness of the element in $Y^{++}$), the map $d^v$ induces a bijection between the set $\SHI_0\times_\leq \SHI_0/G$ of $G-$orbits in $\SHI_0\times_\leq \SHI_0$ and $Y^{++}$.

     \par Further, $d^v$ gives the inverse of the  map $Y^{++}\to K\backslash G^+/K$, as any $g\in G^+$ is in $K.d^v(0,g.0).K$.

     \par For $x,y\in\A$, we say that $x\leq _{Q^\vee}\,y$ (resp. $x\leq _{Q^\vee_\R}\,y$) when $y-x\in Q^\vee_+$ (resp. $y-x\in Q^\vee_{\R+}=\sum_{i\in I}\,\R_{\geq 0}.\qa_i^\vee$).
     We get thus a preorder which is an order at least when $(\qa_i^\vee)_{i\in I}$ is free or $\R_+-$free (\ie $\sum a_i\qa_i^\vee=0,a_i\geq 0\Rightarrow a_i=0,\forall i$).

\subsection{Paths}
\label{suse:Paths}

We consider piecewise linear continuous paths
$\pi:[0,1]\rightarrow \mathbb A$ such that each (existing) tangent vector $\pi'(t)$
belongs to an orbit $W^v.\lambda$ for some $\lambda\in {\overline{C^v_f}}$. Such a path is called a {\it $\lambda-$path}; it is
increasing with respect to the preorder relation $\leq$ on $\mathbb A$.

 For any $t\neq 0$ (resp. $t \neq1$), we let
$\pi'_-(t)$ (resp. $\pi'_+(t)$) denote the derivative of $\pi$ at $t$ from the left
(resp. from the right). Further, we define $w_\pm(t)\in W^v$ to be the smallest
element  in its
$(W^v)_\lambda-$class such that $\pi'_\pm(t)=w_\pm(t).\lambda$ (where $(W^v)_\lambda$ is the  stabilizer in
$W^v$ of $\lambda$).

\par Hecke paths of shape $\ql$ (with respect to the sector germ $\g S_{-\infty}=germ_\infty(-C^v_f)$) are  $\ql-$paths satisfying some further precise conditions, see \cite[3.27]{KM08} or \cite[1.8]{GR13}. For us their interest will appear just below in \ref{1.10}.

But to give a formula for the structure constants of the forthcoming Iwahori-Hecke algebra, we will need slight different Hecke paths whose definition is detailed in Section \ref{sc3}.






  \subsection{Retractions onto $Y^+$}\label{1.10} 

For all $x\in \SHI^+$ there is an apartment containing $x$ and $C_0^-=germ_0(-C^v_f)$ \cite[5.1]{R11} and this apartment is conjugated to $\A$ by an element of $K$ fixing $C_0^-$ (axiom (MA2) ).
    So, by the usual arguments and [\lc, 5.5], see below \ref{1.12}.a), we can define the retraction $\qr_{C_0^-}$ of $\SHI^+$ into $\A$ with center $C_0^-$; its image is $\qr_{C_0^-}(\SHI^+)=\sht=\SHI^+\cap\A$ and $\qr_{C_0^-}(\SHI^+_0)=Y^+$.

    \par Using axioms (MA3), (MA4), \cf  \cite[4.4]{GR08}, we may also define the retraction $\qr_{-\infty}$ of $\SHI$ onto $\A$ with center the sector-germ $\g S_{-\infty}$.

    \par More generally, we may define the retraction $\qr$ of $\SHI$ (resp. of the subset $\SHI_{\geq z}=\{y\in\SHI\mid y\geq z\}$, for a fixed $z$) onto an apartment $A$ with center any sector germ (resp. any local chamber of negative direction with vertex $z$).
 For any such retraction $\qr$, the image of any segment $[x,y]$ with $(x,y)\in\SHI\times_\leq \SHI$ and $d^v(x,y)=\ql\in\overline{C^v_f}$ (resp. and moreover $x,y\in\SHI_{\geq
 z}$) is a $\ql-$path 
    \cite[4.4]{GR08}. In particular, $\qr(x)\leq \qr(y)$.

     \par Actually, the image by $\qr_{-\infty}$ of any segment $[x,y]$ with $(x,y)\in\SHI\times_\leq \SHI$ and $d^v(x,y)=\ql\in Y^{++}$ is a Hecke path of shape $\ql$ with respect to $\g S_{-\infty}$ \cite[th. 6.2]{GR08}, and we have the following.



 \begin{lemm*}  a) For $\ql\in Y^{++}$ and $w\in W^v$, $w.\ql\in\ql-Q^\vee_+$, \ie $w.\ql\leq _{Q^\vee}\,\ql$.

 \par b) Let $\qp$ be a Hecke path of shape $\ql\in Y^{++}$ with respect to $\g S_{-\infty}$, from $y_0\in Y$ to $y_1\in Y$.
  Then, for $0\leq t<t'<1$,
$$
\begin{array}{l}\ql=\qp'_+(t)^{++}=\qp'_-(t')^{++};\\
\qp'_+(t)\leq _{Q^\vee}\qp'_-(t')\leq _{Q^\vee}\qp'_+(t')\leq _{Q^\vee}\qp'_-(1);\\
\qp'_+(0)\leq _{Q^\vee}\,\ql;\\ 
\qp'_+(0)\leq _{Q^\vee_\R}\,(y_1-y_0)\leq _{Q^\vee_\R}\,\qp'_-(1)\leq _{Q^\vee}\,\ql;\\
y_1-y_0 \leq _{Q^\vee}\,\ql.
\end{array}
$$

  \par Moreover $y_1-y_0$ is in the convex hull $conv(W^v.\ql)$ of all $w.\ql$ for $w\in W^v$, more precisely in the convex hull $conv(W^v.\ql,\geq\qp'_+(0))$ of all $w'.\ql$ for $w'\in W^v$, $w'\leq w$, where $w$ is the element with minimal length such that $\qp'_+(0)=w.\ql$.

\par c) If moreover $(\qa_i^\vee)_{i\in I}$ is free, we may replace above $\leq _{Q^\vee_\R}$ by $\leq _{Q^\vee}$.


\par d) If  $x\leq z\leq y$ in $\SHI_0$, then $d^v(x,y)\leq _{Q^\vee}d^v(x,z)+d^v(z,y)$.

 \end{lemm*}


\begin{NB}  In the following\cache{(except perhaps in section \ref{s6})}, we always assume $(\qa_i^\vee)_{i\in I}$ free.
\end{NB}

\begin{proof} Everything is proved in \cite[2.4]{GR13}, except the second paragraph of b).
 Actually we see in \lc that $y_1-y_0$ is the integral of the locally constant vector-valued function $\qp'_+(t)=w_+(t).\ql$, where $w_+(t)$ is decreasing for the Bruhat order \cite[5.4]{GR13}, hence the result.
\end{proof}

\subsection{Chambers of type $0$}\label{1.11}

Let $\mathscr C_0^{\pm}$ be the set of all local chambers with vertices of type $0$ and positive or negative direction.
An element of vertex $x\in\SHI_0$  in this set (resp. its direction) will often be written  $C_x$ (resp. $C_x^v$).
We consider $\mathscr C_0^+\times_\leq \mathscr C_0^+=\{(C_x,C_y)\in \mathscr C_0^+\times \mathscr C_0^+\mid x\leq y\}$.
 We sometimes write $C_x\leq C_y$ when $x\leq y$.

\begin{prop*} \cite[5.4 and 5.1]{R11} Let $x,y\in\SHI$ with $x\leq y$.
We consider two local faces $F_x,F_y$ with respective vertices $x,y$.

\par a) $\{x,y\}$ is included in an apartment and two such apartments $A,A'$ are isomorphic by a Weyl-isomorphism in $G$, fixing $cl_A(\{x,y\})=cl_{A'}(\{x,y\})\supset[x,y]$.

\par b) There is an apartment containing $F_x$ and $F_y$, if we assume moreover $x\stackrel{o}{<} y$ or $x=y$ when $F_x$ and $F_y$ are respectively of positive and negative direction.
\end{prop*}

\parni{\bf Consequences.} 1) We define $W^+ = W^v\ltimes Y^+$ which is a subsemigroup of $W$.

\par If $C_x\in\SHC_0^+$, we know by b) above, that there is an apartment $A$ containing $C_0^+$ and $C_x$.
But all apartments containing $C_0^+$ are conjugated to $\A$ by $K_I$ (Axiom (MA2)), so there is $k\in K_I$ with $k^{-1}.C_x\subset\A$.
 Now the vertex $k^{-1}.x$ of $k^{-1}.C_x$ satisfies $k^{-1}.x\geq0$, so there is $\mathbf{w}\in W^+$ such that $k^{-1}.C_x=\mathbf{w}.C_0^+$.

\par When $g\in G^+$, $g.C_0^+$ is in $\SHC_0^+$ and there are $k\in K_I$, $\mathbf{w}\in W^+$ with $g.C_0^+=k.\mathbf{w}.C_0^+$, \ie $g\in K_I.W^+.K_I$.
We have proved the {\it Bruhat decomposition} $G^+=K_I.W^+.K_I$.
\medskip
\par 2) Let $x\in \SHI_0$, $C_y\in\SHC_0^+$ with $x\leq y$, $x\not=y$.
 We consider an apartment $A$ containing $x$ and $C_y$ (by b) above) and write $C_y=F(y,C^v_y)$ in $A$.
 For $y'\in y+C^v_y$ sufficiently near to $y$, $\qa(y'-x)\not=0$ for any root $\qa$.
 So $]x,y')$ is in a unique local chamber $pr_x(C_y)$ of vertex $x$; this chamber satisfies $[x,y)\subset \overline{pr_x(C_y)}\subset cl_A(\{x,y'\})$ and does not depend of the choice of $y'$.
 Moreover, if $A'$ is another apartment containing $x$ and $C_y$, we may suppose $y'\in A\cap A'$ and $]x,y')$, $cl_A(\{x,y'\})$, $pr_x(C_y)$ are the same in $A'$.
 The local chamber $pr_x(C_y)$ is well determined by $x$ and $C_y$, it is the projection of $C_y$ in
$\sht^+_x\SHI$.

 \par The same things may be done changing accordingly $+$ to $-$ and $\leq$ to $\geq$.
 But, in the above situation, if $C_x\in\SHC_0^+$, we have to assume $x\stackrel{o}{<} y$ to define the analogous $pr_y(C_x)\in\SHC_0^+$.

\begin{prop}\label{1.12} In the situation of Proposition \ref{1.11},

\par a) If $x\stackrel{o}{<} y$ or $F_x$ and $F_y$ are respectively of negative and positive direction,  any two apartments $A,A'$ containing $F_x$ and $F_y$ are isomorphic by a Weyl-isomorphism in $G$ fixing the convex hull of $F_x$ and $F_y$ (in $A$ or $A'$).

\par b) If $x=y$ and the directions of $F_x,F_y$ have the same sign, any two apartments $A,A'$ containing $F_x$ and $F_y$ are isomorphic by a Weyl-isomorphism in $G$, $\qf:A\to A'$, fixing $F_x$ and $F_y$.
 If moreover $F_x$ is a local chamber, any minimal gallery from $F_x$ to $F_y$ is fixed by $\qf$ (and in $A\cap A'$).

 \par c) If $F_x$ and $F_y$ are of positive directions and $F_y$ is spherical, any two apartments $A,A'$ containing $F_x$ and $F_y$ are isomorphic by a Weyl-isomorphism  in $G$ fixing $F_x$ and $F_y$.

\end{prop}

\begin{rema*} The conclusion in c) above is less precise than in a) or in \ref{1.11}.a.
We may actually improve it when the hovel is associated to a very good family of parahorics, as defined in \cite{R13} and already used in \cite{GR08}.
Then, using the notion of half good fixers, we may assume that the isomorphism in c) above fixes some kind of enclosure of $F_x$ and $F_y$ (containing the convex hull).
This particular case includes the case of an almost split Kac-Moody group over a local field.
\end{rema*}

\begin{proof} The assertion a) (resp. b)) is Proposition 5.5 (resp. 5.2) of \cite{R11}.
To prove c) we improve a little the proof of 5.5 in \lc and use the classical trick that says that it is enough to assume that, either $F_x$ or $F_y$ is a local chamber.
We assume now that $F_x=C_x$ is a local chamber; the other case is analogous.

\par We consider an element $\QO_x$ (resp. $\QO_y$) of the filter $C_x$ (resp. $F_y$) contained in $A\cap A'$.
We have $x\in\overline{\QO_x}$, $y\in\overline{\QO_y}$ and one may suppose $\QO_x$ open and  $\QO_y$ open in the support of $F_y$.
There is an isomorphism $\qf:A\to A'$ fixing $\QO_x$.
Let $y'\in\QO_y$, we want to prove that $\qf(y')=y'$.
As $F_y$ is spherical, $x\leq y\stackrel{o}{<}y'$, hence $x\stackrel{o}{<}y'$.
So $x'\leq y'$ for any $x'\in \QO_x$ ($\QO_x$ sufficiently small).
Moreover $[x',y']\cap\QO_x$ is an open neighbourhood of $x'$ in $[x',y']$. By the following lemma, we get $\qf(y')=y'$.
\end{proof}

\begin{lemm*} Let us consider two apartments $A,A'$ in $\SHI$, a subset $\QO\subset A\cap A'$, a point $z\in A\cap A'$ and an isomorphism $\qf:A\to A'$ fixing (pointwise) $\QO$. We assume that there is $z'\in\QO$ with $z'\leq z$ or $z'\geq z$ and $[z',z]\cap\QO$ open in $[z',z]$; then $\qf(z)=z$.
\end{lemm*}

\begin{NB} This lemma tells, in particular, that any isomorphism $\qf:A\to A'$ fixing a local facet $F\subset A\cap A'$ fixes $\overline F$.
\end{NB}

\begin{proof} $\qf\vert_{[z',z]}$ is an affine bijection of $[z',z]$ onto its image in $A'$, which is the identity in a neighbourhood of $z'$.
But \ref{1.11}.a) tells that $[z',z]\subset A\cap A'$ and the identity of $[z',z]$ is an affine bijection (for the affine structures induced by $A$ and $A'$).
Hence $\qf\vert_{[z',z]}$ coincides with this affine bijection; in particular $\qf(z)=z$.
\end{proof}

\subsection{$W-$distance}\label{1.13}

Let $(C_x,C_y)\in\mathscr C_0^+\times_\leq \mathscr C_0^+$, there is an apartment $A$ containing $C_x$ and $C_y$.
We identify $(\A,C_0^+)$ with $(A,C_x)$ \ie we consider the unique $f\in Isom^W_\R(\A,A)$ such that $f(C_0^+)=C_x$.
 Then $f^{-1}(y)\geq 0$ and there is $\mathbf{w}\in W^+$ such that  $f^{-1}(C_y)=\mathbf{w}.C_0^+$.
 By \ref{1.12}.c, $\mathbf{w}$ does not depend on the choice of $A$.

\par We define the {\it $W-$distance} between the two local chambers $C_x$ and $C_y$ to be this unique element: $d^W(C_x,C_y) = \mathbf{w}\in W^+ = Y^+\rtimes W^v$.
 If $\mathbf{w}=\ql.w$, with $\ql\in Y^+$ and $w\in W^v$, we write also $d^W(C_x,y)=\ql$.
 As $\leq$ is $G-$invariant, the $W-$distance is also $G-$invariant.
 When $x=y$, this definition coincides with the one in \ref{1.3}.2.

 \par If $C_x,C_y,C_z\in\SHC_0^+$, with $x\leq y\leq z$, are in a same apartment, we have the Chasles relation: $d^W(C_x,C_z)=d^W(C_x,C_y).d^W(C_y,C_z)$.

\par When $C_x=C_0^+$ and $C_y=g.C_0^+$ (with $g\in G^+$), $d^W(C_x,C_y)$ is the only $\mathbf{w}\in W^+$ such that $g\in K_I.\mathbf{w}.K_I$.
We have thus proved the uniqueness in Bruhat decomposition: $G^+=\coprod_{\mathbf{w}\in W^+}\,K_I.\mathbf{w}.K_I$.

The $W-$distance classifies  the orbits of $K_I$ on $\{C_y\in\SHC_0^+\mid y\geq 0\}$, hence also the orbits of $G$ on $\mathscr C_0^+\times_\leq \mathscr C_0^+$.


\section{Iwahori-Hecke Algebras}\label{s2}

Throughout this Section, we assume that $(\alpha_i^\vee)_{i\in I}$ is free and we consider any ring $R$. 
 To each $\mathbf w\in W^+$, we associate a function $T_{\mathbf w}$ from  $\mathscr C_0^+\times_\leq \mathscr C_0^+$ to $R$ defined by
$$
T_{\mathbf w}(C,C') =
\left\{
\begin{array}{l}
1 \quad\hbox{ if } d^W(C,C') = \mathbf w,\\
0 \quad\hbox{ otherwise.}
\end{array}
\right.
$$
Now we  consider the following free $R-$module
$$^I\shh_R^\SHI = \{\qf = \sum_{\mathbf w\in W^+} a_{\mathbf w}T_{\mathbf w}\mid a_{\mathbf w}\in R,\ a_{\mathbf w} = 0 \hbox{ except for a finite number of } \mathbf w\},
$$ We endow this $R-$module with the convolution product:

$$
(\qf*\psi)(C_x,C_y)=\sum_{C_z}\,\qf(C_x,C_z)\psi(C_z,C_y).
$$ where $C_z\in \mathscr C^+_0$ is such that $x\leq z \leq y$. It is clear that this product is associative and $R-$bilinear.  We prove below 
that this product is well defined.

\par As in \cite[2.1]{GR13}, we see easily that $^I\shh_R^\SHI $ can be identified with the natural convolution algebra of the functions $G^+\to R$, bi-invariant under $K_I$ and with finite support.



\begin{lemm}\label{LeFinite}
Let $\g S^-\subset A$ be a sector-germ with negative direction in an apartment $A$, $\rho_- : \SHI\to A$ the corresponding retraction, and $\mathbf w\in W^+$. Then the set
$$
P = \{d^W(\rho_- (C_x), \rho_-(C_y))\in W^+\mid \forall (C_x,C_y)\in \mathscr C_0^+\times_\leq \mathscr C_0^+,\ d^W(C_x,C_y) = \mathbf w\}
$$
is finite and included in a finite subset $P'$ of $W^+$ depending only  on $\mathbf w$ and on the position of $C_x$ with respect to $\g S^-$ (i.e. on the codistance $w_x\in W^v$ from $C_x$ to the local chamber $C_x^-$ in $x$ of direction $\g S^-$).

Let us write $\mathbf w = \lambda.w$ for $\lambda \in Y^+$ and $w\in W^v$. If we assume $C_x$ and $\g S^-$ opposite (i.e. $w_x = 1$), then any $\mathbf v = \mu.v\in P'$ satisfies $\lambda\leq_{Q^\vee}\mu\leq_{Q^\vee} \lambda^{++}$ and $\qm$ is in $conv(W^v.\ql^{++})$.
More precisely $\qm$ is in the convex hull $conv(W^v.\ql^{++},\geq\ql)$ of all $w'.\ql^{++}$ for $w'\in W^v$, $w'\leq w_\ql$, where $w_\ql$ is the element with minimal length such that $\ql=w_\ql .\ql^{++}$.

\par If moreover $\lambda\in Y^{++}$, then $\mu = \lambda$ and $v\leq w$. In particular, for $\mathbf w = \lambda\in Y^{++}$, $P = \{\mathbf w\} = \{\lambda\}$.
\end{lemm}

\begin{proof}
We consider an apartment $A_1$ containing $C_x$ and $C_y$. We set $C'_y = C_x +(y-x)$ in $A_1$. Identifying $(\mathbb A, C_0^+)$ with $(A_1,C_x)$ (resp. $(A_1, C'_y)$), we have $y = x+ \lambda$ (resp. $C_y = wC'_y$).

We have to prove that the possibilities for $\rho_-(C_y)$ vary in a finite set determined by $\rho_-(C_x)$, $\mathbf w$, and $w_x$. We shall prove this by successively showing the same kind of result for $\rho_-([x,y))$, $\rho_-(y)$ and $\rho_- (C'_y)$.
 Up to isomorphism, one may suppose $C_x\subset A$.

a) Fixing a reduced decomposition for $w_\lambda$ gives a minimal gallery between $C_x$ and $[x,y)$. By retraction, we get a gallery with the same type from $\rho_-(C_x)$ to $\rho_- ([x,y))$. The possible foldings of this gallery determine the possibilities for $\rho_- ([x,y))$. More precisely, $\rho_- ([x,y)) = x + w'(\lambda^{++}_A)[0,1)$ for $w'\leq w_\lambda$ and $\lambda^{++}_A$ the image in $A$ of $\lambda^{++}$ by the identification of $(\mathbb A, C_0^+)$ with $(A,C_x)$.

b) We fix now $\rho_- ([x,y))$. By Lemma \ref{1.10} b), $\rho_- ([x,y])$ is a Hecke path $\pi$ of shape $\lambda^{++}$ (with respect to $\g S^-$). Its derivative $\pi'_+(0)$ is well determined by $\rho_- ([x,y))$. We identify $A$ with $\mathbb A$ in such a way that $\g S^-$ has direction $-C_f^v$. Then $\lambda_A^{++} = w_x(\lambda^{++})$ and $\pi'_+(0) = w'w_x(\lambda^{++})$, with $w'$ as above. By Lemma \ref{1.10} b), $\pi'_+(0)\leq_{Q^\vee} \rho_-(y) - \rho_-(x) \leq_{Q^\vee} \lambda^{++}$. So there is a finite number of possibilities for $\rho_-(y)$.

c) Now we fix $\rho_- ([x,y))$, $\qr_-(y)$ and investigate the possibilities for $\rho_- (C'_y)$. Let $\xi\in Y^{++}$ and in the interior of the fundamental chamber $C_f^v$. In the apartment $A_1$, with $(A_1, C_x)$ identified with $(\mathbb A, C_0^+)$, we consider $x' = x + \xi$ and $y' = y+\xi$ (hence $y' = x' + \lambda$).

As in a) and b) above, we get that there is a finite number of possibilities for $\rho_-(x')$.

c1) On one side, we may also enlarge the segment $[x,x']$ with $[x',x'')$, where $x'' = x' + \qe\xi$. On the other side, $[x,x']$ can be described as a path $\pi_1:[0,1]\to A_1$, $\pi_1(t) = x + t\xi$. The retracted path $\pi = \rho_-(\pi_1)$ satisfies $\rho_-(x') - \rho_-(x)\leq_{Q^\vee} \pi'_+(1)\leq_{Q^\vee} \lambda^{++}$, again by Lemma \ref{1.10}.
So there is a finite number of possibilities for $\pi'_+(1)$, i.e. for $[x',x'')$. But there exists (in $A_1$) a minimal gallery of the type of a reduced decomposition of $w_\lambda$ from the unique local chamber $(C_x + \xi)$ containing $[x',x'')$ to $[x',y')$. Hence, there exists a gallery of the same type between (a local chamber containing) $\rho_-([x',x''))$ and $\rho_-([x',y'))$. Therefore, there is a finite number of possibilities for $\rho_-([x',y'))$.

As in b) above, we deduce that there is a finite number of possibilities for $\rho_-(y')$.

c2) The path $\rho_-([y,y'])$ is a Hecke path of shape $\xi$ from $\rho_-(y)$ to $\rho_-(y')$. By \cite{GR08} Corollary 5.9, there exists a finite number of such paths. In particular, there is a finite number of possibilities for the segment-germ $\rho_-([y,y'))$ and for $\rho_-(C'_y)$.

d) Next, we fix $\rho_-(C'_y)$. Fixing a reduced decomposition for $w$ gives a minimal gallery between $C'_y$ and $C_y$, hence a gallery of the same type between $\rho_-(C'_y)$ and $\rho_-(C_y)$. So, the number of possible $\rho_-(C_y)$ is finite and $d^W(\rho_-(C'_y),\rho_-(C_y))\leq w$.

e) Finally, let us consider the case $w_x = 1$, then $\lambda^{++}_A = \lambda^{++}$. So, in b), we get $\pi'_+(0) = w'(\lambda^{++})$ with $w'\leq w_\lambda$, hence $\pi'_+(0)\geq_{Q^\vee} w_\lambda(\lambda^{++}) = \lambda$ and $\lambda\leq_{Q^\vee} \pi'_+(0)\leq_{Q^\vee} \rho_-(y) - \rho_-(x) = \mu\leq_{Q^\vee} \lambda^{++}$.
If moreover $\lambda$ is in $Y^{++}$, then $\lambda = \lambda^{++}$ and $\mu = \lambda$. The Hecke path of shape $\lambda$ $\rho_-([x,y])$ is the segment $[\rho_-(x),\rho_-(x) + \lambda]$. Its dual dimension is $0$ (see \cite{GR08} 5.7). By \cite{GR08} 6.3, there is one and only one segment in $\SHI$ with end $y$ that retracts onto this Hecke path: any apartment containing $y$ and $\g S^-$ contains $[x,y]$.
But $C_x$ is the enclosure of $x$ and $C'_y=C_y$ (computation in $A_1$). So, any apartment containing $\g S^-$ and $C_y'$ contains $C_x$. Therefore, we have $\lambda = d^W(C_x,C'_y) = d^W(\rho_-(C_x),\rho_-(C'_y))$.

The end of the proof of the lemma follows then from d) above.
\end{proof}

\begin{prop}\label{PrFinite1}

Let $C_x, C_y, C_z\in \mathscr C_0^+$ be such that $x\leq z \leq y$ and $d^W(C_x,C_z) = \mathbf w\in W^+$, $d^W(C_z,C_y) = \mathbf v\in W^+$. Then $d^W(C_x,C_y)$ varies in a finite subset $P_{\mathbf w,\mathbf v}$ of $W^+$, depending only on $\mathbf w$ and $\mathbf v$.

Let us write $\mathbf w = \lambda.w$ and $\mathbf v = \mu.v$ for $\lambda,\mu \in Y^+$ and $w,v\in W^v$. If we assume $\lambda = \lambda^{++}$ and $w=1$, then any $\mathbf w' = \nu.u\in P_{\mathbf w,\mathbf v}$ satisfies $\lambda + \mu \leq_{Q^\vee} \nu\leq_{Q^\vee} \lambda + \mu^{++}$ and $\qn-\ql\in conv(W^v.\qm^{++},\geq\qm)\subset conv(W^v.\qm^{++})$.

\par  If, moreover, $\mu = \mu^{++}\in Y^{++}$, then $\nu = \lambda + \mu$ and $u\leq v$. In particular, for $\mathbf w = \lambda$, $\mathbf w' = \mu$ in $Y^{++}$, $P_{\mathbf w,\mathbf v}=\{\lambda + \mu\}$.

\end{prop}

\begin{proof}
Now we consider any apartment $A$ containing $C_x$, the sector-germ $\g S^-$ opposite $C_x$ and the retraction $\rho_-$ as in Lemma \ref{LeFinite}. Then $\rho_-(C_x) = C_x$ and $d^W(C_x,\rho_-(C_z))$ varies in a finite subset $P_x$ of $W^+$ depending on $\mathbf w$, by Lemma \ref{LeFinite}.
If $d^W(C_x,\rho_-(C_z)) = \lambda'.w'$, then the relative position $w_z\in W^v$ of $C_z$ and $\g S^-$ is equal to $w'$. Applying once more Lemma \ref{LeFinite} to $C_z$ and $C_y$, we get that $d^W(\rho_-(C_z),\rho_-(C_y))$ varies in a finite subset $P_{w'}$ of $W^+$ depending only on $\mathbf v$ and $w'$.
Finally, $d^W(C_x,\rho_-(C_y))$ varies in the finite subset $P_{\mathbf w,\mathbf v} = \{\mathbf w'.\mathbf v'\in W^+\mid \mathbf w' = \lambda'.w'\in P_x,\ \mathbf v'\in P_{w'}\}$. Taking now $A$ containing $C_x$ and $C_y$, we get $d^W(C_x,C_y) = d^W(C_x,\rho_-(C_y))\in P_{\mathbf w,\mathbf v}$.

To finish, suppose that $\lambda = \lambda^{++}$ and $w=1$. By Lemma \ref{LeFinite}, $P_1 = \{ \lambda\}$, hence $w' = w_z = 1$. Applying again Lemma \ref{LeFinite}, any $\mathbf v' = \mu'.v'\in P_{w'}$ satisfies $\mu\leq_{Q^\vee} \mu'\leq_{Q^\vee} \mu^{++}$. So any $\mathbf w'' = \nu.u$ in $P_{\mathbf w,\mathbf v}$ is equal to $(\lambda + \mu').v'$ for $\mu'.v'\in P_{w'} = P_1$, hence  $\lambda + \mu \leq_{Q^\vee} \nu = \lambda + \mu'\leq_{Q^\vee} \lambda + \mu^{++}$. If moreover $\mu\in Y^{++}$, then $\nu = \lambda + \mu$ and $u\leq v$. The last particular case is now clear.

\end{proof}

\begin{prop}\label{PrFinite2}

Let us fix two local chambers $C_x$ and $C_y$ in $\mathscr C_0^+$ with $x\leq y$ and $d^W(C_x,C_y) = \mathbf u\in W^+$. We consider $\mathbf w$ and $\mathbf v$ in $W^+$. Then the number $a_{\mathbf w,\mathbf v}^{\mathbf u}$ of $C_z\in \mathscr C_0^+$ with $x\leq z\leq y$, $d^W(C_x,C_z) = \mathbf w$ and $d^W(C_z,C_y) = \mathbf v$ is finite (i.e. in $\mathbb N$).

If we assume $\mathbf w = \lambda$, $\mathbf v = \mu$ and $\mathbf u = \nu$, then $a_{\mathbf w,\mathbf v}^{\mathbf u} = a_{\lambda,\mu}^\nu \geq 1$ (resp. $=1$) when $\ql\in Y^{++}$, $\qm\in Y^{+}$ (resp. $\ql,\qm\in Y^{++}$) and $\qn=\ql+\qm$.
\end{prop}

\begin{proof}
 We have $d^v(x,z) = \lambda^{++}$ and $d^v(z,y) = \mu^{++}$. So, by \cite{GR13}, 2.5, the number of possible $z$ is finite. Hence, we fix $z$ and count the possible $C_z$.

Let $C'_z$ be the local chamber in $z$ containing $[z,y)$ and $[z,y')$ for $y'$ in a sufficiently small element of the filter $C_y$. By convexity, $C'_z$ is well determined by $z$ and $C_y$. But in an apartment containing $C_y, C_z$ (hence also $C'_z$), we see that $d^W(C'_z, C_z)$ is well determined by $\mathbf v$. So there is a gallery (of a fixed type) from $C'_z$ to $C_z$, thus the number of possible $C_z$ is finite.

Assume now that $\mathbf w= \lambda\in Y^{++}$, $\mathbf v = \mu\in Y^{+}$ and $\mathbf u = \lambda + \mu$. Taking an apartment $A_1$ containing $C_x$ and $C_y$, it is clear that the local chamber $C_z$ in $A_1$ such that $d^W(C_x,C_z) = \lambda$ satisfies also $d^W(C_z, C_y) = \mu$ (as $d^W(C_x, C_y) = \lambda + \mu$). So $a_{\lambda,\mu}^{\lambda + \mu} \geqslant 1$. We consider now any $C_z$ satisfying the conditions, with moreover $\mu\in Y^{++}$.

As in Proposition \ref{PrFinite1}, we choose $A$ containing $C_x$ and $\g S^-$ opposite $C_x$. We saw in Lemma \ref{LeFinite} e) that any apartment containing $C_z$ and $\g S^-$ contains $C_x$ and $d^W(C_x, \rho_-(C_z)) = \lambda$. With the same Lemma applied to $C_z$ and $C_y$, we see that any apartment  containing $C_z$ and $\g S^-$ contains $C_y$. In particular, there is an apartment $A_1$ containing $C_x, C_z, C_y$ and $d^W(C_x, C_z) = \lambda$, $d^W(C_z, C_y) = \mu$, $d^W(C_x, C_y) = \lambda + \mu$. But $\lambda,\mu\in Y^{++}$, so $C_z$ is in the enclosure of $C_x$ and $C_y$.
Therefore, $C_z$ is unique: any other apartment $A_2$ containing $C_x$ and $C_y$ contains $x,y$ (with $x\leq y$) and $x'=x+\qx$, $y'=y+\qx$ (with $x'\leq y'$), for $\qx\in C^v_x=C^v_y$ small; by \ref{1.11}.a, $A_2$ contains $z\in cl_{A_1}(\{x,y\})$, $z'=z+\qx\in cl_{A_1}(\{x',y'\})$, hence also $C_z\subset cl_{A_1}(]z,z'))$.

\end{proof}

\begin{theo}\label{ThAlgebra}

For any ring $R$, $^I\mathcal H_R^\SHI$ is an algebra with identity element $Id = T_1$ such that
$$
T_{\mathbf w} * T_{\mathbf v} = \sum_{\mathbf u\in P_{\mathbf w,\mathbf v}} a_{\mathbf w, \mathbf v}^{\mathbf u} T_{\mathbf u}
$$ and $T_\lambda * T_\mu = T_{\lambda + \mu}$, for $\ql,\qm\in Y^{++}$.

\end{theo}

\begin{proof}
It derives from Propositions \ref{PrFinite1} and \ref{PrFinite2}, as the function $T_{\mathbf w} * T_{\mathbf v}:\SHC_0^+\times_{\leq}\SHC_0^+\to R$ is clearly $G-$invariant.
\end{proof}

\begin{defi}\label{DeAlgebra}
The algebra $^I\mathcal H_R^\SHI$ is the Iwahori-Hecke algebra associated to $\SHI$ with coefficients in $R$.
\end{defi}

\par The structure constants $a_{\mathbf w, \mathbf v}^{\mathbf u}$ are non-negative integers.
We conjecture that they are  polynomials in the parameters $q_i,q_i'$ with coefficients in $\Z$ and that these polynomials depend only on $\A$ and $W$.
 We prove this in the following section for  $\mathbf w, \mathbf v$ generic, see the precise hypothesis just below.
 We get also this conjecture for some $\A,W$ when all $q_i,q_i'$ are equal; in the general case we get only that they are Laurent polynomials, see \ref{5.7}.

Geometrically, it is possible to get more informations about $T_\ql*T_\qm$ when $\ql\in Y^{++},\qm\in Y^+$, but we shall obtain them algebraically (Corollary \ref{4.3}).

\section{Structure constants}\label{sc}

In this section, we compute the structure constants $a_{\mathbf w, \mathbf v}^{\mathbf u}$ of the Iwahori-Hecke algebra $^I\mathcal H_R^\SHI$, assuming that $\mathbf v =\mu . v$ is regular and $\mathbf w = \lambda . w$ is spherical, \ie $\mu$ is regular and $\lambda$ spherical (see 1.1 for the definitions). We will adapt some results obtained in the spherical case in \cite{GR13} to our situation.

These structure constants depend on the shape of the standard apartment $\A$ and on the numbers $q_M$ of \ref{1.3}.
 Recall that the number of (possibly) different parameters is at most $2.\vert I\vert$. We denote by $\shq=\{q_1,\cdots,q_l,q'_1=q_{l+1},\cdots,q'_l=q_{2l}\}$ this set of parameters.

\subsection{Centrifugally folded galleries of chambers}
\label{sc1}

Let $z$ be a point in the standard apartment $\mathbb A$. We have twinned buildings $\mathcal T_z^+\SHI$ (resp. $\mathcal T_z^-\SHI$).
 We consider their unrestricted structure, so the associated Weyl group is $W^v$ and the chambers (resp. closed chambers) are the local chambers $C=germ_z(z+C^v)$ (resp. local closed chambers $\overline{C}=germ_z(z+\overline{C^v})$), where $C^v$ is a vectorial chamber, \cf \cite[4.5]{GR08} or \cite[{\S{}} 5]{R11}.
  The distances (resp. codistances) between these chambers are written $d^W$ (resp. $d^{*W}$).
  To $\A$ is associated a twin system of apartments $\A_z = (\A_z^-,\A_z^+)$.

  \par  We choose in $\A^-_z$ a negative (local) chamber $C^-_z$ and denote by $C^+_z$ its opposite in $\A^+_z$.
  We consider the system of positive roots $\QF^+$ associated to $C^+_z$. Actually, $\QF^+=w.\QF^+_f$, if $\QF^+_f$ is the system $\QF^+$ defined in \ref{1.1} and $C^+_z=germ_z(z+w.C^v_f)$.
   We denote by $(\qa_i)_{i\in I}$ the corresponding basis of $\QF$ and by $(r_i)_{i\in I}$ the corresponding generators of $W^v$. Note that this change of notation is limited to Section 3.

\par Fix a reduced decomposition of  an element $w\in W^v$, $w = r_{i_1}\cdots r_{i_r}$ and let
$\mathbf i = (i_1,..., i_r)$ be the type of the decomposition.
 We consider now galleries of (local) chambers $\mathbf c = (C_z^-,C_1,...,C_r)$ in the
apartment $\mathbb A_z^-$ starting at $C_z^-$ and of type $\mathbf i$.

The set of all these galleries is in bijection with the set $\Gamma (\mathbf i) = \{1,r_{i_1}\}\times\cdots\times \{1,r_{i_r}\}$ via the map $(c_1,...,c_r)\mapsto (C_z^-, c_1 C_z^-,...,c_1\cdots c_r C_z^-)$.
Let $\beta_j = -c_1\cdots c_j (\alpha_{i_j})$, then $\beta_j$ is the root corresponding to the common
limit hyperplane $M_j = M({\beta_j},-\qb_j(z))$ of type $i_j$ of $C_{j-1} =c_1\cdots c_{j-1} C_z^- $ and $C_j = c_1\cdots c_j C_z^-$ and satisfying $\qb_j(C_j)\geq{}\qb_j(z)$.

\begin{defi*}
Let $\mathfrak Q$ be a chamber in $\mathbb A_z^+$.
A gallery $\mathbf c = (C_z^-,C_1,...,C_r)\in\Gamma(\mathbf i)$ is said to be centrifugally folded with respect to $\mathfrak Q$ if $C_j = C_{j-1}$ implies that  $M_j $ is a wall and separates $\mathfrak Q$ from $C_j = C_{j-1}$. We denote this set of centrifugally folded galleries by $\Gamma_{\mathfrak Q}^+(\mathbf i)$.

\end{defi*}

\subsection{Liftings of galleries}
\label{sc2}

Next, let $\qr_{\mathfrak Q} : \SHI_z \to \mathbb A_z$ be the retraction centered at $\mathfrak Q$. To a gallery of chambers $\mathbf c = (C_z^-,C_1,...,C_r)$ in $\Gamma(\mathbf i)$,
one can associate the set of all galleries of type $\mathbf i$ starting at $C_z^-$ in $\SHI_z^-$ that retract onto $\mathbf c$, we denote this set by $\mathcal C_{\mathfrak Q}(\mathbf c)$.
We denote the set of minimal galleries (i.e. $C_{j-1}\ne C_j$) in $\mathcal C_{\mathfrak Q}(\mathbf c)$ by $\mathcal C^m_{\mathfrak Q}(\mathbf c)$. Recall from \cite{GR13}, Proposition 4.4, that the set $\mathcal C^m_{\mathfrak Q}(\mathbf c)$ is nonempty if, and only if, the gallery $\mathbf c$ is centrifugally folded with respect to $\mathfrak Q$. Recall also from loc. cit., Corollary 4.5, that if $\mathbf c \in \Gamma_{\mathfrak Q}^+(\mathbf i)$, then the number of elements in $\mathcal C^m_{\mathfrak Q}(\mathbf c)$ is:

$$
\sharp \mathcal C^m_{\mathfrak Q}(\mathbf c) = \prod_{j\in J_1} (q_{j} - 1) \times \prod_{j\in J_2} q_{j}
$$
 where $q_j=q_{M_j}\in\shq$, $J_1=\{j\in\{1,\cdots,r\}\mid c_j=1\}$ and $J_2=\{j\in\{1,\cdots,r\}\mid c_j =r_{i_j} \mathrm{\ and\ } M_j \mathrm{\ is\ a\ wall\ separating\ } {\mathfrak Q} \mathrm{\ from\ } C_j\}$.

\cache{
$$
\sharp \mathcal C^m_{\mathfrak Q}(\mathbf c) = \prod_{k=1}^{t(\mathbf c)} q_{j_k} \times \prod_{l=1}^{r(\mathbf c)} (q_{j_l} - 1)
$$
 where $q_j=q_{z,\qb_j}=q_{z,\qa_{i_j}}\in\shq$ and $t(\mathbf c)$ and
$r(\mathbf c) $ are easy statistics on the gallery (see \cite{GR13} Corollary 4.5).
}

\subsection{Liftings of Hecke paths}
\label{sc3}

The Hecke paths we consider here are slight modifications of those used in \cite{GR13}. Let us fix a local positive chamber $C_x\in \mathscr C_0^+\cap \mathbb A$.
Namely, a Hecke path of shape $\mu^{++}$ with respect to $C_x$ in $\A$ is a $\mu^{++}-$path in $\mathbb A$ that we denote by $\pi = [z'=z_0,z_1,...,z_{\ell_\pi},y]$ and that satisfies the following assumptions.
 For all $z=\pi(t)$, $z\not= z_0=\pi(0)$, we ask $x \stackrel{o}{<} z$ and then we choose the local negative chamber $C^-_z$ as $C^-_z=pr_z(C_x)$ such that $\overline{C^-_z}$ contains $[z,x)$ and $[z,x')$ for $x'$ in a sufficiently small element of the filter $C_x$.
 Then we assume moreover
 that for all $k\in\{1,...,\ell_\pi\}$, there exists a $(W^v_{z_k}, C^-_{z_k})-$chain from $\pi'_-(t_k)$ to $\pi'_+(t_k)$, where $z_k=\pi(t_k)$.
  \cache{ where $C^-_{z_k}$ is the local chamber in $z_k = \pi(t_k)$ containing $[z_k,x)$ and $[z_k,x')$ for $x'$ in a sufficiently small element of the filter $C_x$.}
 More precisely, this means that, for all $k\in\{1,...,\ell_\pi\}$, there exist finite sequences
$(\xi_0=\pi_-'(t),\xi_1,\dots,\xi_s=\pi_+'(t))$ of vectors in $V$ and
$(\beta_1,\dots,\beta_s)$  of  real roots such that, for all $j=1,\dots,s$:
\begin{itemize}
\item [i)] $r_{\beta_j}(\xi_{j-1})=\xi_j$,
\item[ii)] $\beta_j(\xi_{j-1})<0$,
\item[iii)] $r_{\beta_j}\in{W^v_{\pi(t_k)}}$ {\it i.e.} $\beta_j(\pi(t_k))\in\mathbb Z$,
\item[iv)] each $\qb_j$ is positive  with respect to $C_x$ \ie $\qb_j(z_k - C_x)>0$.
\end{itemize}

The centrifugally folded galleries are related to the lifting of Hecke paths by the following lemma that we proved in \cite{GR13} Lemma 4.6.

Suppose $z\in \A, x \stackrel{o}{<} z$. 
Let $\xi$ and $\eta$ be two segment germs in $\A_z^+$. Let $-\eta$ and $-\xi$ opposite respectively $\eta$ and $\xi$ in $\A_z^-$.
 Let $\mathbf i$ be the type of a minimal gallery between $C_z^-$ and $C_{-\xi}$, where $C_{-\xi}$ is the negative (local) chamber containing $-\xi$ such that $d^W(C_z^-, C_{-\xi})$ is of minimal length.
  Let $\mathfrak Q$ be a chamber of $\A_z^+$ containing $\eta$. We suppose $\qx$ and $\eta$ conjugated by $W^v_z$.

\begin{lemm*} The following conditions are equivalent:
\par (i) There exists an opposite $\zeta$ to $\eta$ in $\SHI_z^-$ such that $\rho_{\A_z, C^-_z} ( \zeta) = -\xi$.

\par (ii)  There exists a gallery $\mathbf c \in \Gamma_{\mathfrak Q}^+(\mathbf i)$ ending in $-\eta$.

\par (iii) There exists a $(W^v_z, C_z^-)-$chain from $\xi$ to $\eta$.

\par  Moreover the possible $\qz$ are in one-to-one correspondence with the disjoint union of the sets  $\mathcal C^m_{\mathfrak Q}(\mathbf c)$ for $\mathbf c$ in the set $\Gamma_{\mathfrak Q}^+(\mathbf i,-\eta)$ of galleries in $\Gamma_{\mathfrak Q}^+(\mathbf i)$ ending in $-\eta$.
\end{lemm*}

 \par For an Hecke path as above and for $k\in\{1,...,\ell_\pi\}$, we define the segment germs $\eta_k = \pi_+(t_k)=\pi(t_k)+\pi'_+(t_k).[0,1)$ and $-\qx_k=\qp_-(t_k)=\pi(t_k)-\pi'_-(t_k).[0,1)$.
 As above  $\mathbf i_k$ is the type of a minimal gallery between $C_{z_k}^-$ and $C_{-\xi_k}$, where $C_{-\xi_k}$ is the negative (local) chamber such that  $-\xi_k\subset \overline{C_{-\xi_k}}$ and $d^W(C_{z_k}^-, C_{-\xi_k})$ is of minimal length.
 Let $\mathfrak Q_k$ be a fixed chamber  in $\A^+_{z_k}$ containing $\eta_k$
and $\Gamma_{\mathfrak Q_k}^+(\mathbf i_k, -\eta_k)$ be the set of all the galleries $(C^-_{z_k}, C_1,...,C_r )$  of type $\mathbf i_k$ in $\A^-_{z_k}$, centrifugally folded with respect to $\mathfrak Q_k$ and with $-\eta_k\in \overline{C_r}$.

\cache{Recall also that, for all $k\in\{1,...,\ell_\pi\}$, we define $w_\pm(t_k)\in W^v$ to be the smallest
element  in its $(W^v)_\lambda-$class such that $\pi'_\pm(t_k)=w_\pm(t_k).\mu^{++}$.
Let $\mathbf i_k$ be a reduced decomposition of $w_-(t_k)$ and let $\mathfrak Q_k$ be a fixed chamber  \modif{in $\A^+_{z_k}$ containing $\eta_k = \pi_+(t_k)=\pi(t_k)+\pi'_+(t_k).[0,1)$.
We set $-\qx_k=\qp_-(t_k)=\pi(t_k)-\pi'_-(t_k).[0,1)$.}
 Let $\Gamma_{\mathfrak Q_k}^+(\mathbf i_k, -\eta_k)$ be the set of all the galleries $(C^-_{z_k}, C_1,...,C_r = C_{\ell(w_-(t_k))})$ \modif{  in $\A^-_{z_k}$,} centrifugally folded with respect to $\mathfrak Q_k$ and with $-\eta_k\in \overline{C_r}$.} 

\medskip
Let us denote the retraction $\rho_{\mathbb A, C_x} : \SHI_{\geq x} \to \mathbb A$ simply by $\rho$  and recall that $y=\pi(1)$.
 Let $S_{C_x}(\pi, y)$ be the set of all segments $[z,y]$ such that $\rho ([z,y]) = \pi$, in particular, $\rho(z) = z'$. The following two theorems are proved in the same way as Theorem 4.8 and Theorem 4.12 of \cite{GR13}, in particular, we lift the path $\pi$ step by step starting from the end of $\pi$.

\begin{theo}\label{sc4} The set $S_{C_x}(\pi, y)$ is non empty if, and only if,  $\pi$ is a Hecke path with respect to $C_x$. Then, we have a bijection
$$
S_{C_x}(\pi, y)\simeq \prod_{k=1}^{\ell_\pi} \coprod_{\mathbf c\in\Gamma_{\mathfrak Q_k}^+(\mathbf i_k,-\eta_k)} \mathcal C^m_{\mathfrak Q_k} (\mathbf c).
$$

\par In particular, the number of elements in this set is a polynomial  in the numbers $q\in\shq$ with coefficients in $\Z$ depending only on $\A$.
\end{theo}

\begin{theo}\label{sc5} Let $\ql,\qm,\qn\in Y^{++}$ with $\ql$ spherical. Then, the number $m_{\ql,\qm}(\qn)$ of triangles $[0,z,\qn]$ in $\SHI$ with $d^v(0,z)=\ql$ and $d^v(z,\qn)=\qm$ is equal to:

\begin{equation}
\label{eq:SC}
m_{\ql,\qm}(\qn)=\sum_{w\in W^v/(W^v)_\lambda}\sum_\qp\prod_{k=1}^{\ell_\qp}\quad\sum _{\mathbf c\in\Gamma_{\mathfrak Q_k}^+(\mathbf i_k,-\eta_k)} \sharp \mathcal C^m_{\mathfrak Q_k} (\mathbf c)
\end{equation}
\par\noindent where $\qp$ runs over the set of Hecke paths of shape $\qm$ with respect to $C_x$ from $w.\ql$ to $\qn$ and $\ell_\qp$, $\Gamma_{\mathfrak Q_k}^+(\mathbf i_k,-\eta_k)$ and $\mathcal C^m_{\mathfrak Q_k} (\mathbf c)$ are defined as above for each such  $\qp$.

\end{theo}

\begin{rema*} In theorems \ref{sc4}, \ref{sc5} above and in \cite{GR13}, it is interesting to precise that, if $t_{\ell_\pi}=1$, \ie $z_{\ell_\pi}=y$, then, in the above formulas, $-\eta_{\ell_\pi}$ and $\mathfrak Q_{\ell_\pi}$
are not well defined: $\pi_+(1)$ does not exist.
 We have to understand that $ \coprod_{\mathbf c\in\Gamma_{\mathfrak Q_{\ell_\pi}}^+(\mathbf i_{\ell_\pi},-\eta_{\ell_\pi})} \mathcal C^m_{\mathfrak Q_{\ell_\pi}} (\mathbf c)$ is the set of all minimal galleries of type $\mathbf i_{\ell_\pi}$ starting from $C_y^-$ (whose cardinality is $\prod_{j=1}^r
\,q_{i_j}$, if $\mathbf i_{\ell_\pi}=(i_1,\cdots,i_r)$).
\end{rema*}

\subsection{The formula}\label{sc6}

Let us fix two local chambers $C_x$ and $C_y$ in $\mathscr C_0^+$ with $x\leq y$ and $d^W(C_x,C_y) = \mathbf u\in W^+$. We consider $\mathbf w$ and $\mathbf v$ in $W^+$. Then we know that the number $a_{\mathbf w,\mathbf v}^{\mathbf u}$ of $C_z\in \mathscr C_0^+$ with $x\leq z\leq y$, $d^W(C_x,C_z) = \mathbf w$ and $d^W(C_z,C_y) = \mathbf v$ is finite, see Proposition \ref{PrFinite2}. In order to obtain a formula for that number, we first use equivalent conditions on the $W-$distance between the chambers.

\begin{lemm*}
1) Assume $\ql$ spherical. Let $C^-_z=pr_z(C_x)$ and let $w^+_\lambda$ be the longest element such that $w_\lambda^+.\lambda\in \overline{C_f^v}$. Then
$$
d^W(C_x,C_z) = \lambda . w \Longleftrightarrow
\left \{
\begin{array}{l}
d^W(C_x,z) = \lambda \\
d^{*W}(C^-_z,C_z) = w_\lambda^+ w .
\end{array}
\right.
$$

2) Assume $\mu$ regular. Let $C^+_z=pr_z(C_y)$ and let $w_\mu$ be the unique element such that $\mu^{++} = w_\mu.\mu\in \overline{C_f^v}$. Then
$$
d^W(C_z,C_y) = \mu . v \Longleftrightarrow
\left \{
\begin{array}{l}
d^W(C_z,C^+_z) = w_\mu^{-1} \\
d^{W}(C^+_z,C_y) = \mu^{++}w_\mu v.
\end{array}
\right.
$$

As we assume $\mu$ regular, then $C'_y=pr_y(C_z)$ (resp.
$C^+_z=pr_z(C_y)$) is the unique local chamber in $y$ (resp. $z$) containing
$[y,z)$ (resp. $[z,y)$) and we have :

   $$d^W(C^+_z,C_y)=\mu^{++}w_\mu v \iff d^v(z,y)=\mu^{++} \mathrm{\ and\ }
d^{*W}(C'_y,C_y)=w_\mu v.$$

\end{lemm*}

\begin{proof}

1) By convexity, $C^-_z$ is in any apartment containing $C_x$ and $C_z$. Let us fix such an apartment $A$ and identify $(A,C_x)$ with $(\mathbb A, germ_0(C_f^v))$. By definition, we have $d^W(C_x,z) = d^W(C_x, z + C_x)$. Then, of course, $d^W(C_x,z) = \lambda$. Next as $\lambda$ is supposed spherical, the stabilizer $(W^v)_\lambda$ is finite, so $w^+_\lambda$ is well defined  and $x \stackrel{o}{<} z$, so $C_z^-$ is well defined.
Moreover, $d^W(op_A C^-_z, z+C_x) = w^+_\lambda$ and $d^W(z+C_x, C_z) = w$. Therefore, by Chasles, we get $d^W(op_A C^-_z, C_z) = w^+_\lambda w$, but, by definition, $d^{*W}(C^-_z,C_z)  = d^W(op_A C^-_z, z+C_z)$.

2) The first assertion is the Chasles' relation, as $C_z, C_y, C^+_z$ (and $C'_y$) are in a same apartment $A'$.
The second comes from the fact that, if $\mu$ is regular, then $d^W(C^+_z,C^+_{zy}) = d^v(z,y)\in Y^{++}$, where $C^+_{zy}$ opposites $C'_y$ at $y$  in $A'$.
 Moreover, $d^{*W}(C'_y,C_y) = d^W(C^+_{zy}, C_y)\in W^v $ by definition, so we conclude by Chasles.
\cache{
The second comes from the fact that if $\mu$ is regular, then $d^W(C^+_{zy}, C_y) = d^W(C^+_z, C_{yz})$, where $C^+_{zy}$ opposites $C'_y$ at $y$ and $C_{yz}$ is in relative position $v$ with $C_z$, in any apartment containing all these local chambers. Moreover, $d^{*W}(C'_y,C_y) = d^W(C^+_{zy}, C_y)$ by definition and $d^W(C^+_z, C_{yz}) = w_\mu v$, by Chasles.
}
\end{proof}

\begin{theo}\label{sc7}
Assume $\mu^{}$ is regular and $\lambda$ is spherical.
We choose the standard apartment $\A$ containing $C_x$ and $C_y$. Then
$$
a_{\mathbf w,\mathbf v}^{\mathbf u} = \sum_{\qp, t_{\ell_\pi = 1}}\Bigg [\Bigg (\prod_{k=1}^{\ell_\qp - 1}\ \sum _{\mathbf c\in\Gamma_{\mathfrak Q_k}^+(\mathbf i_k,-\eta_k)} \sharp \mathcal C^m_{\mathfrak Q_k} (\mathbf c)\Bigg ) \Bigg (\sum _{\mathbf d\in\Gamma_{C_y}^+(\mathbf i_\ell,\tilde C_y)} \sharp \mathcal C^m_{C_y} (\mathbf d) \Bigg ) \Bigg (\sum _{\mathbf e\in\Gamma_{C^-_{z_0}}^+(\mathbf i,C'_{z_0})} \sharp \mathcal C^m_{C^-_{z_0}} (\mathbf e) \Bigg ) \Bigg ] +
$$
$$
+ \sum_{\qp, t_{\ell_\pi < 1}}\Bigg [\Bigg (\prod_{k=1}^{\ell_\qp}\ \sum _{\mathbf c\in\Gamma_{\mathfrak Q_k}^+(\mathbf i_k,-\eta_k)} \sharp \mathcal C^m_{\mathfrak Q_k} (\mathbf c)\Bigg )  \Bigg (\sum _{\mathbf e\in\Gamma_{C^-_{z_0}}^+(\mathbf i,C'_{z_0})} \sharp \mathcal C^m_{C^-_{z_0}} (\mathbf e) \Bigg ) \Bigg ],
$$
where the $\pi$, in the first sum, runs over the set of all Hecke paths in $\A$ with respect to $C_x$ of shape $\mu^{++}$ from $x+\lambda=z_0$ to $x+ \nu=y$ such that $t_{\ell_\pi} = 1$, whereas in the second sum, the paths have to satisfy $t_{\ell_\pi} <1$  and $d^{*W}(C'_y,C_y) = w_\mu v$,  where $C^-_y=pr_y(C_x)$ is the local chamber in $y$  containing $[y,x)$ and $[y,x')$ for $x'$ in a sufficiently small element of the filter $C_x$.

\par Moreover $\mathbf i$ is a reduced decomposition of $w_\mu $, $C'_{z_0}$ is the local chamber at ${z_0}$  in $\A$ defined by $d^{*W}(C^-_{z_0},C'_{z_0}) = w_\lambda^+ w$, $\mathbf i_\ell$ is
the type of a minimal gallery from $C^-_y$ to the local chamber $C^*_y$ at $y$ in $\A$ containing the segment germ $\pi_-(y)=y-\pi'_-(1).[0,1)$ and $\tilde C_y$ is the unique local chamber at $y$ in $\mathbb A$ such that $d^{*W}(\tilde C_y,C_y) = w_\mu v$.
 The rest of the notation is as defined above.

\cache{$C'_{z_0}$ is the local chamber at ${z_0}$ defined by $d^{*W}(C^-_{z_0},C'_{z_0}) = w_\lambda^+ w$ in a fixed apartment $A'$ containing $C_x$ and $C^+_z$}
\end{theo}

\begin{proof}

Recall that, to compute the structure constants, we use the retraction $\rho = \rho_{\mathbb A, C_x} : \SHI \to \mathbb A$, where $C_x$ and $C_y$ are fixed and in $\mathbb A$.
We have $y=\qr (y)=x+\nu$ and the condition $d^W(C_x,z)=\ql$ is equivalent to $\qr (z)=x+\ql=z_0$.
We want to prove a formula of the form
$$
a_{\mathbf w,\mathbf v}^{\mathbf u} = \sum_{\pi}\Bigg ( \hbox{number of liftings of }\pi\Bigg )\times \Bigg (\hbox{number of } C_z\Bigg ),
$$ where $\pi$ runs over some set of  Hecke paths with respect to $C_x$ of shape $\mu^{++}$ from $x+\lambda$ to $x+ \nu$. It is possible to calculate like that for, in the case of a regular $\mu^{++}$, $\rho(C^+_z)$ is well determined by $\pi$. Hence, the number of $C_z$ only depends on $\pi$ and not on the lifting of $\pi$.

The local chambers $C_z$ satisfying $d^{*W}(C^-_z,C_z) = w_\lambda^+ w$ and $d^W(C_z,C^+_z) = w_\mu^{-1}$ are at the end of a minimal gallery starting at $C_z^+$ of type $\mathbf i$ and retracting by $\qr_{A',C^-_z}$ onto the local chamber $C'_z$ at $z$ defined by $d^{*W}(C^-_{z},C'_{z}) = w_\lambda^+ w$ in a fixed apartment $A'$ containing $C_x$ and $C^+_z$.
So their number is given by the number of minimal galleries starting at $C^+_z$ of type $\mathbf i$ and retracting on a centrifugally folded gallery $\mathbf e$ of type $\mathbf i$ ending in $C'_z$.
In other words, their number is given by the cardinality of the set $\mathcal C^m_{C^-_z}(\mathbf e)$, for each $\mathbf e\in\Gamma^+_{C^-_z}(\mathbf i,C'_z)$.
  Using an isomorphism fixing $C_x$ and sending $A'$ to $\A$, we may replace in this formula $z, C^-_z, C'_z$ and $C^+_z$ by ${z_0}, C^-_{z_0}, C'_{z_0}$ and the unique local chamber $C^+_{z_0}$ in $\A$ containing the segment germ $\pi_+(0)=z_0+\pi'_+(0).[0,1)$. Hence:
$$
\hbox{number of } C_z = \sum _{\mathbf e\in\Gamma_{C^-_{z_0}}^+(\mathbf i,C'_{z_0})} \sharp \mathcal C^m_{C^-_{z_0}} (\mathbf e).
$$

\medskip
Now, we compute the number of liftings of a Hecke path $\pi$ starting from the formula in Theorem \ref{sc5} and according to the two conditions $d^W(C_x,z) = \lambda $ and $d^{W}(C^+_z,C_y) = \mu^{++}w_\mu v$. The first one fixes one element in the set $W^v/(W^v)_\lambda$, namely the coset of $w^+_\lambda$, \ie $\pi(0)=x+\ql$.
The second one is equivalent to the fact that the segment $[z,y]$ is of type $\mu^{++}$ and $d^{*W}(C'_y,C_y) = w_\mu v$, as we have seen in the Lemma above.

Further, we have that $t_{\ell_\pi}<1 \Longleftrightarrow \pi_-(y)\in C^-_y$.
\cache{where $C^-_y$ is the local chamber in $y$  containing $[y,x)$ and $[y,x')$ for $x'$ in a sufficiently small element of the filter $C_x$.} 
 If $\pi_-(y)\in C^-_y$ then $\rho(C'_y) = C'_y = C^-_y$, whence, $d^{*W}(C^-_y,C_y) = w_\mu v$.
Since we lift the Hecke path into a segment backwards starting with its behaviour at $y = \pi(1)$, there is nothing more to count.

If $t_{\ell_\pi}=1$, then $\pi_-(y)\in  C^*_y=\rho(C'_y)\ne C_y^-$.
We want to lift the path but with the condition that $d^{*W}(C'_y,C_y) = w_\mu v$, which may be translated in $\qr'(C'_y)=\tilde C_y$, for $\qr'=\qr_{\A,C_y}$.
\cache{So, we let $\mathbf i_\ell$ be the type of a minimal gallery from $C^-_y$ to $C^*_y=\rho(C'_y)$ and $\tilde C_y$ be the unique local chamber in $\mathbb A$ such that $d^{*W}(\tilde C_y,C_y) = w_\mu v$. Moreover,}
 Since $\mu^{++}$ is regular, to find $[y,z)$  it is enough to  find $C'_y$ \ie to lift  $\tilde C_y$ with respect to $\rho'$.
  The liftings of $\tilde C_y$ are then given by the liftings of all the centrifugally folded galleries in $\A$ with respect to $C_y$ of type $\mathbf i_\ell$ from $C^-_y$ to {$\tilde C_y$} to minimal galleries. Therefore, their number is given by the cardinality of the set $\mathcal C^m_{C_y}(\mathbf d)$, for each $\mathbf d\in\Gamma^+_{C_y}(\mathbf i_\ell,\tilde C_y)$. The rest of the lifting procedure is the same as in the proof of Theorem 4.12 in \cite{GR13}.
\end{proof}

\subsection{Consequence}\label{sc8}

\par The above explicit formula, together with the formula for $\sharp \mathcal C^m_{\mathfrak Q}(\mathbf c)$ in \ref{sc2}, tell us that the structure constant $a_{\mathbf w,\mathbf v}^{\mathbf u} $ is a polynomial in the parameters $q_i,q'_i\in\shq$ with coefficients in $\Z$ and that this polynomial depends only on $\A$, $W$, $\mathbf w$, $\mathbf v$ and $\mathbf u$.
So we have proved the conjecture following Definition \ref{DeAlgebra} in this generic case: when $\ql$ is spherical and $\mu$ regular.


\section{Relations}\label{s3}

Here we study the Iwahori-Hecke algebra $^I\mathcal H_R^\SHI$ as a module over $\mathcal H_R(W^v)$ and we prove the first instance of the Bernstein-Lusztig relation.
For short, we write $^I\mathcal H_R=\,^I\mathcal H_R^\SHI$ and $T_i=T_{r_i}$ (when $i\in I$).

\begin{prop}\label{3.1} Let $\ql\in Y^+$, $w\in W^v$ and $i\in I$, then,

\par 1) $T_{\ql . w}*T_i=T_{\ql . wr_i}$ if, and only if, either $(w(\qa_i))(\ql)<0$ or $(w(\qa_i))(\ql)=0$ and $\ell(wr_i)>\ell(w)$. Otherwise $T_{\ql .w}*T_i=(q_i-1)T_{\ql .w}+q_iT_{\ql .wr_i}$.

\par 2) $T_i*T_{\ql .w}=T_{r_i(\ql) . r_iw}$ if, and only if, either $\qa_i(\ql)>0$ or $\qa_i(\ql)=0$ and $\ell(r_iw)>\ell(w)$. Otherwise $T_i*T_{\ql .w}=(q_i-1)T_{\ql .w}+q_iT_{r_i(\ql).r_iw}$.
\end{prop}

\begin{proof} We consider local chambers $C_x,C_z,C_y$ with $x\leq z\leq y$ and $d^W(C_x,C_z)=\ql .w$, $d^W(C_z,C_y)=r_i$.
So there is an apartment $A$ containing $C_x,C_z$ and, if we identify $(A,C_x)$ to $(\A,C_0^+)$, we have $C_z=(\ql .w)(C_x)$.
Moreover $y=z$, $C_z\not=C_y$ and $C_z,C_y$ share a panel $F_i$ of type $i$.
 We write $D$ the half apartment of $A$ containing $C_x$ and with wall $\partial D$ containing $F_i$.

 \par We first note that
 $$
 C_z\subset D \Longleftrightarrow \big ((w(\qa_i))(\ql) < 0\big ) \hbox{ or } \bigg ((w(\qa_i))(\ql) = 0 \hbox{ and } \ell(wr_i)>\ell(w)\bigg ).
 $$ Then, by Lemma \ref{1.3}.2, there exists an apartment $A'$ containing $C_y$ and $D$, hence also $C_x,C_z,C_y$.
 So $d^W(C_x,C_y)=\ql .wr_i$.
 The panel $F_i=F^\ell(z,F_i^v)\subset A$ is a spherical local face, so, for any $p\in z+F_i^v\subset A$, we have $z \stackrel{o}{<} p$, hence $x \stackrel{o}{<} p$.
 By \ref{1.12}.a, any apartment $A''$ containing $C_x$ and $F_i$ contains $C_z$; moreover $C_z$ is well determined by $F_i$ and $C_x$.
 The number $a_{\ql .w,r_i}^{\ql .wr_i}$ of \ref{PrFinite2} is equal to $1$ and we have proved that $T_{\ql .w}*T_i=T_{\ql .wr_i}$.

 \par If $C_z$ is not in $D$, we denote by $C'_z$ the local chamber in $D$ with panel $F_i$. By the above argument, $C'_z$ is well determined by $F_i$ and $C_x$, moreover $d^W(C_x,C'_z)=\ql .wr_i$.
 There are two cases for $C_y$: either $C_y=C'_z$ or not.
  If $C_y=C'_z$, then $d^W(C_x,C_y)=\ql .wr_i$ and, if $C_x,C_y$ are given, there are $q_i$ possibilities for $C_z$ (all local chambers covering $F_i$ and different from $C'_z$): $a_{\ql .w,r_i}^{\ql .wr_i}=q_i$.
  If $C_y\not=C'_z$, then $d^W(C_x,C_y)=\ql .w$ and, if $C_x,C_y$ are given, there are $q_i-1$ possibilities for $C_z$ (all local chambers covering $F_i$ and different from $C'_z,C_y$): $a_{\ql .w,r_i}^{\ql .w}=q_i-1$.

  \par We have proved 1) and we leave to the reader the similar proof of 2).
\end{proof}

\subsection{The subalgebra $\shh_R(W^v)$}\label{3.2}

\par We consider the $R-$submodule $\shh_R(W^v)$ of $^I\shh_R$ with basis $(T_w)_{w\in W^v}$. As $d^W(C_x,C_y)\in W^v$ if and only if $x=y$, it is clearly  a subalgebra of $^I\shh_R$.
 Actually $\shh_R(W^v)$ is the Iwahori-Hecke algebra of the tangent building $\sht^+_x\SHI$ for any $x\in\SHI$.

\par By Proposition \ref{3.1}, we have:

\parni - $T_{w}*T_i=T_{wr_i}$ if $\ell(wr_i)>\ell(w)$ and $T_{w}*T_i=(q_i-1)T_{w}+q_iT_{wr_i}$ otherwise.

\parni - $T_i*T_{w}=T_{r_iw}$ if $\ell(r_iw)>\ell(w)$ and $T_i*T_{w}=(q_i-1)T_{w}+q_iT_{r_iw}$ otherwise.

\par In particular $T_i^2=(q_i-1)T_i+q_iId$ and, for any reduced decomposition $w=r_{i_1}\cdots r_{i_n}$, $T_w=T_{i_1}\cdots T_{i_n}$.

\par Therefore, the algebra $\shh_R(W^v)$ is the well known Hecke algebra associated to the Coxeter system $(W^v,\{r_i\mid i\in I\})$ with (in general unequal) parameters $(q_i)_{i\in I}$ and coefficients in the ring $R$.
It is generated, as an $R-$algebra, by the $T_i$, for $i\in I$.

\par Suppose each $q_i$ invertible in $R$, then, as well known, $T_i^{-1}=q_i^{-1}(T_i-(q_i-1)Id)\in\shh_R(W^v)$ is the inverse of $T_i$.
 In particular any $T_w$ is invertible: for any reduced decomposition $w=r_{i_1}\cdots r_{i_n}$, $T_w^{-1}=T_{i_n}^{-1}\cdots T_{i_1}^{-1}$.

\begin{rema*} If $q_i$ is invertible, it is easy to see from Proposition \ref{3.1} that, either $T_{\ql .wr_i}=T_{\ql .w}*T_i$ or $T_{\ql .wr_i}=T_{\ql .w}*T_i^{-1}$ and, either $T_{r_i(\ql).r_iw}=T_i*T_{\ql .w}$ or $T_{r_i(\ql).r_iw}=T_i^{-1}*T_{\ql .w}$.
\end{rema*}

\begin{coro}\label{3.3} Suppose each $q_i$ invertible in $R$ and consider $\ql\in Y^+$.
 We may write $\ql=w.\ql^{++}$, with $w\in W^v$. Then $T_\ql=T_w*T_{\ql^{++}}*T_w^{-1}$.
\end{coro}

\begin{proof} We consider a reduced decomposition $w=r_{i_n}\cdots r_{i_1}$ and argue by induction on $n$.
So, for $w'=r_{i_{n-1}}\cdots r_{i_1}$ and $\ql'=w'.\ql^{++}$, we have $T_{\ql'}=T_{w'}*T_{\ql^{++}}*T_{w'}^{-1}$.
 We consider $T_w*T_{\ql^{++}}*T_w^{-1}=T_{i_n}*T_{\ql'}*T_{i_n}^{-1}$.
 But $\ell(r_{i_n}w')>\ell(w')$ and $\ql^{++}\in Y^{++}\subset \overline{C^v_f}$, so $\qa_{i_n}(w'.\ql^{++})\geq 0$, \ie $\qa_{i_n}(\ql')\geq 0$.
  We get $T_{i_n}*T_{\ql'}=T_{r_{i_n}(\ql').r_{i_n}}$ by \ref{3.1}.2 and then $T_{i_n}*T_{\ql'}*T_{i_n}^{-1}=T_{r_{i_n}(\ql')}=T_\ql$ by \ref{3.1}.1 (and the above remark).
\end{proof}

\begin{coro}\label{3.4} Let $\ql\in Y^+$ and $w,w'\in W^v$, then we may write
 $$T_{\ql .w'}*T_w=\sum_{w''\leq w}\, a_{\ql .w',w}^{\ql .w'w''}T_{\ql .w'w''}$$
 where each $a_{\ql .w',w}^{\ql .w'w''}$ is a polynomial in the $q_i$ with coefficients in $\Z$ and, when $w'=1$, $a_{\ql,w}^{\ql .w}>0$ is a primitive monomial.
 This polynomial $a_{\ql .w',w}^{\ql .w'w''}$ depends only on $\A$ and on $W$.
\end{coro}

\begin{proof} We write $w=r_{i_1}\cdots r_{i_n}$ and argue by induction on $n$. The result is then clear from Proposition \ref{3.1}.1. We get actually that $a_{\ql,w}^{\ql .w}$ is the product of some of the $q_{i_j}$ ($1\leq j\leq n$).
\end{proof}

\subsection{The Iwahori-Hecke algebra as a right $\shh_R(W^v)-$module}\label{3.5}

\par We assume here that each $q_i$ is invertible in $R$.

\par Given $\ql\in Y^+$, we see from Corollary \ref{3.4} that $\{T_\ql*T_w\mid w\in W^v\}$ and $\{T_{\ql. w}\mid w\in W^v\}$ are two bases of the same $R-$module.
The base-change matrix is triangular with respect to the Bruhat order on $W^v$ and the coefficients are Laurent polynomials in the $q_i$, with coefficients in $\Z$ (primitive Laurent monomials on the diagonal).
These polynomials depend only on $\A$ and $W$.

\par As $\{T_{\ql .w}\mid\ql\in Y^+, w\in W^v\}$ (resp. $\{T_w\mid w\in W^v\}$) is a $R-$basis of $^I\shh_R$ (resp. $\shh_R(W^v)$), this means in particular that  $^I\shh_R$ is a free right $\shh_R(W^v)-$module with basis $\{T_\ql\mid\ql\in Y^+\}$.

\par In particular the $R-$algebra $^I\shh_R$ is generated by the $T_i$ (for $i\in I$) and the $T_\ql$ (for $\ql\in Y^+$) and even by the $T_i$ (for $i\in I$) and the $T_\ql$ (for $\ql\in Y^{++}$), as we see from Corollary \ref{3.3}.

\begin{lemm}\label{3.6} Let $C_1,C_2\in\shc_0^+$, with vertices $x_1,x_2$ be such that $d^W(C_1,C_2)=\ql\in Y^{++}$.
 We consider $i\in I$, $F_1^i$ (resp. $F_2^i$) the panel of type $i$ of $C_1$ (resp. $C_2$).
 In an apartment $A_1$ (resp. $A_2$) containing $C_1$ (resp. $C_2$), we consider the sector panel $\g f_1^-$ (resp. $\g f_2^+$) with base point $x_1$ (resp. $x_2$) and direction opposite the direction of $F_1^i$ (resp. equal to the direction of $F_2^i$).

 \par Then there is an apartment $A$ containing $\g f_1^-$, $\g f_2^+$, $C_1,C_2$ and, in this apartment $A$, the directions of $\g f_1^-$ and $\g f_2^+$, $F_2^i$ and $\g f_1^-$ (resp. $F_1^i$ and $\g f_2^+$) are opposite (resp. equal).
\end{lemm}

\begin{proof} We choose $\ql_i\in F^v(\{i\})\cap Y\subset Y^{++}$.
We write $\g F_j^\pm$ the germ of $\g f_j^\pm$ and $F_j^{\pm v}$ its direction in $A_j$.
 In $A_1$ (resp. $A_2$) we consider the splayed chimney $\g r_1^-=\g r(C_1,F_1^{-v})$ (resp. $\g r_2^+=\g r(C_2,F_2^{+v})$) containing $\g f_1^-$ (resp. $\g f_2^+$) and, for $n\in\N$, the chamber of type $0$ $C_1(-n)=C_1-n\ql_i\subset\g r_1^-$ (resp. $C_2(+n)=C_2+n\ql_i\subset\g r_2^+$); actually we identify $(\A,C_0^+)$ with $(A_1,C_1)$ (resp. $(A_2,C_2)$) to consider $\ql_i$ in $\vect A_1$ (resp. $\vect A_2$).

 \par Then $d^W(C_1(-n),C_1)=d^W(C_2,C_2(+n))=n\ql_i\in Y^{++}$ and $d^W(C_1,C_2)=\ql\in Y^{++}$.
 By (MA3) there is an apartment $A$ containing the germs $\g R_1^-$ and $\g R_2^+$ of $\g r_1^-$ and $\g r_2^+$, hence $C_1(-n)$ and $C_2(+n)$ for $n$ great.
 By Proposition \ref{PrFinite1} and the last paragraph of the proof of \ref{PrFinite2}, $d^W(C_1(-n),C_2(+n))=\ql+2n\ql_i\in Y^{++}$ and $A$ contains $C_1,C_2$.
 By (MA4) $A$ contains also $\g f_1^-\subset\g r_1^-\subset cl_{A_1}(C_1,\g R_1^-)$ and $\g f_2^+\subset\g r_2^+\subset cl_{A_2}(C_2,\g R_2^+)$.
 So all assertions of the Lemma are satisfied.
\end{proof}

\begin{prop}\label{3.7} Let $C_1,C_2,C_3,C_4\in\shc_0^+$ be such that $d^W(C_1,C_2)=\ql\in Y^{++}$, $d^W(C_2,C_3)=r_i$ and $d^W(C_3,C_4)=\qm\in Y^{++}$.
 Then there is a direction of wall $M_i^\infty$ (\cf \cite[{\S{}} 4]{R11} or \cite[5.5]{GR13}), chosen accordingly to $C_1,C_2$ (but independently from $C_3,C_4$), such that $C_1,C_2,C_3,C_4$ are in the extended tree $\SHI(M_i^\infty)$.
\end{prop}

\begin{proof} We denote by $x_1,x_2=x_3,x_4$ the three vertices of $C_1,C_2,C_3,C_4$ and by $F_1^i,F_2^i=F_3^i,F_4^i$ their panels of type $i$.
We choose $\g f_1^-$ associated to $C_1$ and $F_1^i$ in an apartment $A_1$ (resp. $\g f_4^+$ associated to $C_4$ and $F_4^i$ in an apartment $A_4$), as in Lemma \ref{3.6}.
 By this Lemma, using $C_1$ and $C_2$, the direction of $\g f_1^-$ opposites that of $F_2^i=F_3^i$ in some apartment $A_2$ and, using $C_3$ and $C_4$, the direction of $\g f_4^+$ is the same as that of $F_2^i=F_3^i$ in some apartment $A_3$.
 In $A_3$ (resp. $A_2$) we consider the sector face $\g f_3^+$ (resp. $\g f_2^-$) with base point $x_2=x_3$ and same direction as $\g f_4^+$ or $F_2^i=F_3^i$  (resp. same direction as $\g f_1^-$ and opposite $F_2^i=F_3^i$).

 \par We may use the Lemma for $C_1,C_2,\g f_1^-,\g f_3^+$; so the directions of $\g f_1^-$ (or $\g f_2^-$) and $\g f_3^+$ (or $\g f_4^+$) are opposite and $C_1,C_2$ are in a same apartment $A_5$ of $\SHI(M_i^\infty)$, if we consider the direction of wall $M_i^\infty$ associated to the directions of $\g f_1^-$ and $\g f_4^+$.
 Using now the Lemma for $C_3,C_4,\g f_2^-,\g f_4^+$, we see that these filters are in a same apartment $A_6$ of $\SHI(M_i^\infty)$.
\end{proof}

\begin{theo}\label{3.8} Let $\ql,\qm\in Y^{++}$ and $i\in I$.
We write $N=\inf(\qa_i(\ql),\qa_i(\qm))\in\N$ and, for $n\in\N$, $q_i^{*n}=q_iq_i'q_iq_i'\cdots$, with $n$ terms in this product.
\medskip
\par a) If $N=\qa_i(\qm)\leq\qa_i(\ql)$, then $T_\ql*T_i*T_\qm=T_{\ql+\qm}*T_i$ for $N=0$ and, for $N>0$,
\medskip
\parni $T_\ql*T_i*T_\qm=q_i^{*N}T_{\ql+\qm-N\qa_i^\vee}*T_i+(q_i^{*N}-q_i^{*N-1})T_{\ql+\qm-(N-1)\qa_i^\vee}+ \cdots$
\par\qquad\qquad\qquad\qquad\qquad\qquad\qquad\qquad\qquad\qquad
$\cdots + (q_i^{*2}-q_i)T_{\ql+\qm-\qa_i^\vee}+(q_i-1)T_{\ql+\qm}$
\medskip
\par b) If $N=\qa_i(\ql)\leq\qa_i(\qm)$, then $T_\ql*T_i*T_\qm=T_i*T_{\ql+\qm}$ for $N=0$ and, for $N>0$,
\medskip
\parni $T_\ql*T_i*T_\qm=q_i^{*N}T_i*T_{\ql+\qm-N\qa_i^\vee}+(q_i^{*N}-q_i^{*N-1})T_{\ql+\qm-(N-1)\qa_i^\vee}+ \cdots$
\par\qquad\qquad\qquad\qquad\qquad\qquad\qquad\qquad\qquad\qquad
$\cdots + (q_i^{*2}-q_i)T_{\ql+\qm-\qa_i^\vee}+(q_i-1)T_{\ql+\qm}$
\end{theo}

\begin{remas*} 1) The case b)  is less interesting for us, as we try to express any element in the basis of \ref{3.5} for $^I\shh_R$ considered as a right $\shh_R(W^v)-$module.

\par 2) In the case a) we have $\qm-N\qa_i^\vee=r_i(\qm)$ and $\ql+\qm-N\qa_i^\vee\in Y^{++}$, as $\qa_i(\ql+\qm-N\qa_i^\vee)=\qa_i(\ql)-N$ and $\qa_j(\ql+\qm-N\qa_i^\vee)\geq\qa_j(\ql)+\qa_j(\qm)$ for $j\not=i$.
 So all $\qn$ such that $T_\qn$ appears on the right of the formula are in the $\qa_i^\vee-$chain between $\ql+\qm$ and $\ql+r_i(\qm)$; in particular they are all in $Y^{++}$.

 \par 3) We call relation a) or  relation b) the Bernstein-Lusztig relation for the $T_\ql$, (BLT) for short.
 We shall use it essentially when $\ql=\mu$.

 \par 4) When $\qa_i(\ql)$ or $\qa_i(\qm)$ is odd, we know that $q'_i=q_i$, \cf \ref{1.3}.5.
\end{remas*}

\begin{proof} We consider $C_1,C_2,C_3,C_4$ and $M_i^\infty$ as in Proposition \ref{3.7}.
When $N=0$ the results come from \ref{3.1}.
We concentrate on the case $0<N=\qa_i(\qm)\leq\qa_i(\ql)$; the other case is left to the reader.
 We have to evaluate $d^W(C_1,C_4)$ and, given $C_1,C_4$ satisfying $d^W(C_1,C_4)=\mathbf u$, to count the number of possible $C_2,C_3$.
 By Proposition \ref{3.7} everything is in the extended tree $\SHI(M_i^\infty)$, which is semi-homogeneous with thicknesses $1+q_i,1+q_i'$.
 By Proposition \ref{3.1}.2, $C_3$ is well determined by $C_2,C_4$ and lies in any apartment containing $C_2,C_4$; moreover $d^W(C_2,C_4)=r_i(\qm).r_i$.

 \par We consider an apartment $A_1$ (resp. $A_2$) of $\SHI(M_i^\infty)$ containing $C_1$ and $C_2$ (resp. $C_2$ and $C_4$, hence also $C_3$).
 We identify $(A_1,C_1)$ and $(A_2,C_2)$ with $(\A,C_0^+)$; we consider the retraction $\qr_1$ (resp. $\qr_2$) of $\SHI(M_i^\infty)$ onto $A_1$ (resp. $A_2$) with center $C_1$ (resp. $C_2$).
 The closed chambers in an apartment of $\SHI(M_i^\infty)$ are stripes limited by walls of direction $M_i^\infty$.
 In $A_1=\A$, these walls are $M(\qa_i,n)$, $n\in\Z$ and we write $S_1^k$ the stripe $S_1^k=\{x\mid k\leq\qa_i(x)\leq k+1\}$, in particular $C_1\subset S_1^0$ and $C_2\subset S_1^{\qa_i(\ql)}$.
 In $A_2=\A$, we get also stripes $S_2^k=\{x\mid k\leq\qa_i(x)\leq k+1\}$ such that $C_2\subset S_2^0=S_1^{\qa_i(\ql)}$, $C_3\subset S_2^{-1}$ and $C_4\subset S_2^{-N-1}$.

 \par We have $C_2=C_1+\ql$ in $A_1$ and $\qr_2(C_4)=C_3+r_i(\qm)$ in $A_2$.
 To find $d^W(C_1,C_4)$ we have to determine the image of $C_4$ under $\qr_1$.
 It depends actually on the highest number $j$ such that $S_2^{-j}$ (hence also $S_2^{0},\cdots,S_2^{-j+1}$) is in $A_1$.
 A classical result for affine buildings (clear for extended trees and generalized to hovels in \cite[2.9.2]{R11}) tells, then, that there is an apartment containing the stripes $S_2^{-j-1},\cdots,S_2^{-N-1}$ and the half apartment $\bigcup_{k\leq\qa_i(\ql)-j-1}\,S_1^k$.

 \par If $j=0$, $S_2^{-1}$ or $C_3$ is not in $A_1$, so $\qr_1(C_3)=C_2$ and, more generally, $\qr_1(S_2^{-k})=S_1^{\qa_i(\ql)+k-1}$, for $k\geq 1$ (see the picture below).
 We get $\qr_1(C_4)=C_2+\qm$ and $d^W(C_1,C_4)=\ql+\qm$.
 When $C_1$ and $C_4$ are fixed with this $W-$distance, we have to count the number of possible $C_2$.
 But $C_3\subset S_2^{-1}$ is in the enclosure of $C_1\subset S_1^0$ and $C_4\subset S_2^{-N-1}$: it is well determined by $C_1$ and $C_4$. Now $C_2$ has to share its panel of type $i$ with $C_3$ and to be neither in $S_2^{-1}$ nor in $S_1^{\qa_i(\ql)-1}$; so there are $q_i-1$ possibilities.

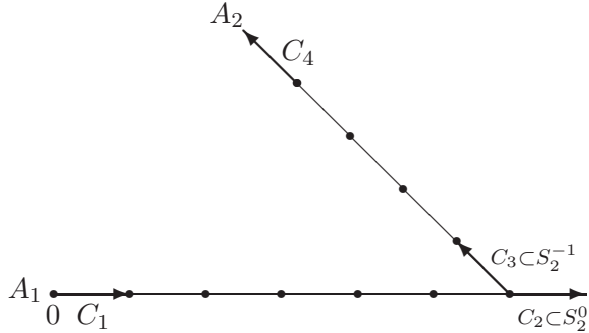
\begin{center}
 \setlength{\unitlength}{1cm}
\vskip-60pt
\begin{picture}(7,7)

\put(-0.1,0.6){$0$}
\put(0.3,0.6){$ C_1$}
\put(5.75,1.4){$\scriptstyle{C_3\subset S_2^{-1}}$}
\put(6.1,0.6){$\scriptstyle{C_2\subset S_2^0}$}
\put(-0.6,0.95){$A_1$}
\put(2.05,4.6){$A_2$}
\put(3,4.1){$C_4$}

\put(3.2,3.8){\circle*{0.1}}
\put(3.2,3.8){\circle*{0.1}}

\put(0,1){\circle*{0.1}}
\put(1,1){\circle*{0.1}}
\put(2,1){\circle*{0.1}}
\put(3,1){\circle*{0.1}}
\put(4,1){\circle*{0.1}}
\put(5,1){\circle*{0.1}}
\put(6,1){\circle*{0.1}}
\put(3.9,3.1){\circle*{0.1}}
\put(4.6,2.4){\circle*{0.1}}
\put(5.3,1.7){\circle*{0.1}}

\put(0,1){\line(1,0){6}}
\put(6,1){\line(-1,1){2.8}}

\thicklines
\put(0,1){\vector(1,0){1}}
\put(6,1){\vector(-1,1){0.7}}
\put(3.2,3.8){\vector(-1,1){0.7}}
\put(6,1){\vector(1,0){1}}

\end{picture}
\end{center}

 \par If $1\leq j\leq N-1$, then $A_1$ contains $S_2^{0}=S_1^{\qa_i(\ql)},S_2^{-1}=S_1^{\qa_i(\ql)-1},\cdots,S_2^{-j}=S_1^{\qa_i(\ql)-j}$ but not $S_2^{-j-1},\cdots,S_2^{-N-1}$ (see the picture below).
 So $\qr_1(S_2^{-k})=S_1^{\qa_i(\ql)-2j+k}$, for $k\geq j$.
 The image of the line segment $[x_2,x_4]=[x_2,x_2+\qm]$ under $\qr_1$ is $\qr_1([x_2,x_4])=[x_2,x_2+(j/N)r_i(\qm)]\cup[x_2+(j/N)r_i(\qm),x_2+(j/N)r_i(\qm)+((N-j)/N)\qm]$.
 As $N=\qa_i(\qm)$ and $r_i(\qm)=\qm-N\qa_i^\vee$, this means that $\qr_1(C_4)=C_2+\qm-j\qa_i^\vee$.
  When $C_1$ and $C_4$ are fixed with this $W-$distance, we have to count the number of possible $C_2$.
  As $S_1^{0},\cdots,S_1^{\qa_i(\ql)-j-1},S_2^{-j-1},\cdots,S_2^{-N-1}$ are well determined by $C_1,C_4$, we have to count the possibilities for $(S_1^{\qa_i(\ql)-j},\cdots,S_1^{\qa_i(\ql)})$.
  As above there are $q_i-1$ possibilities for $S_1^{\qa_i(\ql)-j}$ (or $q'_i-1$ if $j$ is odd) and then $q'_i$ (or $q_i$) possibilities for $S_1^{\qa_i(\ql)-j+1}$, \etc.
  Finally the total number of possibilities is $(q_i-1)q'_iq_iq'_i\cdots$ or $(q'_i-1)q_iq'_iq_i\cdots$ (according to $j$ being even or odd) with $j+1$ terms in the product.
  The last factor is necessarily $q_i$, so this total number is $(q_i^{*j+1}-q_i^{*j})$.

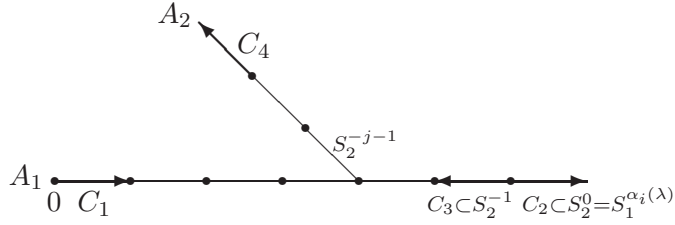
  \cache{J'ai enlev\'e la branche sans \'etiquette dans le dessin ci-dessous }
\cache{ 
\begin{center}
 \setlength{\unitlength}{1cm}
\vskip-60pt
\begin{picture}(7,7)

\put(-0.1,0.6){$0$}
\put(0.3,0.6){$ C_1$}
\put(4.9,0.6){$\scriptstyle{C_3\subset S_2^{-1}}$}
\put(6.15,0.6){$\scriptstyle{C_2\subset S_2^0=S_1^{\alpha_i(\lambda)}}$}
\put(-0.6,0.95){$A_1$}
\put(1.35,3.1){$A_2$}
\put(2.4,2.7){$C_4$}
\put(3.65,1.4){$\scriptstyle{S_2^{-j-1}}$}

\put(3.2,3.8){\circle*{0.1}}
\put(3.2,3.8){\circle*{0.1}}

\put(0,1){\circle*{0.1}}
\put(1,1){\circle*{0.1}}
\put(2,1){\circle*{0.1}}
\put(3,1){\circle*{0.1}}
\put(4,1){\circle*{0.1}}
\put(5,1){\circle*{0.1}}
\put(6,1){\circle*{0.1}}
\put(3.3,1,7){\circle*{0.1}}
\put(2.6,2.4){\circle*{0.1}}
\put(3.9,3.1){\circle*{0.1}}
\put(4.6,2.4){\circle*{0.1}}
\put(5.3,1.7){\circle*{0.1}}

\put(0,1){\line(1,0){6}}
\put(6,1){\line(-1,1){2.8}}
\put(4,1){\line(-1,1){1.4}}

\thicklines
\put(0,1){\vector(1,0){1}}
\put(6,1){\vector(-1,0){1}}
\put(2.6,2.4){\vector(-1,1){0.7}}
\put(6,1){\vector(1,0){1}}

\end{picture}
\end{center}
} 

\begin{center}
 \setlength{\unitlength}{1cm}
\vskip-60pt
\begin{picture}(7,7)

\put(-0.1,0.6){$0$}
\put(0.3,0.6){$ C_1$}
\put(4.9,0.6){$\scriptstyle{C_3\subset S_2^{-1}}$}
\put(6.15,0.6){$\scriptstyle{C_2\subset S_2^0=S_1^{\alpha_i(\lambda)}}$}
\put(-0.6,0.95){$A_1$}
\put(1.35,3.1){$A_2$}
\put(2.4,2.7){$C_4$}
\put(3.65,1.4){$\scriptstyle{S_2^{-j-1}}$}


\put(0,1){\circle*{0.1}}
\put(1,1){\circle*{0.1}}
\put(2,1){\circle*{0.1}}
\put(3,1){\circle*{0.1}}
\put(4,1){\circle*{0.1}}
\put(5,1){\circle*{0.1}}
\put(6,1){\circle*{0.1}}
\put(3.3,1,7){\circle*{0.1}}
\put(2.6,2.4){\circle*{0.1}}

\put(0,1){\line(1,0){6}}
\put(4,1){\line(-1,1){1.4}}

\thicklines
\put(0,1){\vector(1,0){1}}
\put(6,1){\vector(-1,0){1}}
\put(2.6,2.4){\vector(-1,1){0.7}}
\put(6,1){\vector(1,0){1}}

\end{picture}
\end{center}

 \par It is convenient to look at the cases $j=N$ or $j=N+1$ simultaneously.
 This means that $S_2^{-N}=S_1^{\qa_i(\ql)-N}$ is in $A_1$; in particular the panel $F_4^i$ of type $i$ of $C_4$ is in $A_1$, in the wall $\{x\mid\qa_i(x)=\qa_i(\ql)-N\}$.
 More precisely $F_4^i$ is the panel of type $i$ of $C'_4=C_1+\ql+r_i(\qm)\subset A_1$.
 This means that $(T_{\ql+r_i(\qm)}*T_i)(C_1,C_4)\geq 1$.
 Conversely if $C_1,C_4$ are fixed satisfying this condition, we can find $C_2,C_3$ with the required $W-$distances.
 We have now to count the number of possibilities for $C_2,C_3$ \ie for $C_2$ or for $(S_1^{\qa_i(\ql)-N},\cdots,S_1^{\qa_i(\ql)})$.
 The number of possibilities for $S_1^{\qa_i(\ql)-N}$ is exactly $(T_{\ql+r_i(\qm)}*T_i)(C_1,C_4)$.
 Then the number of possibilities for $S_1^{\qa_i(\ql)-N+1},\cdots,S_1^{\qa_i(\ql)}$ is alternatively $q_i$ or $q'_i$.
 Finally the total number of possibilities for $C_2$ is $q_i^{*N}(T_{\ql+r_i(\qm)}*T_i)(C_1,C_4)$ (as, when $N$ is odd, $q_i=q'_i$).
\end{proof}


\section{New basis}\label{s4}

In this section, we prove that left multiplication by $T_\mu$, for $\mu \in Y^{++}$, is injective. That allows us to introduce a new basis of the Iwahori-Hecke algebra $^I\shh_R$ in terms of $(T_w)_{w\in W^v}$ and $(X^\lambda)_{\lambda\in Y^+}$.

\medskip
We suppose $\Z\subset R$ and  each $q_i$, $q'_i$ in $R^\times$, the invertibles in $R$.
As we saw in 4.5,  $^I\shh_R$ is a free right $\shh_R(W^v)-$module with basis $\{T_\ql\mid\ql\in Y^{+}\}$.
For $\lambda\in Y^{++}$ and $H\in\shh_R(W^v)$, we say that $T_\ql*H$ 
is of degree $\ql$.

For $i\in I$ and $\Omega$ a subset of the model apartment $\mathbb A$, we write $c(i)(\Omega)$ the convex hull of $\Omega\cup r_i( \Omega)$. For $(i_1, i_2, \ldots, i_h)\in I^h$  and $(\ql_0, \ql_1, \ldots, \ql_h)\in  (Y^{++})^{h+1}$, we define :
$D(i_h) (\ql_{h-1} , \ql_h)= \ql_{h-1}+c(i_h) (\ql_h)$ and, by induction for $k$ from $h-1$ to 1,
$D(i_k,\ldots, i_h) (\ql_{k-1}, \ql_{k}, \ldots  , \ql_h)= \ql_{k-1}+c(i_k) (D(i_{k+1},\ldots, i_h) (\ql_{k}, \ql_{k+1}, \ldots  , \ql_h))$ (of course $c(i_h)(\lambda_h) = c(i_h)(\{\lambda_h\})$.

\begin{lemm}\label{4.1} With notation as above,
\medskip
\par a) if $\ql'_{h-1}\in D(i_h) (\ql_{h-1} , \ql_h)$, then

\noindent  $D(i_k,\ldots, i_{h-2}, i_{h-1}) (\ql_{k-1}, \ql_{k}, \ldots  ,\ql_{h-2}, \ql'_{h-1})\subset D(i_k,\ldots, i_{h-1}, i_h) (\ql_{k-1}, \ql_{k}, \ldots  , \ql_{h-1}, \ql_h)$;

\medskip
\par b) if $r_{i_1} r_{i_2}\cdots r_{i_h}$ is a reduced word in $W^v$ and $ \ql\in D(i_1,\ldots, i_h) (\ql_{0}, \ql_{1}, \ldots  , \ql_h)$, then $ \ql_{0}+r_{i_1} (\ql_1)+r_{i_1}r_{i_2} (\ql_{2})+\cdots +r_{i_1} r_{i_2}\cdots r_{i_h}(\ql_{h})\leq _{Q^\vee_\R}\, \ql$.
\end{lemm}

\begin {rema*}
If the expression $r_{i_1} r_{i_2}\cdots r_{i_h}$ is reduced, we get $D(i_1,\ldots, i_h) (0,0, \ldots  , 0,\ql_h) =conv (\{w(\ql_h)\mid w\leq_{_{B}} r_{i_1} r_{i_2}\cdots r_{i_h}\})$ where $\leq_{_{B}}$ denotes the Bruhat order.
\end{rema*}
\begin{proof}

\par The proof of a) is easy.

\par b) We have   $D(i_1,\ldots, i_h) (\ql_{0}, \ql_{1}, \ldots  , \ql_h)\subset   \ql_{0}+c(i_1) (\ql_1)+c(i_1,i_2) (\ql_{2})+\cdots+ c(i_1,i_2,\ldots ,i_h)(\ql_{h})$, with $c(i_1,i_2,\ldots ,i_k)(\ql_{k}) =c(i_1)(c(i_2)(\ldots ( c(i_k)(\ql_{k}))\!\ldots\!))=conv (\{w(\ql_k)\mid w\leq_{_{B}} r_{i_1} r_{i_2}\cdots r_{i_k}\})$ where $0\leq k\leq h$ and $\leq_{_{B}}$ denotes the Bruhat order.
For  $w\leq_{_{B}} r_{i_1} r_{i_2}\cdots r_{i_k}$, there is a sequence $w=w_0$, $w_1,\ldots,w_r=r_{i_1} r_{i_2}\cdots r_{i_k}$ such that, for each $1\leq i<r$, there is a reduced decomposition
$w_{i+1}=r_{j_1} r_{j_2}\cdots r_{j_{p-1}}r_{j_p}r_{j_{p+1}} \cdots r_{j_q}$
with $w_{i}=r_{j_1} r_{j_2}\cdots r_{j_{p-1}}r_{j_{p+1}}\cdots r_{j_q}$.
 Then 
 $w_{i}(\ql_k)=w_{i+1}(\ql_k)+\qa_{j_p}\big (r_{j_{p+1}} \cdots r_{j_q}(\ql_k)\big )r_{j_1} r_{j_2}\cdots r_{j_{p-1}}(\qa_{j_p}^\vee)$ and $Q^\vee_+ $ contains the term $\big (r_{j_q} \cdots r_{j_{p+1}}(\qa_{j_p})\big ) (\ql_k)r_{j_1} r_{j_2}\cdots r_{j_{p-1}}(\qa_{j_p}^\vee)$ by minimality of the expressions $r_{j_1} r_{j_2}\cdots r_{j_{p-1}}r_{j_p}$ and $r_{j_q} \cdots r_{j_{p+1}}r_{j_p}$.
 So we get by induction that $w(\ql_k)\geq_{Q^\vee}r_{i_1} r_{i_2}\cdots r_{i_k}(\ql_{k})$ and $w(\qm)\geq_{Q^\vee_\R}r_{i_1} r_{i_2}\cdots r_{i_k}(\ql_{k})$ for any $\qm\in c(i_1,\ldots,i_k)(\ql_k)$.
 The expected result is now clear.
\end{proof}

\begin{prop}\label{4.2}
For any expression $H_k=T_{\ql_0}*T_{i_1}*T_{\ql_1}*T_{i_2}*\cdots *T_{\ql_{k-1}}*T_{i_k}*T_{\ql_k}*H$ with $\ql_i\in Y^{++} $, 
$H\in \shh_\Z(W^v)$ and any $\qm \in Y^{++} $ sufficiently great, the product $T_\qm* H_k$ may be written as a $R$-linear combination of elements $T_{\qn}*H_{\qn}$ with $\qn \in \qm+ D(i_1,\ldots, i_k) (\ql_{0}, \ql_{1}, \ldots  , \ql_k)$ and $H_{\qn}\in  \shh_R(W^v)$.

\par Moreover, if  $r_{i_1} r_{i_2}\cdots r_{i_k}$ is a reduced word and $\qn_{0}= \qm + \ql_{0}+r_{i_1} (\ql_1)+r_{i_1}r_{i_2} (\ql_{2})+\cdots +r_{i_1} r_{i_2}\cdots r_{i_k}(\ql_{k})$, then $H_{\qn_{0}}\in R^\times T_{i_1}*T_{i_2}*\dots *T_{i_k}*H$ and more precisely the constant in $R^\times$ is a primitive monomial in the $q_i, {q'_i}$.
Further, $H_{\qn_{0}}$ is the only $H_{\qn}$ in $R\,T_{i_1}*T_{i_2}*\dots *T_{i_k}*H$.
\end{prop}
\begin{NB} So one may write $T_\qm* H_k=\sum_{\qn,w}\,a_{\qn,w}T_\qn*T_w$, with $a_{\qn,w}\in R$, $\qn$ running in $\qm+ D(i_1,\ldots, i_k) (\ql_{0}, \ql_{1}, \ldots  , \ql_k)$ and $w$ in $W$.
Moreover we get from the following proof, that each $a_{\qn,w}$ is a Laurent polynomial in the parameters $q_i, q'_i$, with coefficients in $\Z$; these polynomials depend only on the expression $H_k$, on $\A$ and on $W$.
\end{NB}

\begin {proof}

\par  The proof is easy in the following special case  (I).

(I).  We say that  the expression of $H_k$ is {\bf normalizable of length $k$} when it satisfies the following properties:

(i) $\ql_{k-1}-\ql_{k}\in Y^{++}$,

(ii) For all $h$ from $k$ to 2, $\ql_{h-2}-D(i_{h},\ldots, i_k) (\ql_{h-1}, \ql_{h}, \ldots  , \ql_k)\subset \overline{C^v_f}
$.

 \noindent For such an expression,  we write $D(H_k)= D(i_1,\ldots, i_k) (\ql_{0}, \ql_{1}, \ldots  , \ql_k)$.

We will then prove that $T_{\ql_0}*T_{i_1}*T_{\ql_1}*T_{i_2}*\cdots *T_{\ql_{k-1}}*T_{i_k}*T_{\ql_k}*H$ is a $\Z[ q_i, {q'_i}]-$linear combination of normalizable elements $H'_{k-1}$  of length $k-1$ such that  $D(H'_{k-1})\subset D(H_k)$.

Using the fact $\ql_{k-1}-\ql_{k}\in Y^{++}$ and  Theorem \ref{3.8}, or (BLT), for $T_{\ql_{k-1}}*T_{i_k}*T_{\ql_k}$, we have:

\noindent $
\begin{array}{cccc}
  (E)\quad H_k&= & &  q_{i_k}^{*({\qa_{i_k}(\ql_{k}))}} T_{\ql_0}*T_{i_1}*T_{\ql_1}*\cdots *T_{i_{k-1}}*T_{\ql^{(\qa_{i_k}(\ql_k))}_{k-1}}*(T_{i_k}*H)
  \\
 & &  + &\!{\displaystyle\sum_{ h=0}^{\qa_{i_k}(\ql_k)-1} }\! (q_i^{*(h+1)}-q_i^{*(h)})T_{\ql_0}*T_{i_1}*T_{\ql_1}*\cdots*T_{i_{k-1}} *T_{{\ql^{(h)}_{k-1}}}*H     \\

\end{array}$

\noindent  with ${\ql^{(h)}_{k-1}}=\ql_{k-1}+\ql_{k}-h\qa_{i_k}^\vee$, in particular $\ql^{(\qa_{i_k}(\ql_k))}_{k-1}=\ql_{k-1}+r_{i_k}(\ql_{k})$.
 Let us consider, for each $0\leq h\leq \qa_{i_k}(\ql_k)$, $\ql'_i=\ql_i$ for $i\leq k-2$ and $\ql'_{k-1}={\ql^{(h)}_{k-1}}$, then $(\ql'_0, \ldots, \ql'_{k-1})$ satisfies
 $\ql'_{k-2}-\ql'_{k-1}\in Y^{++}$, by (ii) for $h=k$ and $\ql'_{k-1}\in D(i_k) (\ql_{k-1}, \ql_k)$,
and,  for all $h$ from $k-1$ to 2, $\ql'_{h-2}-D(i_{h},\ldots, i_{k-1}) (\ql'_{h-1}, \ldots  , \ql'_{k-1})\subset \overline{C^v_f}$.
This last result comes from (ii) $\ql'_{h-2}-D(i_{h},\ldots, i_k) (\ql_{h-1}, \ql_{h}, \ldots  , \ql_k)\subset \overline{C^v_f}$ and  the inclusion $D(i_{h},\ldots, i_{k-1}) (\ql'_{h-1}, \ql'_{h}, \ldots  , \ql'_{k-1})
 \subset D(i_{h},\ldots, i_k) (\ql_{h-1}, \ql_{h}, \ldots  , \ql_k)$, coming from Lemma \ref{4.1} a). We have $T_{i_k}*H\in  \shh_R(W^v)$, so every term of the right hand side of (E) is a normalizable element $H'_{k-1}$ of length $k-1$ with $D(H'_{k-1})\subset D(H_k)$.

  By induction on each term, after $k$ steps, we obtain  $H_k$ as a $\Z[q_i, q'_i]-$linear combination of $T_{\qn}*H_{\qn}$ with $\qn \in  D(H_k)$ and $H_{\qn}\in  \shh_R(W^v)$.

Moreover, if the decomposition  $r_{i_1} r_{i_2}\cdots r_{i_k}$ is  reduced,
we take $\qn_0= \ql_{0}+r_{i_1} (\ql_1)+r_{i_1}r_{i_2} (\ql_{2})+\cdots +r_{i_1} r_{i_2}\cdots r_{i_k}(\ql_{k})$ and  look  more carefully at the decomposition
$(E)$. For $0\leq h<\qa_{i_k}(\ql_k)$, we have $\qn_{0}\notin D(T_{\ql_0}*T_{i_1}*T_{\ql_1}*\cdots *T_{{\ql^{(h)}_{k-1}}}*H) \subset D(H_k)$  by Lemma \ref{4.1}b). Indeed, if $\ql\in  D(T_{\ql_0}*T_{i_1}*T_{\ql_1}*\cdots *T_{{\ql^{(h)}_{k-1}}}*H) $, then, by minimality of $r_{i_1} r_{i_2}\cdots r_{i_k}$,
  $\qn_{0} \leq _{Q^\vee}\,\,\qn_{0}^{(h)} \leq _{Q^\vee} \ql$ with $\qn_{0}^{(h)} = \ql_{0}+r_{i_1} (\ql_1)+r_{i_1}r_{i_2} (\ql_{2})+ \cdots +r_{i_1} r_{i_2}\cdots r_{{k-1}}(\ql_{i_{k-1}}^{(h)})  \not= \qn_{0} $.
   So the unique term of degree $\qn_{0}$  of the final decomposition comes from the term of first kind (i.e. obtained like the first term of the right hand side of (E)) in every step of the reduction and is also the only term containing all the $T_{i_j}$.
   And so, we prove that, in front of the term $T_{\qn_0}*T_{i_1}*T_{i_2}*\cdots *T_{i_{k}}*H $ obtained for $\qn_0$, the constant is equal to the primitive monomial $C= q_{i_k}^{*(\qa_{i_k} (\ql_{k}))}q_{i_{k-1}}^{*(\qa_{i_{k-1}}(\ql_{{k-1}}+r_{i_k}(\ql_{k}) ))}\cdots q_{i_1}^{*(\qa_{i_1}(\ql_{{1}}+r_{i_2} (\ql_{2})+\cdots +r_{i_2}\cdots r_{i_k} (\ql_{k})))}$.

\medskip

 Let us consider now the general case and first prove the following result (II).
\medskip
\par
 (II) If  $H_k=T_{\ql_0}*T_{i_1}*T_{\ql_1}*T_{i_2}*\cdots *T_{\ql_{k-1}}*T_{i_k}*T_{\ql_k}*H$ with $\ql_i\in Y^{++} $,  $H\in \shh_R(W^v)$, we can choose  $\mu_0 \in Y^{++} $ such that
$T_{\mu_0}* H_k$ can be written as a $R$-linear combination of normalizable  expressions $H'_k$ of length $\leq k$ and with $D(H'_k)\subset \qm_0+ D(i_1,\ldots, i_k) (\ql_{0}, \ql_{1}, \ldots  , \ql_k) $.
\medskip
\par
We prove this result for $H_{k-h}= T_{\ql_h}*T_{i_{h+1}}*T_{\ql_{h+1}}*\cdots *T_{\ql_{k-1}}*T_{i_k}*T_{\ql_k}*H$ by  decreasing induction on $0\leq h\leq k-1$.
For $h=k-1$, we have $H_{1}= T_{\ql_{k-1}}*T_{i_k}*T_{\ql_k}*H$. Choose $\qm_{k-1} =\ql_{k}$, then $T_{\qm_{k-1}}* H_1$ is normalizable of length 1 and  $D(T_{\qm_{k-1}}* H_1)\subset \qm_{k-1}+D( i_k) (\ql_{k-1}, \ql_k) $.

Let  $0\leq h\leq k-2$ and suppose that we can choose $\qm_{h+1}\in Y^{++} $ such that $T_{\qm_{h+1}}*H_{k-(h+1)}= T_{\qm_{h+1}}*T_{\ql_{h+1}}*T_{i_{h+2}}*\cdots *T_{i_k}*T_{\ql_k}*H$ can be written as a $R-$linear combination of normalizable  expressions $H'_{k-(h+1)}$ of length $\leq {k-(h+1)}$ and with $D(H'_{k-(h+1)})\subset \qm_{h+1}+ D(i_{h+2},\ldots, i_k) (\ql_{h+1}, \ldots  , \ql_k) $.
Let us write these  normalizable expressions  $H'_{k-(h+1)}= T_{\ql'_0}*T_{i'_1}*T_{\ql'_1}*T_{i'_2}*\cdots *T_{i'_{k'}}*T_{\ql'_{k'}}*H'$, where $k'\leq {k-(h+1)}$ and $(\ql'_0, \ldots, \ql'_{k'})$ satisfies (i) and (ii).
Consider $\qm_{h}^{min}\in Y^{++} $ such that $\qm_{h}^{min}-D(i'_{1},\ldots, i'_{k'}) (\ql'_{0}, \ql'_{1}, \ldots  , \ql'_{k'})\subset \overline{C^v_f}$ for all these expressions.  We take  $\qm_{h}= \qm_{h}^{min}+2\qm_{h+1} +r_{i_{h+1}} (\qm_{h+1})$, then $T_{\qm_{h}} *H_{k-h}=T_{\qm_{h}} *T_{\ql_h}*T_{i_{h+1}}*H_{k-(h+1)}=T_{\qm_{h}^{min}+\ql_{h}+\qm_{h+1}}*T_{\qm_{h+1}+r_{i_{h+1}} (\qm_{h+1})}*T_{i_{h+1}}*H_{k-(h+1)} $.

By (BLT), we have:

 $\begin{array}{cc}(E') \quad &q_{i_{h+1}}^{*(\qa_{i_{h+1}}(\qm_{h+1}))}T_{\qm_{h} }*H_{k-h}= T_{\qm_{h}^{min}+\ql_{h}+2\qm_{h+1}}*T_{i_{h+1}}*T_{\qm_{h+1}}*H_{k-(h+1)}\\
 &\qquad -\displaystyle {\sum_{j=0}^{\qa_{i_{h+1}}( \qm_{h+1})-1}}(q_{i_{h+1}}^{*(j+1)}-q_{i_{h+1}}^{*(j)})T_{\ql_{h}+\qm_{h}^{min}+2\qm_{h+1}-j\qa_{i_{h+1}}^\vee}*T_{\qm_{h+1}}*H_{k-(h+1)}\\
\end{array}$.

\noindent  The choice of  $\qm_{h}^{min}$ and the hypothesis on $T_{\qm_{h+1}}*H_{k-(h+1)}$   allow us to say that we have written $T_{\qm_{h} }*H_{k-h}$ as a $R-$linear combination of normalizable  expressions $H'_{k-h}$ of length $\leq {k-h}$ with $D(H'_{k-h})\subset \qm_{h}^{min}+2\qm_{h+1}+ D(i_{h+1},\ldots, i_k) (\ql_{h},\ql_{h+1}+\qm_{h+1}, \ldots  , \ql_k) $ for the first term and  $D(H'_{k-h})\subset \qm_{h}^{min}+2\qm_{h+1}-j\qa_{i_{h+1}}^\vee+ D(i_{h+1},\ldots, i_k) (\ql_{h},\ql_{h+1}+\qm_{h+1}, \ldots  , \ql_k) $ for the others.
We need to be more precise to prove $D(H'_{k-h})\subset \qm_{h}+ D(i_{h+1},\ldots, i_k) (\ql_{h}, \ldots  , \ql_k) $.

By the part I) of this proof and the hypothesis on $T_{\qm_{h+1}}*H_{k-(h+1)}$  we know that this element can be written $\displaystyle {\sum_{\QL} }c_{\QL}T_{\QL}*H^{{\QL}}$ with $\QL=  \qm_{h+1}+ \QL'$ where $\QL'\in D(i_{h+2},\ldots, i_k) (\ql_{h+1}, \ldots  , \ql_k)
$ $c_\QL\in R$, $H^{{\QL}}\in\shh_R(W^v)$. The first term of the right hand side of $(E')$ becomes:

 \noindent $T_{\qm_{h}^{min}+\ql_{h}+2\qm_{h+1}}*T_{i_{h+1}}*(\displaystyle {\sum_{\QL} }c_{\QL}T_{\QL}*H^{{\QL}})= T_{\ql_{h}+2\qm_{h+1}}*(\displaystyle {\sum_{\QL} }c_{\QL}T_{\qm_{h}^{min}}*T_{i_{h+1}}*T_{\QL}*H^{{\QL}})$.

  By the condition on $\qm_{h}^{min}$ and (BLT), we write it

\noindent $\begin{array}{cc}  &T_{\ql_{h}+2\qm_{h+1}}*\biggl {(}\!\displaystyle {\sum_{\Lambda} }\!c_{\Lambda} \biggl {(}q_{i_{h+1}}^{* (\qa_{i_{h+1}}{(}\Lambda)) }T_{\qm_{h}^{min}+r_{i_{h+1}(\Lambda)}}*T_{i_{h+1}}*H^{{\Lambda}}\biggr {)}\biggr{)}\\
 &+    T_{\ql_{h}+2\qm_{h+1}}*\biggl {(}\!\displaystyle {\sum_{\Lambda} }\!c_{\Lambda} \biggl {(} \,\displaystyle {\sum_{j=0}^{\qa_{i_{h+1}} (\Lambda)-1}}  \!(q_{i_{h+1}}^{*(j+1)}-q_{i_{h+1}}^{*(j)})T_{\qm_{h}^{min}+\Lambda-j\qa_{i_{h+1}}^\vee}*H^{{\Lambda}}\biggr {)}\biggr{)}.\\
 \end{array}$

\noindent The first term of this sum will be $\displaystyle {\sum_{\Lambda} }c_{\Lambda}q_{i_{h+1}}^{* (\qa_{i_{h+1}} (\Lambda)) }T_{\ql_{h}+2\qm_{h+1}+\qm_{h}^{min}+r_{i_{h+1}(\Lambda)}}*T_{i_{h+1}}*H^{{\Lambda}}$ and ${\ql_{h}+2\qm_{h+1}+\qm_{h}^{min}+r_{i_{h+1}}(\Lambda)}= \ql_{h}+2\qm_{h+1}+\qm_{h}^{min}+r_{i_{h+1}}( \qm_{h+1})+r_{i_{h+1}}( \Lambda ' )= \ql_{h}+\qm_{h}+r_{i_{h+1}}( \Lambda ' )$  is an element of  $\ql_{h}+\qm_{h}+r_{i_{h+1}}( D(i_{h+2},\ldots, i_k) (\ql_{h+1}, \ldots  , \ql_k) )$ which is included, as expected,  in $\qm_{h}+ D(i_{h+1}, i_{h+2},\ldots, i_k) (\ql_{h}, \ql_{h+1}, \ldots  , \ql_k) $.

\parni The second term  is $\displaystyle {\sum_{\Lambda} }c_{\Lambda} \biggl {(}\displaystyle {\sum_{j=0}^{\qa_{i_{h+1}} (\Lambda)-1}}  (q_{i_{h+1}}^{*(j+1)}-q_{i_{h+1}}^{*(j)})T_{\ql_{h}+2\qm_{h+1}+\qm_{h}^{min}+\Lambda-j\qa_{i_{h+1}}^\vee}*H^{{\Lambda}} \biggr {)}$.
And we see that in fact (E') becomes (E''):

\noindent $\begin{array}{cl}  q_{i_{h+1}}^{*(\qa_{i_{h+1}}(\qm_{h+1}))}T_{\qm_{h} }*H_{k-h} & = \displaystyle {\sum_{\Lambda} }c_{\Lambda}q_{i_{h+1}}^{* (\qa_{i_{h+1}} (\Lambda)) }T_{ \ql_{h}+\qm_{h}+r_{i_{h+1}}( \Lambda ' )}*T_{i_{h+1}}*H^{{\Lambda}} \\
&\, +\displaystyle {\sum_{\Lambda} }c_{\Lambda}\displaystyle {\sum_{j=0}^{\qa_{i_{h+1}} (\Lambda)-1}} (q_{i_{h+1}}^{*(j+1)}-q_{i_{h+1}}^{*(j)})T_{\ql_{h}+2\qm_{h+1}+\qm_{h}^{min}+\Lambda-j\qa_{i_{h+1}}^\vee}*H^{{\Lambda}}\\
&\, -\displaystyle {\sum_{\Lambda} }c_{\Lambda}\displaystyle {\sum_{j=0}^{\qa_{i_{h+1}}( \qm_{h+1})-1}}(q_{i_{h+1}}^{*(j+1)}-q_{i_{h+1}}^{*(j)})T_{\ql_{h}+\qm_{h}^{min}+2\qm_{h+1}-j\qa_{i_{h+1}^\vee}}*T_{\Lambda}*H^{{\Lambda}}\\
 & = \displaystyle {\sum_{\Lambda} }c_{\Lambda}q_{i_{h+1}}^{* (\qa_{i_{h+1}} (\Lambda)) }T_{ \ql_{h}+\qm_{h}+r_{i_{h+1}}( \Lambda ' )}*T_{i_{h+1}}*H^{{\Lambda}}\\
&\, +\displaystyle {\sum_{\Lambda} }c_{\Lambda}\varepsilon_\Lambda \displaystyle {\sum_{j}} (q_{i_{h+1}}^{*(j+1)}-q_{i_{h+1}}^{*(j)})T_{\ql_{h}+2\qm_{h+1}+\qm_{h}^{min}+\Lambda-j\qa_{i_{h+1}}^\vee}*H^{{\Lambda}},\\
\end{array}$

\noindent where $\qa_{i_{h+1}}( \qm_{h+1})\leq j\leq\qa_{i_{h+1}}( \Lambda)-1$ and $\varepsilon_\Lambda= 1$  if $\qa_{i_{h+1}}( \qm_{h+1})\leq \qa_{i_{h+1}}( \Lambda)$ (\ie $\qa_{i_{h+1}}( \Lambda') \geq 0$) and  $\qa_{i_{h+1}}( \Lambda)\leq j\leq\qa_{i_{h+1}}( \qm_{h+1})-1$ and $\varepsilon_\Lambda= -1$   if $\qa_{i_{h+1}}( \qm_{h+1})\geq \qa_{i_{h+1}}( \Lambda)$  (\ie $\qa_{i_{h+1}}( \Lambda') \leq 0$).
For these values of $j$, by using $\Lambda-j\qa_{i_{h+1}}^\vee=r_{i_{h+1}} (\qm_{h+1})+ j'\qa_{i_{h+1}}^\vee+\Lambda '$ with $j'= \qa_{i_{h+1}}(\qm_{h+1})-j$,  we have
 $\ql_{h}+2\qm_{h+1}+\qm_{h}^{min}+\Lambda-j\qa_{i_{h+1}}^\vee=\ql_{h}+\qm_{h}+ j'\qa_{i_{h+1}}^\vee+\Lambda '$.
  If $\qa_{i_{h+1}}( \qm_{h+1})\leq \qa_{i_{h+1}}( \Lambda)$, $\qa_{i_{h+1}}( \qm_{h+1})-\qa_{i_{h+1}}( \Lambda)+1 \leq j'\leq 0 $ that is  $-\qa_{i_{h+1}}( \Lambda')+1 \leq j'\leq 0 $.
  If $\qa_{i_{h+1}}( \qm_{h+1})\geq \qa_{i_{h+1}}( \Lambda)$, then $\qa_{i_{h+1}}( \qm_{h+1})-\qa_{i_{h+1}}( \Lambda) \geq j'\geq 1 $ that is  $-\qa_{i_{h+1}}( \Lambda') \geq j'\geq 1 $.
 In all cases,
  $ j'\qa_{i_{h+1}}^\vee+\Lambda '$ is between $\Lambda '$ and $r_{i_{h+1}} (\Lambda ')$ and so, as expected,   $\ql_{h}+2\qm_{h+1}+\qm_{h}^{min}+\Lambda-j\qa_{i_{h+1}}^\vee\in \qm_{h}+ D(i_{h+1}, i_{h+2},\ldots, i_k) (\ql_{h}, \ql_{h+1}, \ldots  , \ql_k) $.

 So we have proved that $T_{\mu_0}* H_k$ can be written as a $R$-linear combination of normalizable  expressions $H'_k$ of length $\leq k$ and with $D(H'_k)\subset \qm_0+ D(i_1,\ldots, i_k) (\ql_{0}, \ql_{1}, \ldots  , \ql_k) $.
 By the I) of the proof we can write it as a $R$-linear combination of elements $T_{\qn}*H_{\qn}$ with $\qn \in \qm_{0}+ D(i_1,\ldots, i_k) (\ql_{0}, \ql_{1}, \ldots  , \ql_k)$ and $H_{\qn}\in  \shh_R(W^v)$.

Like in I) we can say, if moreover the decomposition  $r_{i_1} r_{i_2}\cdots r_{i_k}$ is  reduced, that only the term $\displaystyle {\sum_{\Lambda} }c_{\Lambda}q_{i_{h+1}}^{*( \qa_{i_{h+1}}( \Lambda)) }T_{\ql_{h}+2\qm_{h+1}+\qm_{h}^{min}+r_{i_{h+1}(\Lambda)}}*T_{i_{h+1}}*H^{{\Lambda}}$ (which contains $T_{i_{h+1}}$)  in $(E'')$ can give us  a term of lowest degree  $\qm_{h}+\ql_{h}+r_{i_{h+1}} (\ql_{h+1})+\cdots + r_{i_{h+1}}\cdots r_{i_k} (\ql_{k})$. More precisely,  the term of lowest degree comes from the term with $\Lambda_0 = \qm_{h+1}+\ql_{h+1}+r_{i_{h+2}} (\ql_{h+2})+\cdots + r_{i_{h+2}}\cdots r_{i_k} (\ql_{k})$ for which we have $ \qa_{i_{k+1} } (\Lambda_{0})\geq  \qa_{i_{k+1} } (\qm_{h+1})$.
So, it's easy to see by induction  that the coefficient of that term is a primitive monomial in the $q_i, {q'_i}$.
 \end{proof}

 \begin{coro}\label{4.3}

a)  For $\ql\in Y^+$ and $\qm\in Y^{++}$ sufficiently great, we have

 $T_\qm *T_\ql=\sum_{\ql  \leq _{Q^\vee} \qn\leq _{Q^\vee} \ql^{++}}T_{\qm+\qn} *H^\qn$ with $H^\qn\in \shh_R(W^v)$.

 b) More precisely,  if $H^\qn \not=  0$ then $\qm+\qn\in Y^{++}$ and $\qn $ is in the convex hull $conv(W^v. \ql^{++})$ of $W^v. \ql^{++}$ or better in the convex hull $conv(W^v. \ql^{++}, \geq \ql)$ of all $w'.\ql^{++}$ for $w'\leq_{_{B}}w_\ql $, with $w_\ql$ the smallest element of $W^v$ such that $\ql= w_\ql . \ql^{++}$.

 c) For $\qn=\ql$, $H^\ql$ is a strictly positive integer $a_\ql$ which may be written as a primitive monomial in $q_i, q'_i$, $i\in I$ (depending only on $\A$).

 d) In a) above, we may write $H^\qn=\sum_{w\in W^v}\,a_{\qm,\ql}^{\qn,w}T_w$ and, then each $a_{\qm,\ql}^{\qn,w}$ is a Laurent polynomial in the parameters $q_i, q'_i$ with coefficients in $\Z$, depending only on $\A$ and $W$.
 \end{coro}

 \begin {proof}
 Only the result c) is new (cf. Propositions \ref{PrFinite1} and \ref{PrFinite2}), and we already saw that the constant term in $ H^\ql $ is in $\Z_{> 0}$. We have to prove that $H^\ql\in \shh_R(W^v)$ is actually a constant (for $\qm $ sufficiently great).
Write $\ql= w_\ql (\ql^{++})$ (with $w_\ql$ minimal in $W^v$  for this property), choose a minimal decomposition $w_\ql=r_{i_1} r_{i_2}\cdots r_{i_k}$, by corollary \ref{3.3} we have $T_\ql=T_{i_1}*T_{i_2}*\cdots *T_{i_k}*T_{\ql^{++}} *T_{i_k}^{-1}*\cdots *T_{i_1}^{-1}$.
Then, by Proposition \ref{4.2}, for $\qm$ great, $T_\qm*T_\ql$  may be written as a $R$-linear combination of elements $T_{\qm+\qn}*(H^{\qn}_1*T_{i_k}^{-1}*\cdots * T_{i_1}^{-1})$ with
 $\qn \in  D(i_1,\ldots, i_k) (0,  \ldots ,0 , \ql^{++})$ and $H^{\qn}_1\in  \shh_R(W^v)$ with term of lowest degree $\qn_{0}= \ql$.
Moreover $H^{\ql}=H^{\ql}_1*T_{i_k}^{-1}*\cdots * T_{i_1}^{-1}$ is a primitive monomial in the $q_i, {q'_i}$.

\par To prove d), we remark that $T_{i_k}^{-1}*\cdots *T_{i_1}^{-1}$ may be written $\sum_{w\in W^v}\,a_wT_w$ with $a_w\in\Z[(q_i^{\pm 1})_{i\in I}]$ and we apply \ref{4.2} with $H=T_w$.
\end{proof}

 \begin{coro}\label{4.4}

 In $^I\shh_R$,  for $\qm\in Y^{++}$ the left multiplication by $T_\qm$ is injective.
\end{coro}
\begin {proof}

As $T_{\qm_1+\qm_2}=T_{\qm_1}*T_{\qm_2} $ for $\qm_1,\qm_2\in Y^{++}$, we may assume $\qm$ sufficiently great.
Let $H\in \, ^I\shh_R\setminus \{0\} $. We may write $H=\sum_{j\in J} T_{\ql_j}*H^j$ with $\ql_j\in Y^+$ and $0\not= H^j\in \shh_R(W^v)$. We choose $\ql_{j_0} $ minimal among the $\ql_j$ for $\leq _{Q^\vee}$.
Then  $T_\qm* H=\sum_{j\in J}\sum_{\qm+\ql_{j}\leq _{Q^\vee}\nu_j}T_{\qn_j}*H^{\qn_j, j}*H^j$. Hence $\qn_{j_0}=\qm+\ql_{j_0}$ is minimal for  $\leq _{Q^\vee}$ and $H^{\qn_{j_0}, j_0}$ is a monomial in $q_i$, $q'_i$; so $H^{\qn_{j_0}, j_0}*H^{j_0}\not = 0$ and $T_\qm*H\not= 0$.
\end{proof}

\begin{theo}\label{4.5}

\par 1) For any $\ql \in Y^+$, there is a unique $X^\ql\in\,  ^I\shh_R$ such that:
for all $\qm\in Y^{++}$ with $\ql+\qm \in Y^{++}$, we have $T_\qm *X^\ql=T_{\ql+\qm}$.

\par 2) More precisely, $$X^\ql =b_\ql T_\ql +\sum_{\qn} T_\qn*H'^\nu ,$$ where $H'^\qn\in \shh_R(W^v)$, $\qn\in conv(W^v. \ql^{++}, \geq \ql)\setminus \{\ql\}$ and $b_\ql $ is a primitive monomial in $q_i^{-1}, {q'_i}^{-1}$.

\par 3) For $\ql \in Y^{++}$, we have $X^\ql=T_\ql$.
For $\ql, \ql' \in Y^{+}$, $X^\ql * X^{\ql'}=X^{\ql + \ql'}=X^{\ql'}* X^{\ql}$.
\end{theo}
\begin{remas*} :
 \par a)  We have two bases for the free right $\shh_R(W^v)-$module $^I\shh_R$, $\{T_\ql\mid\ql\in Y^+\}$ and $\{X^\ql\mid\ql\in Y^+\}$. The change of bases matrix is triangular (for the order $\geq _{Q^\vee}$) with diagonal coefficients primitive monomials in $q_i^{-1}, {q'_i}^{-1}$.
  From \ref{4.3}.d  we get that all coefficients of this matrix are Laurent polynomials in the parameters  $q_i,q_i'$, with coefficients in $\Z$, depending only on $\A$ and on $W$.

 \par b) By 1) above and Corollary \ref{4.4}, it is clear that the left multiplication by $X^\ql$ is injective, for any $\ql\in Y^+$.
 \end{remas*}
 \begin{proof}

By Corollary \ref{4.4}, the uniqueness is clear and 3) follows from the relation $ T_\ql*T_\qm=T_{\ql+\qm }$ of the Theorem \ref{ThAlgebra}. We have just to prove 1) and 2) for a $\qm\in Y^{++}$ (chosen sufficiently great).

We argue by induction on the height $ht(\ql^{++}-\ql )$ of $\ql^{++}-\ql$ with respect to the free family $(\qa_i^\vee)$ in $Q^\vee$. When the height is 0, $\ql=\ql^{++}$ and $X^\ql=T_\ql$.
By Corollary \ref{4.3}, we write
$$T_\qm *T_\ql=a_\ql T_{\qm+\ql} +\displaystyle{\sum_{\ql  \leq _{Q^\vee} \qn\leq _{Q^\vee} \ql^{++};  \ql  \not=  \qn}T_{\qm+\qn} *H^\qn}
$$ with $H^\qn\in \shh_R(W^v)$ and $\qn\in conv (W^v.\ql^{++})$ hence $\qn^{++}\in conv (W^v.\ql^{++})$ (in particular $\qn^{++}\leq_{Q^\vee}\ql^{++}$ ). \cf Lemma \ref{1.10}.a).

So $ht(\qn^{++}-\qn )<ht(\ql^{++}-\ql )$. By induction and for $\mu$ sufficiently great, we can  consider the element $X^\qn$ such that $T_{\qm+\qn}=T_\qm *X^\qn $; we can write it $X^\qn=\displaystyle{\sum_{\qn  \leq _{Q^\vee} \qn'\leq _{Q^\vee} \qn^{++}}T_{\qn'} *H^{\qn',\nu}}  $  and we may take $X^\ql={a_\ql}^{-1}T_\ql- \bigg (\displaystyle{\sum_{\ql  \leq _{Q^\vee} \qn\leq _{Q^\vee} \ql^{++};  \ql  \not=  \qn}X^\qn*H^\nu} \bigg )\\ {\hskip 9em}
={a_\ql}^{-1}T_\ql- \biggl {(}\displaystyle{\displaystyle{\sum_{\ql  \leq _{Q^\vee} \qn\leq _{Q^\vee} \ql^{++};  \ql  \not=  \qn}\biggl {(}\sum_{\qn  \leq _{Q^\vee} \qn'\leq _{Q^\vee} \qn^{++}}}T_{\qn'} *H^{\qn',\nu}\biggr {)}*H^\nu}\biggr {)}$.

\end{proof}

\begin {prop} \label{4.6}

For $\ql\in Y^+$ and $i\in I$ we have the following relations :

\par a) If $\qa_i(\ql) \geq 0$, then  $T_i*X^\ql= q_i^{* (\qa_{i}(\ql))}X^{r_i(\ql)}*T_i+\displaystyle{\sum_{h=0}^{\qa_i(\ql)-1} (q_i^{* (h+1)}-q_i^{* (h)}) X^{\ql-h\qa_i^\vee}}.$

\par b) If $\qa_i(\ql)< 0$, then

$T_i*X^\ql= {1\over q_i^{* (-\qa_{i}(\ql))}}X^{r_i(\ql)}*T_i- {1\over q_i^{* (-\qa_{i}(\ql))}}\displaystyle{\sum_{h=\qa_i(\ql)}^{-1}} \Bigl{(}{q_i}^{* (-\qa_{i}(\ql)+h+1)}-{q_i}^{*{(-\qa_{i}(\ql)}+ h)}\Bigr {)} X^{\ql-h\qa_i^\vee}$. 
\end{prop}
\begin{NB} These relations are the Bernstein-Lusztig relations for the $X^\ql$, (BLX) for short.
\end{NB}
\cache {On remarque qu'en multipliant la formule du b) par un terme $ q_i^{*2N}$ pour N assez grand les d\'enominateurs de la seconde formule se simplifient et on obtient $ q_i^{*(2N)}T_i*X^\ql=  q_i^{*( 2N+\qa_{i}(\ql))}X^{r_i(\ql)}*T_i- \displaystyle{\sum_{h=\qa_i(\ql)}^{-1}} \Bigl{(}{q_i}^{* (2N+h+1)}-{q_i}^{*(2N+ h)}\Bigr {)} X^{\ql-h\qa_i^\vee}$ qui est  une formule ressemblant \`a celle obtenue \`a partir de a) par la m\^eme op\'eration. (on choisira la fa\c{c}on de pr\'esenter ceci en fonction des choix du paragraphe 5)

Si on compare a) pour $\ql$ avec $\qa_i (\ql)\geq 0$ et b) pour $r_i(\ql )$, on trouve $T_i*X^\ql- q_i^{* (\qa_{i}(\ql))}X^{r_i(\ql)}*T_i=X^\ql*T_i- q_i^{* (\qa_{i}(\ql))}T_i*X^{r_i(\ql)}$. }

\begin{proof}

If $\ql\in Y^{++}$, by Proposition \ref{3.8} a), we know that
$X^\ql*T_i*X^\ql=X^{\ql+\ql}*T_i$  when $\qa_i(\ql)=0$ and, when $\qa_i(\ql)>0$,
 $X^\ql*T_i*X^\ql=q_i^{*\qa_i(\ql)}X^{\ql+r_i(\ql)}*T_i+(q_i^{*(\qa_i(\ql))}-q_i^{*(\qa_i(\ql)-1)})X^{\ql+\ql-(\qa_i(\ql)-1)\qa_i^\vee}+ \cdots
+ (q_i^{*2}-q_i)X^{\ql+\ql-\qa_i^\vee}+(q_i-1)X^{\ql+\ql}$, so we have the result.
\medskip

In the general case, $\ql\in Y^+$,
we write $\ql= \qm-\qn$ with  $\qm, \qn$ chosen in $Y^{++}$. By Theorem \ref{4.5}, $X^\qn * X^\ql=X^\qm$. From (BLX) for $X^\qm$ and $X^\qn$, we have~:

\parni $$
  T_i * X^\qm = q_i^{* (\qa_{i}(\ql+\qn))}X^{r_i(\ql+\qn)}*T_i+\displaystyle{\sum_{h=0}^{\qa_i(\ql+\qn)-1} (q_i^{* (h+1)}-q_i^{* (h)}) X^{\qn+\ql-h\qa_i^\vee}}$$

which can also be written
$$
\begin{array}{rcl}
T_i*X^{\qn+\ql} =(T_i*X^\qn)*X^\ql & = &
\bigg ( q_i^{* (\qa_{i}(\qn))}X^{r_i(\qn)}*T_i+\displaystyle{\sum_{h=0}^{\qa_i(\qn)-1} (q_i^{* (h+1)}-q_i^{* (h)}) X^{\qn-h\qa_i^\vee}}\bigg )*X^\ql\\
 & = & q_i^{* (\qa_{i}(\qn))}X^{r_i(\qn)}*T_i*X^\ql+\displaystyle{\sum_{h=0}^{\qa_i(\qn)-1} (q_i^{* (h+1)}-q_i^{* (h)}) X^{\qn+\ql-h\qa_i^\vee}}.
\end{array}$$

\parni If  $\qa_i(\ql) \geq 0$, we obtain

$$
q_i^{* (\qa_{i}(\qn))}X^{r_i(\qn)}*T_i*X^\ql =  q_i^{* (\qa_{i}(\ql+\qn))}X^{r_i(\qm)}*T_i+ \displaystyle{\sum_{h=\qa_i(\qn)}^{\qa_i(\ql+\qn)-1} (q_i^{* (h+1)}-q_i^{* (h)}) X^{\qn+\ql-h\qa_i^\vee}}.
$$
We take $h'=h-\qa_i(\qn)$,  then $X^{\qn+\ql-h\qa_i^\vee}= X^{\qn-\qa_i(\qn)\qa_i^\vee+\ql-h'\qa_i^\vee}=X^{r_i(\qn)+\ql-h'\qa_i^\vee}$ and $q_i^{* (\qa_i(\qn)+h')}=q_i^{* \qa_i(\qn)}q_i^{*h'}$ (by $q_i=q_i'$ if $\qa_i(\qn)$ is odd, and an easy calculation  if $\qa_i(\qn)$ is even).
  So,  $q_i^{* (\qa_{i}(\qn))}X^{r_i(\qn)}*T_i*X^\ql=q_i^{* (\qa_{i}(\qn))}X^{r_i(\qn)} * \bigg (q_i^{* (\qa_{i}(\ql))}X^{r_i(\ql)}*T_i+\displaystyle{\sum_{h'=0}^{\qa_i(\ql)-1} (q_i^{* (h'+1)}-q_i^{* (h')}) X^{\ql-h'\qa_i^\vee}}\bigg )$. And we are done thanks to the injectivity of left multiplication by $X^{r_i(\nu)}$.

\parni If  $\qa_i(\ql) < 0$, we obtain
$$
q_i^{* (\qa_{i}(\qn))}X^{r_i(\qn)}*T_i*X^\ql =  q_i^{* (\qa_{i}(\ql+\qn))}X^{r_i(\ql+\qn)}*T_i- \displaystyle{\sum_{h=\qa_i(\ql+\qn)}^{\qa_i(\qn)-1} (q_i^{* (h+1)}-q_i^{* (h)}) X^{\qn+\ql-h\qa_i^\vee}}.
$$

  We have  $q_i^{* (\qa_{i}(\qn))}=q_i^{* (-\qa_{i}(\ql) )}q_i^{* (\qa_{i}(\ql+\qn))}$ by an easy calculus if  $\qa_{i}(\qn)$ and $\qa_{i}(\ql)$ are even and because $q_i=q'_i$ whenever $\qa_{i}(\qn)$ or $\qa_{i}(\ql)$ is odd. So,

  \parni $X^{r_i(\qn)}*T_i*X^\ql ={1\over q_i^{* (-\qa_{i}(\ql))}}X^{r_i(\ql+\qn)}*T_i -{1\over q_i^{* (\qa_{i}(\qn))}}  \displaystyle{\sum_{h=\qa_i(\ql+\qn)}^{\qa_i(\qn)-1} (q_i^{* (h+1)}-q_i^{* (h)}) X^{\qn+\ql-h\qa_i^\vee}}$ and we have (because of the injectivity of the left multiplication by $X^{r_i(\qn)}$):

 \parni $
 T_i*X^\ql ={1\over q_i^{* (-\qa_{i}(\ql))}}X^{r_i(\ql)}*T_i -{1\over q_i^{* (\qa_{i}(\qn))}}\displaystyle{\sum_{h=\qa_i(\ql+\qn)}^{\qa_i(\qn)-1} (q_i^{* (h+1)}-q_i^{* (h)}) X^{\ql+(\qa_i(\qn)-h)\qa_i^\vee}}\\
 ={ 1\over q_i^{* (-\qa_{i}(\ql))}}X^{r_i(\ql)}*T_i -{1\over q_i^{* (\qa_{i}(\qn))}q_i^{* (-\qa_{i}(\ql))}} \displaystyle{\sum_{h=\qa_i(\ql)}^{-1} (q_i^{* (\qa_{i}(\qn)-\qa_{i}(\ql)+h+1)}-q_i^{* (\qa_{i}(\qn)-\qa_{i}(\ql)+h)}) X^{\ql-h\qa_i^\vee}}\\
 ={ 1\over q_i^{* (-\qa_{i}(\ql))}}X^{r_i(\ql)}*T_i -{1\over q_i^{* (-\qa_{i}(\ql))}} \displaystyle{\sum_{h=\qa_i(\ql)}^{-1} (q_i^{* (-\qa_{i}(\ql)+h+1)}-q_i^{* (-\qa_{i}(\ql)+h)}) X^{\ql-h\qa_i^\vee}}$

  \end{proof}

\subsection{The classical Bernstein-Lusztig relation} \label{4.7}

\par The {\it module} $\qd: Q^\vee\to R$ is defined by $\qd(\sum_{i\in I} a_i\qa_i^\vee)=\prod_{i\in I} (q_iq'_i)^{a_i}$ \cite[5.3.2]{GR13}.
After replacing eventually $R$ by a bigger ring $R'$ containing some square roots $\sqrt{q_i},\sqrt{q'_i}$ of $q_i,q'_i$ (with  $\sqrt{q_i}=\sqrt{q'_i}$, if $q_i=q'_i$), we assume moreover that there exists an homomorphism $\qd^{1/2}:Y\to R^\times$, such that $\qd(\ql)=(\qd^{1/2}(\ql))^2$ for any $\ql\in Q^\vee$ and $\qd^{1/2}(\qa_i^\vee)=\sqrt{q_i}.\sqrt{q'_i}$.
In particular $\sqrt{q_i}^{\pm 1}$ and $\sqrt{q'_i}^{\pm 1}$ are well defined in $R^\times$. 
In the common example where $R=\R$ or $\C$, these expressions are chosen to be the classical ones:  $\qd^{1/2}(Y)\subset \R_+^*$.

\par We define $H_i= (\sqrt{q_i})^{-1} T_i$ and   $Z^\ql=\qd^{-1/2} (\ql) X^\ql$ for $\ql\in Y^+$.
 When $w=r_{i_1}\cdots r_{i_n}$ is a reduced decomposition, we set $H_w=H_{i_1}*\cdots *H_{i_n}$; this does not depend of the chosen decomposition of $w$.

\par We may translate the relations (BLX) for these elements.

\begin {prop*}

For $\ql\in Y^{++}$, we have the following relation:

 $ H_i*Z^\ql= Z^{r_i(\ql)}*H_i+\displaystyle{\sum_{k=0}^{\lfloor{\qa_i(\ql)-1\over 2}\rfloor} (\sqrt{q_i}-\sqrt{q_i}^{-1}){ Z^{\ql-(2k)\qa_i^\vee}}}+\displaystyle{\sum_{k=0}^{\lfloor{\qa_i(\ql)\over 2}\rfloor -1}(\sqrt{q'_i }-\sqrt{q'_i}\,^{-1}){ Z^{\ql-(2k+1)\qa_i^\vee}}}$.
 \end {prop*}

 \begin{remas*} 1) This is the Bernstein-Lusztig relation for the $Z^\ql$, (BLZ) for short.

 \par 2) In the following section, we shall consider an algebra containing $^I\shh_R$ and, for any $i\in I$, an element $Z^{-{\qa_i}^\vee}$ satisfying $Z^{\ql-h{\qa_i}^\vee}=Z^\ql*(Z^{-{\qa_i}^\vee})^h$ for $h\in\N$, $\ql,\ql-h\qa_i^\vee\in Y^+$.
 In such an algebra the relation (BLZ) may be rewritten (using that $\sqrt{q_i}=\sqrt{q_i'}$ if $\qa_i(\ql)$ is odd) as the classical Bernstein-Lusztig relation (BL):

\par $ H_i*Z^\ql=Z^{r_i(\ql)}*H_i+(\sqrt{q_i}-\sqrt{q_i}^{-1})\displaystyle{{Z^\ql-Z^{r_i(\ql)}}\over{1-Z^{-2\qa_i^\vee}}}+(\sqrt{q'_i }-\sqrt{q'_i}\,^{-1})\displaystyle{{Z^{\ql-\qa_i^\vee}-Z^{r_i(\ql)-\qa_i^\vee}}\over{1-Z^{-2\qa_i^\vee}}}$

\parni \ie $ H_i*Z^\ql-Z^{r_i(\ql)} *H_i= b (\sqrt{q_i}, \sqrt{q'_i}; Z^{-{\qa_i}^\vee}) (Z^\ql-Z^{r_i(\ql)})$
 where $b(t,u;z)= \displaystyle {{t-t^{-1}+(u-u^{-1})z}\over {1-z^2}}$.
 This is the same relation as in \cite[4.2]{Ma03}, up to the order; see below in 3).


 %


\par 3) Actually this relation (BLZ) is still true when $\ql\in Y^+$ and $\qa_i(\ql)\geq 0$ (same proof as below).
If $\qa_i(\ql)<0$, we leave to the reader the proof of the following relation:

\par $T_i*Z^\ql
 = Z^{r_i(\ql)}*T_i-\Biggl{(}\displaystyle{\sum_{h \, even, h=2}^{-\qa_i(\ql)}} \Bigl{(}{q_i}
-{1}\Bigr {)}  Z^{\ql+h\qa_i^\vee}+\displaystyle{\sum_{h\, odd, h=1}^{-\qa_i(\ql)}} \Bigl{(} \sqrt{q_i.q'_i}-{\sqrt{q_i.q'_i}\over q'_i}\Bigr {)}Z^{\ql+h\qa_i^\vee}\Biggr {)}$

\par In the situation of 2) above, it may be rewritten:

\par $ H_i*Z^\ql-Z^{r_i(\ql)}*H_i=(\sqrt{q_i}-\sqrt{q_i}^{-1})\displaystyle{{Z^\ql-Z^{r_i(\ql)}}\over{1-Z^{-2\qa_i^\vee}}}+(\sqrt{q'_i }-\sqrt{q'_i}\,^{-1})\displaystyle{{Z^{\ql-\qa_i^\vee}-Z^{r_i(\ql)-\qa_i^\vee}}\over{1-Z^{-2\qa_i^\vee}}}$

\par\qquad\qquad\qquad\qquad\quad$= b (\sqrt{q_i}, \sqrt{q'_i}; Z^{-{\qa_i}^\vee}) (Z^\ql-Z^{r_i(\ql)})$

\par It is the same relation (BLZ) as above. Moreover, it's easy to see  in the first equality that $ H_i*Z^\ql-Z^{r_i(\ql)}*H_i=Z^\ql*H_i-H_i*Z^{r_i(\ql)}$.
Actually we shall see in section \ref {s5} that this same relation is true for any $\ql\in Y$ in a greater algebra containing elements $Z^\ql$ for $\ql\in Y$.
\cache{Reformuler cette phrase apr\`es avoir \'ecrit le {\S{}} 5.}
 \end{remas*}
 \begin{proof}

From $Z^\ql=\qd^{-1/2} (\ql) X^\ql$ and $\qd^{1/2}(\qa_i^\vee)=\sqrt{q_i.q'_i}$ , we get

\parni $
\begin{array}{ccc}Z^{\ql-h\qa_i^\vee} & =&\qd^{-1/2} (\ql-h\qa_i^\vee) X^{\ql-h\qa_i^\vee}\\
& =&\qd^{-1/2} (\ql) (\qd^{1/2}( \qa_i^\vee))^h X^{\ql-h\qa_i^\vee}\\
& =&\qd^{-1/2} (\ql) (\sqrt{q_i.q'_i})^h X^{\ql-h\qa_i^\vee}\\
\end{array}$

\parni By  $\qa_i(\ql) \geq 0$ and  (BLX) we have

\parni  $$ T_i*Z^\ql= q_i^{* (\qa_{i}(\ql))}\bigl {(}\sqrt{q_i.q'_i}\,  \big {)}^{-\qa_i(\ql)}Z^{r_i(\ql)}*T_i+\displaystyle{\sum_{h=0}^{\qa_i(\ql)-1} (q_i^{* (h+1)}-q_i^{* h})}\bigl {(}\sqrt{q_i.q'_i}\, \bigr {)}^{(-h)}{ Z^{\ql-h\qa_i^\vee}}.$$
Moreover $q_i^{* h}=q_iq'_iq_i\cdots$ with $h$ terms in the product so $q_i^{* h}=\bigl {(}\sqrt{q_i.q'_i}\, \bigr {)}^{h}$ if $h$ is even and $q_i^{* h}=q_i\bigl {(}\sqrt{q_i.q'_i}\, \bigr {)}^{(h-1)}$ if $h$ is odd.  So,  if $\qa_i(\ql)$ is even,
 we have
$$T_i*Z^\ql= Z^{r_i(\ql)}*T_i+\displaystyle{\sum_{k=0}^{{\qa_i(\ql)-2}\over 2} (q_i-1){ Z^{\ql-(2k)\qa_i^\vee}}}+\displaystyle{\sum_{k=0}^{{\qa_i(\ql)-2\over 2}} (q_iq'_i-q_i)\bigl {(}\sqrt{q_iq'_i} \, \bigr {)}^{-1}{ Z^{\ql-(2k+1)\qa_i^\vee}}}.
$$ If $\qa_i(\ql)$ is odd, we have $q_i=q'_i$,  and we obtain  $T_i*Z^\ql= Z^{r_i(\ql)}*T_i+\displaystyle{\sum_{h=0}^{\qa_i(\ql)-1} }(q_i-1){ Z^{\ql-h\qa_i^\vee}}$.
  In both cases, by $H_i= (\sqrt{q_i})^{-1} T_i$, we get:

\parni $H_i*Z^\ql= Z^{r_i(\ql)}*H_i+\displaystyle{\sum_{k=0}^{\lfloor{\qa_i(\ql)-1\over 2}\rfloor} (\sqrt{q_i}-\sqrt{q_i}^{-1}){ Z^{\ql-(2k)\qa_i^\vee}}}+\displaystyle{\sum_{k=0}^{\lfloor{\qa_i(\ql)\over 2}\rfloor -1} (\sqrt{q'_i}-\sqrt{q'_i}\,^{-1}){ Z^{\ql-(2k+1)\qa_i^\vee}}}.
$

\end{proof}


\section{Bernstein-Lusztig-Hecke Algebras}\label{s5}

The aim of this section is to define, in a formal way, an  associative algebra
 $^{BL}\mathcal H_R$, called the Bernstein-Lusztig-Hecke algebra. This construction by generators and relations is motivated by the results obtained in the previous section (in particular \ref {4.6}) and we will be able next to identify $^I\shh_R$ and a subalgebra of  $^{BL}\mathcal H_R$ (up to some hypotheses on $R$).


 We use the same notations as before, even if the objects are somewhat different. This choice will be justified by the identification obtained at the end of this section.

 \par We consider $\A$ as in \ref{1.2} and $Aut(\A)\supset W=W^v\ltimes Y\supset W^a$, with $Y$ a discrete group of translations.

 \subsection{The module $^{BL}\mathcal H_{R_1}$ }\label{5.0}

 We consider now the ring $R_1=\Z[{(\qs_{i}}^{\pm 1}, {\qs'_{i}}^{\pm 1})_{ i\in I} ]$ where the indeterminates ${\qs_{i}}, {\qs'_{i}}$ satisfy the following relations (as $q_i$ and $q'_i$ in \ref{1.3}.5 because in the further identification, ${\qs_{i}}, {\qs'_{i}}$ will play the role of  $\sqrt{q_i}$ and $\sqrt {q'_i}$).

 If $\qa_i(Y)=\Z$, then ${\qs_{i}}={\qs'_{i}}$.

 If $r_i$ and $r_j$ are conjugated (i.e. if $\qa_i(\qa_j^\vee)=\qa_j(\qa_i^\vee)=-1$), then ${\qs_{i}}={\qs_{j}}={\qs'_{i}}={\qs'_{j}}$.

 We denote by  $^{BL}\mathcal H_{R_1}$ the free $R_1$-module with basis $(Z^\ql H_w)_{\ql\in Y, w\in W^v}$.
 For short, we write $H_i=H_{r_i}, H_w=Z^0H_w$ and $Z^\ql=Z^\ql H_1$.

\begin{theo}\label{5.1}
 There exists a unique multiplication $*$ on $^{BL}\mathcal H_{R_1}$ which makes it an associative unitary $R_1$-algebra with unity $H_1$ and satisfies the following conditions:

 (1) $ \forall \ql\in Y \quad  \forall w\in W^v \qquad Z^\ql*H_w= Z^\ql H_w$,

 \smallskip

  (2) $ \forall i\in I \quad \forall w\in W^v \qquad H_i*H_w= H_{r_iw}\, $ if $\,\ell (r_iw)>\ell (w)$

 \qquad\qquad\qquad\qquad\qquad\qquad\qquad \quad$ = ({\qs_{i}}-{\qs_{i}}^{-1}) H_w+H_{r_iw}\,$ if $\, \ell (r_iw)<\ell (w)$,

   \smallskip
  (3) $ \forall \ql\in Y \quad  \forall \qm\in Y \qquad Z^\ql*Z^\qm= Z^{\ql+\qm}$,

   \smallskip

  (4) $\forall \ql\in Y \quad  \forall i\in I \qquad H_i*Z^\ql-Z^{r_i(\ql)}*H_i=b(\qs_i, \qs'_i; Z^{-\qa_i^\vee}) (Z^\ql-Z^{r_i(\ql)})$; where $b(t,u; z)={{(t-t^{-1}) + (u-u^{-1})z}\over {1-z^2}}$.

  \end{theo}

 \begin {remas}\label{5.1a}

  1) It is already known (see \eg \cite[Th. 7.1]{Hu90} or \cite[IV {\S{}} 2 exer. 23]{Bo68}) that the free submodule with basis $(H_w)_{w\in W^v}$ can be equipped, in a unique way, with a multiplication $*$ that satisfies (2) and gives it a structure of an associative unitary algebra called the "Hecke algebra of the group $W^v$ over $R_1$" and denoted by $\mathcal H_{R_1} (W^v)$.

  2) The submodule $\mathcal H_{R_1} (Y)$ with basis $(Z^\ql)_{\ql\in Y}$ will be a commutative subalgebra.

  \par 3)  When all $\qs_i,\qs'_i$ are equal, the existence of this algebra $^{BL}\mathcal H$  is stated in \cite{GaG95}  and justified by an action on some Grothendieck group.

  \par 4) This $R_1-$algebra depends only on $\A$ and $Y$ (\ie $\A$ and $W$).
  We call it the Bernstein-Lusztig-Hecke algebra over $R_1$ (associated to $\A$ and $W$).
\end{remas}

 \subsection{Proof of  Theorem \ref{5.1}}\label{5.1b}

  1) The uniqueness of the multiplication $*$ is clear: by associativity and distributivity, we have only to identify $H_w*Z^\qm$. If $w=r_{i_1}r_{i_2} \cdots r_{i_n}$ is a reduced decomposition, then, by (2), (4) and remark 1), $H_w*Z^\qm=H_{i_1}*(H_{i_2}*(\cdots *(H_{i_n}*Z^\qm)\cdots ))$ has to be a well defined linear combination of terms $Z^\qn H_u$ : $H_w*Z^\qm=\sum_{k} a_k Z^{\qn_k}H_{u_k}$ with $a_k\in R_1$, $\qn_k\in Y, u_k\in W^v$.

  2) Construction of $*$. We define $H_w*Z^\qm$ as above and we have to prove that it does not depend on the reduced decomposition $w=r_{i_1}r_{i_2} \cdots r_{i_n}$.

  a) We define $L_i\in End_{R_1} (^{BL}\mathcal H_{R_1})$ by : 

  $L_i(Z^\qm H_w)= H_i*(Z^\qm H_w)=Z^{r_i(\qm)} (H_i*H_w)+b(\qs_i, \qs'_i; Z^{-\qa_i^\vee}) (Z^\qm-Z^{r_i(\qm)})*H_w$

  \noindent  where $H_i*H_w=H_{r_iw}$ if $\ell (r_iw)>\ell (w)$
and  $H_i*H_w=({\qs_{i}}-{\qs_{i}}^{-1}) H_w+H_{r_iw}$ if $\ell (r_iw)<\ell (w)$.

By Matsumoto's theorem \cite[IV 1.5 prop. 5]{Bo68}, the expected independence will be a consequence of the braid relations, i.e.:

$(*)\qquad L_i(L_j(L_i(\ldots (Z^\ql H_w)\ldots)))=L_j(L_i(L_j(\ldots (Z^\ql H_w)\ldots)))$ (with $m_{i,j}$ factors $L$ on each side), whenever the order $m_{i,j}$ of $r_ir_j$ is finite.

As $\mathcal H_{R_1} (W^v)$ is known to be an algebra, it is enough  to prove $(*)$ for $w=1$. We may also suppose $ \qa_j(\qa_i^\vee)\not=0$ as otherwise $L_i$ and $L_j$ commute clearly.

We choose $i, j\in I$ with $m_{i,j}$ finite, then $\pm \qa_i, \pm \qa_j$ generate a finite root system $\Phi_{i,j}$ of rank 2 (or 1 if $i=j$). Moreover, $Y'=\ker (\qa_i)\cap \ker  (\qa_j)\cap Y$ is cotorsion free in Y; let $Y''$ be a supplementary module containing $\qa_i^\vee$ and $ \qa_j^\vee$. $Y''$ is a lattice (of rank 2 or 1) between the lattices $Q_{i,j}^\vee$ of coroots and $P_{i,j}^\vee$ of coweights, associated to $\Phi_{i,j}$.

Any $\ql\in Y$ may be written $\ql = \ql'+\ql''$ with $\ql'\in Y'$ and $\ql''\in Y''$. By (4), $L_i (Z^{\ql'})=Z^{\ql'}H_i$ and $L_j(Z^{\ql'})=Z^{\ql'}H_j$. So we have to prove $(*)$ for $\ql=\ql''\in Y''$. We shall do it by comparing with some Macdonald's results.

b) In  \cite{Ma03} Macdonald builds affine Hecke algebras $\cal{H}(W(R, L')) $ over $\R$, associated to any finite irreducible root system $R$ and any lattice $L'$ between the lattices of coroots and coweights; more precisely this algebra is associated to the extended affine Weyl group $ W( R,L') =W(R) \ltimes L'$. It is defined by generators and relations, but it is proven that it is endowed with a basis $(Y^\ql T(w))_{\ql\in L', w\in W(R)}$  [\lc; 4.2.7] and satisfies relations analogous to (1), (2), (3), (4) as above. There are parameters $(\qt_i)_{i\in I}$ and $\qt_0$ which are reals (but may be algebraically independent over $\Q$, so may be considered as indeterminates) and satisfy $\qt_i= \qt_j$ if  $\qa_i(\qa_j^\vee)=\qa_j(\qa_i^\vee)=-1$. The relation (4) is satisfied with $\qs_i=\qt_i$ and $\qs'_i=\qt_i$ when $\qa_i(L')=\Z$, $\qs'_i=\qt_0$ when $\qa_i(L')=2\Z$.

c) In the case $R=\Phi_{i,j}$, irreducible, $L'=Y''$,  we may choose $\qt_i, \qt_j $ and $\qt_0$ such that the relations (4) are the same, for us and Macdonald: either  $\qa_i(\qa_j^\vee)=-1$ or $\qa_j(\qa_i^\vee)=-1$, so $\qt_0= \qs_i'$ or $\qt_0= \qs_j'$.  In particular $R_1$ may be identified to a subring of $\R$. The operators $L_i$ and $L_j$ of both theories coincide on the elements $Z^\ql H_v$ (identified with $Y^\ql T(v)$ in Macdonald's work) for $\ql\in L'=Y''$ and $v\in \langle r_i, r_j\rangle$. So $(*)$ is satisfied as $\cal{H}(W(R, L')) $ is an associative algebra.

d) So, if $H_w*Z^\qm=\sum_{k} a_k Z^{\qn_k}H_{u_k}$, with $a_k\in R_1, \qn_k\in Y, u_k\in W^v$, we define the product of $Z^\ql H_w$ and $Z^\qm H_v$ by: $(Z^\ql H_w)*(Z^\qm H_v)=\sum_{k} a_k Z^{\ql +\qn_k}*(H_{u_k}*H_v)$. We get a distributive multiplication on $^{BL}\mathcal H_{R_1}$  with unit $H_1$.

3) Associativity.

a) Using the associativity in $\mathcal H_{R_1}(Y)$ and $\mathcal H_{R_1}(W^v)$ and the formula 2.d above, it is clear that, for any $\ql \in Y, w\in W^v, E_1 ,E_2\in {^{BL}\mathcal H}_{R_1}$ , we have:

(R1) \quad $Z^\ql *(E_1 *E_2)= (Z^\ql *E_1 )*E_2,$

(R2) \quad $E_1 *(E_2*H_w)= (E_1 *E_2)*H_w.$

\parni We need also to prove (for $\ql_1, \ql_2\in Y, w, w_1, w_2\in W^v,E\in  {^{BL}\mathcal H}_{R_1}$),

(A) \quad $H_w*(Z^{\ql_1}*Z^{\ql_2})=(H_w*Z^{\ql_1})*Z^{\ql_2}$,

(B) \quad $H_{w_1}*(H_{w_2}*E)=(H_{w_1}*H_{w_2})*E$.

\parni Then the general associativity will follow : using (R1), (R2), (A), (B) and the formula 2d for the product,  it is not too difficult (and left to the reader) to prove that :

$$
\begin{array}{rcl}
(Z^{\ql_1}H_{w_1})*\big ( (Z^{\ql_2}H_{w_2})*(Z^{\ql_3}H_{w_3})\big) & = & Z^{\ql_1}*(H_{w_1}*\big ( (Z^{\ql_2}H_{w_2})*Z^{\ql_3})\big )*H_{w_3}\\
& = & Z^{\ql_1}*\big ((H_{w_1}*Z^{\ql_2})*(H_{w_2}*Z^{\ql_3})\big )*H_{w_3}\\
& = & Z^{\ql_1}*\big((H_{w_1}*(Z^{\ql_2}H_{w_2}))*Z^{\ql_3}\big )*H_{w_3}\\
& = & \big ((Z^{\ql_1}H_{w_1})*(Z^{\ql_2}H_{w_2})\big)*(Z^{\ql_3}H_{w_3}).\\
\end{array}
$$

b) Proof of (B). This condition is equivalent to the fact that the left multiplication by $\mathcal H_{R_1}(W^v)$ on $ {^{BL}\mathcal H}_{R_1}$ is an action. But the associative algebra  $\mathcal H_{R_1}(W^v)$  is generated by the $H_i$ with relations the braid relations and $H_i^2=(\qs_i-\qs_i^{-1})H_i+H_1$. As $L_i$ is the left multiplication by $H_i$, we have (B) if, and only if, these $L_i$ satisfy the relation $(*)$ in 2.a and

$(**) \quad L_i(L_i (Z^\ql H_v))=(\qs_i-\qs_i^{-1})L_i (Z^\ql H_v)+ Z^\ql H_v.$

As in 2b, we reduce the verification of $(**)$ to the case $v=1$ and $\ql\in Y''$ (associated to $i=j$) i.e. $\ql \in Y''=\Q\qa_i^\vee\cap Y$. Then we look at Macdonald's construction of
$\cal{H}(W(\{\pm \qa_i\}, Y'')) $ with $\qt_i=\qs_i$, $\qt_0=\qs'_i$. We conclude, as in 2.c that $(**)$ is satisfied.

\medskip
c) The proof of (A) is by induction on $\ell (w).$

If $w=r_i$, we have:

\noindent $(H_i*Z^{\ql_1})*Z^{\ql_2}=(Z^{r_i(\ql_1)}H_i)*Z^{\ql_2}+(b(\qs_i, \qs'_i; Z^{-\qa_i^\vee}) (Z^{\ql_1}-Z^{r_i(\ql_1)}))*Z^{\ql_2}\\
= Z^{r_i(\ql_1)}*(Z^{r_i(\ql_2)}H_i+(b(\qs_i, \qs'_i; Z^{-\qa_i^\vee}) (Z^{\ql_2}-Z^{r_i(\ql_2)}))+b(\qs_i, \qs'_i; Z^{-\qa_i^\vee}) (Z^{\ql_1+\ql_2}-Z^{r_i(\ql_1)+\ql_2})\\
= Z^{r_i(\ql_1+\ql_2)}H_i+b(\qs_i, \qs'_i; Z^{-\qa_i^\vee}) (Z^{r_i(\ql_1)+\ql_2}-Z^{r_i(\ql_1)+r_i(\ql_2)})+b(\qs_i, \qs'_i; Z^{-\qa_i^\vee}) (Z^{\ql_1+\ql_2}-Z^{r_i(\ql_1)+\ql_2})\\
= Z^{r_i(\ql_1+\ql_2)}H_i+b(\qs_i, \qs'_i; Z^{-\qa_i^\vee}) (Z^{\ql_1+\ql_2}-Z^{r_i(\ql_1+\ql_2)})\\
= H_i*(Z^{\ql_1}*Z^{\ql_2})$

If the result is known when $\ell(w)=n$. Let us consider $w=w'r_i$ with $\ell(w)=n+1$ and $\ell(w')=n$, then

\noindent $H_w*(Z^{\ql_1}*Z^{\ql_2})=H_{w'}*(H_i*Z^{\ql_1+\ql_2}) $ (left multiplication by $\mathcal H_{R_1}(W^v)$  is an action)

\noindent $ {\hskip 7em}= H_{w'}*((H_i*Z^{\ql_1})*Z^{\ql_2} )$ (case $\ell (w)=1$)

\noindent $  {\hskip 7em}= H_{w'}*\big ((Z^{r_i(\ql_1)}H_i)*Z^{\ql_2}+(b(\qs_i, \qs'_i; Z^{-\qa_i^\vee}) (Z^{\ql_1}-Z^{r_i(\ql_1)}))*Z^{\ql_2} \big ).$

On the other hand, we have

\noindent $(H_w*Z^{\ql_1})*Z^{\ql_2}=(H_{w'}*(H_i*Z^{\ql_1}))*Z^{\ql_2}\\
 {\hskip 7em}=\big (H_{w'}*(Z^{r_i(\ql_1)}H_i+b(\qs_i, \qs'_i;  Z^{-\qa_i^\vee}) (Z^{\ql_1}-Z^{r_i(\ql_1)})\big )*Z^{\ql_2}\\
{\hskip 7em}=\big (H_{w'}*(Z^{r_i(\ql_1)}H_i)\big )*Z^{\ql_2}+\big (H_{w'}*(b(\qs_i, \qs'_i; Z^{-\qa_i^\vee}) (Z^{\ql_1}-Z^{r_i(\ql_1)}))\big )*Z^{\ql_2}.$

The second term of the right hand side is a $R_1$-linear combination of $(H_{w'}*Z^{\ql_1+k\qa_i^\vee})*Z^{\ql_2}$ and we see by induction that it is the same as $H_{w'}*((b(\qs_i, \qs'_i;  Z^{-\qa_i^\vee}) (Z^{\ql_1}-Z^{r_i(\ql_1)}))*Z^{\ql_2} )$ in $H_w*(Z^{\ql_1}*Z^{\ql_2})$.

In the first term, $(H_{w'}*(Z^{r_i(\ql_1)}H_i))*Z^{\ql_2}=((H_{w'}*Z^{r_i(\ql_1)})*H_i))*Z^{\ql_2}$, we can write  $H_{w'}*Z^{r_i(\ql_1)}=\sum_k c_kZ^{\ql_k}H_{w_k}$ and we will use later in the same way $H_i*Z^{\ql_2}=\sum_h a_hZ^{\qm_h}H_{v_h}$ with $c_k, a_h\in R_1, \ql_k, \qm_h\in Y$ and $w_k, v_h\in W^v$. So, we have :

$((\sum_k c_kZ^{\ql_k}H_{w_k})*H_i)*Z^{\ql_2}= (\sum_k c_k(Z^{\ql_k}*(H_{w_k}*H_i)))*Z^{\ql_2}$\qquad (by (R2))

 $  {\hskip 12em}=\sum_k c_k Z^{\ql_k}*((H_{w_k}*H_i)*Z^{\ql_2})$\qquad (by formula 2d)

 $  {\hskip 12em}=\sum_k c_k Z^{\ql_k}*(H_{w_k}*(H_i*Z^{\ql_2}))$\qquad (by (B))

  $  {\hskip 12em}=\sum_k c_k (Z^{\ql_k}*H_{w_k})*(H_i*Z^{\ql_2})$ \qquad (by (R1))

    $  {\hskip 12em}=\sum_k c_k (Z^{\ql_k}*H_{w_k})*(\sum_h a_hZ^{\qm_h}H_{v_h})$

      $  {\hskip 12em}=\sum_{k,h} c_k a_h(Z^{\ql_k}*H_{w_k})*(Z^{\qm_h}*H_{v_h})$

            $  {\hskip 12em}=\sum_{k,h} c_k a_h(((Z^{\ql_k}*H_{w_k})*Z^{\qm_h})*H_{v_h})$ \quad (by  (R2))

     $  {\hskip 12em}=\sum_{h} a_h(((H_{w'}*Z^{r_i(\ql_1)})*Z^{\qm_h})*H_{v_h})$

       $  {\hskip 12em}=\sum_{h} a_h((H_{w'}*(Z^{r_i(\ql_1)}*Z^{\qm_h}))*H_{v_h})$\qquad (by induction)

 $  {\hskip 12em}=\sum_{h} a_hH_{w'}*((Z^{r_i(\ql_1)}*Z^{\qm_h})*H_{v_h})$ \qquad (by (R2))

  $  {\hskip 12em}=H_{w'}*(Z^{r_i(\ql_1)}*(H_i*Z^{\ql_2}))$.  \qquad (by (R1))

   This corresponds to the term $ H_{w'}*((Z^{r_i(\ql_1)}H_i)*Z^{\ql_2})$ in $H_w*(Z^{\ql_1}*Z^{\ql_2})$ so we obtain the equality when $\ell(w)=n+1$.   \qed

   \subsection{Change of scalars} \label{5.2}

\parni{\bf 1)}  Let us suppose that we are given a morphism  $ \varphi$ from $R_1$ to a ring $R$, then we are able to consider, by extension of scalars,  $^{BL}\mathcal H_{R}=R{\otimes_{R_1}} ^{BL}\mathcal H_{R_1}$ as  an $R$-associative algebra. The family $(Z^\ql H_w)_{\ql\in Y, w\in W^v}$ is still a basis of the $R$-module $^{BL}\mathcal H_{R}$.

\parni{\bf 2)}  In order to consider elements similar to the $X^\ql$ of section 4, we are going to define a ring $R_3$ containing $R_1$ such that there exists a group homomorphism $\qd^{1/2}:Y\to {R_3^\times}$ with $\qd(\ql)=\qd^{1/2}(\ql)^2$ for any $\ql\in Q^\vee$ and $\qd^{1/2}(\qa_i^\vee)=\qs_i.\qs'_i$.

Since $Q^\vee$ is a submodule of the free $\Z$-module $Y$, by the elementary divisor theorem, if we denote $m$ the biggest elementary divisor, then for any $\qm\in Y\cap (Q^\vee{\otimes_{\Z}}\R)$, we have $m\qm\in Q^\vee$. Let us consider the ring $R_3=\Z[({\qt_i}^{\pm 1}, {\qt'_i}^{\pm 1})_{i \in I}]$ (with $\qt_i,\qt_i'$ satisfying conditions similar to those of \ref{5.0}) and the identification of  $R_1$ as a subring of $R_3$ given by $\qt_i^m=\qs_i$ and ${\qt'_i}^m=\qs'_i$.
Then, for $\ql \in Y $ we have $m\ql=\sum_{i\in I} a_i \qa_i^\vee +\ql_0$ with the $a_i\in \Z$ and  $\ql_0\notin Q^\vee{\otimes_{\Z}}\R$ and we can define $\qd^{1/2}(\ql)=\Pi_{i\in I} (\qt_i\qt'_i)^{a_i}$ and obtain a group homomorphism from $Y$ to $R_3$, with the wanted properties.

In  $^{BL}\mathcal H_{R_3}$, let us consider $X^\ql= \qd^{1/2}(\ql)Z^\ql$ for $\ql\in Y$ and $T_i=\sigma_i H_i= (\qt_i)^m H_i$. It's easy to see that $T_w=T_{i_1}*T_{i_2}*\cdots*T_{i_n}$ is independent of the choice of a reduced decomposition $r_{i_1}r_{i_2}\cdots r_{i_n}$ of $w$.
It is clear that the family $(X^\ql* T_w)_{\ql\in Y, w\in W^v} $ is a new basis of the $R_3$-module $^{BL}\mathcal H_{R_3}$.

\parni{\bf 3)}  We can give new formulas to define $*$ in terms of these generators.
The relation (4) of the definition of $^{BL}\mathcal H_{R_3}$ can be written as previously :

\parni If $\qa_i(\ql) \geq 0$, then

\parni (BLZ$+$)\quad  $ H_i*Z^\ql= Z^{r_i(\ql)}*H_i+\displaystyle{\sum_{k \, even, \,k=0}^{\qa_i(\ql)-1}  ({\qs_i}-{\qs_i}^{-1}){ Z^{\ql-k\qa_i^\vee}}}+\displaystyle{\sum_{k\, odd, \, k=0}^{\qa_i(\ql)-1} ({\qs'_i}-{\qs'_i}\,^{-1}){ Z^{\ql-k\qa_i^\vee}}}$,

\parni If $\qa_i(\ql)<0$, then

\parni  (BLZ$-$) \quad  $ H_i*Z^\ql= Z^{r_i(\ql)}*H_i-\displaystyle{\sum_{k\, even, \,k=2}^{-\qa_i(\ql)} ({\qs_i}-{\qs_i}^{-1}){ Z^{\ql +k\qa_i^\vee}}}-\displaystyle{\sum_{k \, odd, \, k=1}^{-\qa_i(\ql)} ({\qs'_i}-{\qs'_i}\,^{-1}){ Z^{\ql+k\qa_i^\vee}}}$.

With the same arguments as in  \ref{4.7}, these relations (after changing the variables and writing $ (\qs_i^2)^{* n}=\qs_i^2{\qs'_i}^2\qs_i^2{\qs'_i}^2 \cdots $ with $n$ terms in this product) become :

\noindent (BLX$+$) If $\qa_i(\ql) \geq 0$, then  $T_i*X^\ql= (\qs_i^2)^{* (\qa_{i}(\ql))}X^{r_i(\ql)}*T_i+\displaystyle{\sum_{h=0}^{\qa_i(\ql)-1} ((\qs_i^2)^{* (h+1)}-(\qs_i^2)^{* (h)}) X^{\ql-h\qa_i^\vee}},$

\noindent (BLX$-$) If $\qa_i(\ql)< 0$, then, 

\noindent $T_i*X^\ql= {1\over (\qs_i^2)^{* (-\qa_{i}(\ql))}}X^{r_i(\ql)}*T_i- {1\over (\qs_i^2)^{* (-\qa_{i}(\ql))}}\displaystyle{\sum_{h=\qa_i(\ql)}^{-1}} \Bigl{(}{(\qs_i^2)}^{* (-\qa_{i}(\ql)+h+1)}-{q_i}^{*{(-\qa_{i}(\ql)}+ h)}\Bigr {)} X^{\ql-h\qa_i^\vee}$.

   \smallskip
The other formulas give easily:
\goodbreak
\parni (2') $ \forall i\in I \quad \forall w\in W^v \qquad T_i*T_w= T_{r_iw}\, $ if $\,\ell (r_iw)>\ell (w)$

$\qquad\qquad\qquad\qquad\qquad\qquad\quad\,\,=(\qs_i^2 -1)T_w+\qs_i^2 T_{r_iw}\,$ if $\, \ell (r_iw)<\ell (w)$,

   \smallskip
 \parni (3') $ \forall \ql\in Y \quad  \forall \qm\in Y \qquad X^\ql*X^\qm= X^{\ql+\qm}$.

 In all these relations, we can see that the coefficients are in the subring $R_2=\Z[(\qs_i^{\pm 2},{ \qs'}_i^{\pm 2})_{i\in I}] $ of $R_1$. So, if we consider $^{BLX}\mathcal H_{R_2}$ the $R_2-$submodule with basis $(X^\ql* T_w)_{\ql\in Y, w\in W^v} $, the multiplication $*$ gives it a structure of associative unitary algebra over $R_2$.

 \subsection{The positive Bernstein-Lusztig-Hecke algebra} \label{5.3}

If we consider in $^{BLX}\mathcal H_{R_2}$, the submodule with basis $(X^\ql* T_w)_{\ql\in Y^+, w\in W^v} $,  it is stable by  multiplication  $*$ (in (BLX$+$) and (BLX$-$) if $\ql \in Y^+$ all the $\ql\pm h\qa_i^\vee$ written are also in $Y^+$).
  We denote by $^{BL}{\mathcal H}^+_{R_2}$ this $R_2-$subalgebra of $^{BLX}\mathcal H_{R_2}$.
   Actually, we can define  such  positive Hecke subalgebras inside all algebras in \ref{5.2}.

 Like before, if we are given a morphism  $ \varphi$ from $R_2$ to a ring $R$, we are able to consider, by extension of scalars,  $^{BL}\mathcal H^+_{R}=R{\otimes_{R_2}} ^{BL}\mathcal H^+_{R_2}$.
 Let us consider the ring $R$ of the section 4 (such that $\Z\subset R$ and all $q_i, q'_i$ are invertible in $R$), we can construct a morphism $\phi $ from $R_2$ to $R$ by $\phi (\qs_i^{ 2}  )=q_i$ and $\phi ({\qs'}_i^{ 2}  )=q'_i$. So, we obtain an algebra $^{BL}\mathcal H^+_{R}$ with basis $(X^\ql* T_w)_{\ql\in Y^+, w\in W^v} $ and the same relations as in   $^{I}\mathcal H_{R}$. So :

   \begin{prop*} Over $R$, the Iwahori-Hecke algebra $^{I}\mathcal H_{R}$ and the positive Bernstein-Lusztig-Hecke algebra  $^{BL}\mathcal H^+_{R}$ are isomorphic.
 \end{prop*}

 \begin{rema*} $^{BLX}\mathcal H_{R}$ is a ring of quotients of $^{BL}\mathcal H^+_{R}\simeq {^{I}\mathcal H_{R}}$, as we added, in it, inverses of the $X^\ql=T_\ql$ for $\ql\in Y^{++}$. Actually, from 5.2, 5.4 and similar results, one may prove that $S = \{T_\lambda, \lambda\in Y^{++}\}$ satisfies the right and left \?Ore condition and that the map from $^{BL}\mathcal H^+_{R}$ to the corresponding quotient ring is injective (see e.g. \cite{McCR01} 2.1.6 and 2.1.12).
 \end{rema*}

 \subsection{Structure constants} \label{5.7}

 \par From Proposition \ref{5.3}, we get that the structure constants of the convolution product $*$ of $^{I}\mathcal H_{R}$, in the basis $(X^\ql* T_w)_{\ql\in Y^+, w\in W^v} $, are Laurent polynomials in the parameters $q_i,q_i'$, with coefficients in $\Z$, depending only on $\A$ and $W$.
 By Remark \ref{4.5}.a, we get the same result for the structure constants in the basis $(T_\ql* T_w)_{\ql\in Y^+, w\in W^v} $ and then still the same result for the structure constants $a_{\mathbf w, \mathbf v}^{\mathbf u}$ in the basis $(T_{\mathbf w})_{\mathbf w\in W^+}$ (by \ref{3.5}).

 \par This last result is not as precise as the one expected in the conjecture of {\S{}} \ref{s2}.
 But there is at least one case where we can prove it:

 \begin{rema*} Suppose $\SHI$ is the hovel associated to a split Kac-Moody group $G$ over a local field $\shk$, \cf \cite[{\S{}}  3]{GR13}.
 Then all parameters $q_i,q_i'$ are equal to the cardinality $q$ of the residue field; moreover we know that each $a_{\mathbf w, \mathbf v}^{\mathbf u}$ is an integer and a Laurent polynomial in $q$, with coefficients in $\Z$, depending only on $\A$ and $W$.
 But, as $G$ is split, the same thing is true (without changing $\A$ and $W$) for all unramified extensions of the field $\shk$, hence for infinitely many $q$.
 So the Laurent polynomial $a_{\mathbf w, \mathbf v}^{\mathbf u}$ is an integer for infinitely many integral values of the variable $q$: it has to be a true polynomial.
 \end{rema*}

\section{Extended affine cases and DAHAs}\label{s7}

In this section, we define the extended Iwahori-Hecke algebras and explore their relationship with the Double Affine Hecke Algebras introduced by Cherednik.

 \subsection{Extended groups of automorphisms} \label{7.1}
 \par  We may  consider a group $\widetilde G$ containing the group $G$ of \ref{1.3} and an extension to $\widetilde G$ of the action of $G$ on $\SHI$.
  We assume that $\widetilde G$ permutes the apartments and induces isomorphisms between them, hence $\widetilde G$ is equal to $G.\widetilde N$ where $\widetilde N\supset N$ is  the stabilizer  of $\A$ in $\widetilde G$.
  This group $\widetilde N$ has almost the same properties as the group $N$ described in \ref{1.3}.4 above. But we  assume now $\widetilde W=\qn(\widetilde N)\subset Aut(\A)$ only positive for its action on the vectorial faces; this means that the associated linear map $\vect w$ of any $w\in \widetilde W$ is in $Aut^+(\A^v)$.
  We assume moreover  that $\widetilde W$ may be written $\widetilde W=\widetilde W^v\ltimes Y$, where $\widetilde W^v$ fixes the origin $0$ of $\mathbb A$ and $Y$ is the same group of translations as for $G$ (\cf \ref{1.3}.4 above).
  In particular, $\widetilde W^v$ is isomorphic to the group $\{\vect w\mid w\in\widetilde W\}$ and may be written $\widetilde W^v=\QO\ltimes W^v$ (\cf \ref{1.1} above); moreover $\widetilde W=\QO\ltimes W$, where $\QO$ is the stabilizer of $C_f^v$ in $W^v$.
  Finally, we assume that $G$ contains the fixer $Ker\qn$ of $\A$ in $\widetilde G$; so that $G\lhd \widetilde G$ is the subgroup of all ``vectorially-Weyl'' automorphisms in $\widetilde G$ and $\widetilde G/G\simeq\QO$.

\par As $\widetilde W$ is positive, $\widetilde G$ preserves the preorder ${\leq}$ on $\SHI$.
So $\widetilde G^+=\{g\in\widetilde G\mid 0{\leq}g.0\}$ is a semi-group with $\widetilde G^+\cap G=G^+$.
And $\widetilde W^+=\QO\ltimes W^+=\widetilde W^v\ltimes Y^+\subset \widetilde W$ is also a semigroup, with $\widetilde W^+\cap W=W^+$.

 \subsection{Examples: Kac-Moody and loop groups} \label{7.2}

 \par 1) One considers a field $\shk$ complete for a normalized, discrete valuation with a finite residue field (of cardinality $q$).
 If $\g G$ is an almost split Kac-Moody group-scheme over $\shk$, then the Kac-Moody group $G=\g G(\shk)$ acts on an affine ordered hovel $\SHI$, with the properties described in \ref{1.3}.
See  \cite{R12}, \cite[{{\S}} 3]{GR13} in the split case (where all $q_i,q'_i$ are equal to $q$) and \cite{Ch10}, \cite{Ch11} or \cite{R13} in general.

\par 2) Let $\g G_0$ be a simply connected, almost simple, split, semi-simple algebraic group of rank $r$ over $\shk$.
Its fundamental maximal torus $\g T_0$ is $Q_0^\vee\otimes_\Z\g{Mult}$, where $Q_0^\vee$ (resp. $P_0^\vee$) is the coroot lattice (resp. coweight lattice) of the root system $\QF_0\subset V_0^*$ with Weyl group $W_0^v$.

Then some central extension by $\shk^\times$ of (a subgroup of) the loop group $\g G_0(\shk[t,t^{-1}])\rtimes\shk^\times$ (where $x\in\shk^\times$ acts on $\g G_0(\shk[t,t^{-1}])$ via $t\mapsto xt$) is $G=\g G(\shk)$ for the most popular example $\g G$ of an untwisted, affine, split, Kac-Moody group-scheme over $\shk$.
Its fundamental, maximal torus $\g T$ is $\g{Mult}\times\g T_0\times\g{Mult}=Y\otimes_\Z\g{Mult}$, with cocharacter group $Y=\Z\g c\oplus Q_0^\vee\oplus\Z d$, where $\g c$ is the canonical central element and $d$ the scaling element.

The set $\QF$ of real roots is $\{\qa_0+n\qd\mid\qa_0\in\QF_0,n\in\Z\}$ in the dual $V^*$ of $V=Y\otimes_\Z\R=\R\g c\oplus V_0\oplus\R d$, where $\qd(a\g c+v_0+bd)=b$ and $\qa_0(a\g c+v_0+bd)=\qa_0(v_0)$.
The corresponding Weyl group $W^v$ is actually the affine Weyl group $W_0^a=W_0^v\ltimes Q_0^\vee$ acting linearly on $V$; its action on the hyperplane $d+V_0$ of $V/\R\g c$ is affine: $W_0^v$ acts linearly on $V_0$ and $Q_0^\vee$ acts by translations.
The group $G$ is generated by $T=\g T(\shk)$ and root groups $U_\qa\simeq \shk=\g{Add}(\shk)$ for $\qa\in\QF$; if $\qa=\qa_0+n\qd$, then $U_\qa=\g U_{\qa_0}(t^n.\shk)$.

\par The fundamental apartment $\A$ of the associated hovel is as described in \ref{1.2} with $W=W^v\ltimes Y$ containing the affine Weyl group $W^a=W^v\ltimes Q^\vee$, with $Q^\vee=\Z\g c\oplus Q_0^\vee$.

\par This is the situation considered in \cite{BrKP14}.
We saw in \cite[Rem. 3.4]{GR13} that our group $K$ is the same as the $K$ of \cite{BrKP14}.
It is clear that the Iwahori group $I$ of \lc is included in our group $K_I$.
But from \ref{1.3}.2 and [\lc 3.1.2], we get two Bruhat decompositions $K=\bigsqcup_{w\in W^v}\, K_I.w.K_I=\bigsqcup_{w\in W^v}\, I.w.I$.
So $K_I=I$ and, in this case, our results are the same as those of \lc

\par 3) Let us consider a central schematical quotient $\g G_{00}$ of  $\g G_{0}$.
It is determined by the cocharacter group $Y_{00}$ of its fundamental torus $\g T_{00}$: $Q_0^\vee\subset Y_{00}\subset P_0^\vee$ and $\g T_{00}=Y_{00}\otimes_\Z\g{Mult}$.
The root system $\QF_0\subset V_0^*$ and the Weyl group $W_0^v\subset GL(V_0)$ are the same.

\par We get a more general untwisted, affine, split Kac-Moody scheme $\g G_1$ by "amalgamating" $\g G$ and the $\shk-$split torus $\g T_1=Y_1\otimes_\Z\g{Mult}$ (with $Y_1=\Z\g c\oplus Y_{00}\oplus\Z d$) along $\g T$.
A little more precisely the Kac-Moody group $G_1=\g G_1(\shk)$ is a quotient of the free product of $G$ and $T_{00}=\g T_{00}(\shk)=Y_{00}\otimes_\Z\shk^\times$ by some relations; essentially $T_{00}$ normalizes $T$ and each $U_\qa$ (hence also $G$) and one identifies both copies of $T_{0}$,  \cf \cite[1.8]{R12}.
The new fundamental torus is $\g T_1$. We keep the same $V,\QF,W^v,\A$ and $\SHI$, but now $W_1=W^v\ltimes Y_1\supset W\supset W^a$.

\par 4) We may consider a central extension by $\shk^\times$ of (a subgroup of) the loop group $\g G_{00}(\shk[t,t^{-1}])\rtimes\shk^\times$.
We get thus an extended Kac-Moody group $\widetilde G_2$ (not among the Kac-Moody groups of \cite{T87} or \cite{R12}) which may also be described by amalgamation:
$\widetilde G$  is a quotient of the free product of $G$ and $Y_{00}\otimes_\Z\shk[t,t^{-1}]^*$ by relations similar to those above; in particular the conjugation by $\ql\otimes xt^n$ sends $U_{\qa_0+p\qd}$ to $U_{\qa_0+(p+n\qa(\ql))\qd}$.
The group $\widetilde G_2$ contains $G_1$ as a normal subgroup, its fundamental torus is $T_1=Y_1\otimes_\Z\shk^\times$, with normalizer $\widetilde N_2=N_{\widetilde G_2}(T_1)$ containing $Y_{00}\otimes_\Z\shk[t,t^{-1}]^*\supset Y_{00}\otimes_\Z t^\Z=:t^{Y_{00}}$.

\par The group $\widetilde G_2$ is generated by $t^{Y_{00}}$ and $G_1$ (which contains $N_1=N_2\cap G_1\supset t^{Q_{0}^\vee}$); in particular $\widetilde G_2/G_1\simeq Y_{00}/Q_0^\vee$.
We keep the same $V$ and $\QF$, but now the corresponding vectorial Weyl group is $\widetilde W_2^v=N_2/T_1=W_0^v\ltimes Y_{00}$.
 As in \ref{1.1}, we may also write $\widetilde W_2^v=\QO_2\ltimes W^v$, where $\QO_2$ is the stabilizer in $\widetilde W_2^v$ of $C^v_f$.
 It is well known that $\QO_2$ is a finite group isomorphic to $Y_{00}/Q_0^\vee$; it is isomorphic to its image in the permutation group of the affine Dynkin diagram of $\g G_{00}$ or  $\g G_{0}$ (indexed by $I$) and acts simply transitively on the special vertices of this diagram.

 \par It is not too difficult to extend to $\widetilde G_2$ the action of $G_1$ on the hovel $\SHI$.
 The group $\widetilde N_2$ is the stabilizer of $\A$; it acts through $\widetilde W_2=\widetilde W_2^v\ltimes Y_1\supset W\supset W^a$.
 We are exactly in the situation of \ref{7.1} with $(\widetilde G_2,G_1)$.

 \par 5) We may get new couples $(\widetilde G_j,G_j)$ satisfying \ref{7.1} for the same hovel $\SHI$:

\par We may enlarge $\widetilde G_2$ and $G_1$ by amalgamating them with $T_3=Y_3\otimes_\Z\shk^\times$ along $T_1$ (or with $T_{000}=Y_{000}\otimes_\Z\shk^\times$ along $T_{00}$), where $Y_{00}\subset Y_{000}\subset P_0^\vee$ and $Y_3=\Z.(1/m).\g c\oplus Y_{000}\oplus\Z d$, with $m\in\Z_{>0}$.
Then $\widetilde W_3^v=\widetilde W_2^v$, $\QO_3=\QO_2$, $\widetilde W_3=\widetilde W_2^v\ltimes Y_3$ and $G_3$ is still a Kac-Moody group with maximal torus $T_3$.

\par We may keep $G_1$ (or $G_3$) and take a semi-direct product of $\widetilde G_2$ (or $\widetilde G_3$) by a group $\QG$ of automorphisms of the Dynkin diagram of $\g G_0$, stabilizing $Y_{00}$ (or $Y_{00}$ and $Y_{000}$).
Then   $\widetilde W_4^v=\QG\ltimes \widetilde W_2^v$, $\QO_4=\QG\ltimes\QO_2$ and $\widetilde W_4=\widetilde W_4^v\ltimes Y_2$ (or $\widetilde W_4=\widetilde W_4^v\ltimes Y_3$).

\par 6) We may also add a split torus as direct factor to any of the preceding groups $\widetilde G_i$ or $G_i$, enlarge $\SHI$ by a trivial euclidean factor of the same dimension as the torus and add to $\widetilde W^v$ and $\QO$, as a direct factor, any automorphism group (possibly infinite) of this torus.

 \subsection{Marked chambers} \label{7.3}

 \par We come back to the general situation of \ref{7.1}. We want a set of "geometric objects" in $\SHI$ on which $\widetilde G$ acts with the Iwahori subgroup $K_I$ as one of the isotropy groups.

 \par 1) A {\it marked chamber} in the hovel $\SHI$ is the class of an isomorphism $\qf:\A\to A\in\sha$ sending the fundamental chamber $C_0^+$ to some local chamber $C_x$, modulo the equivalence $\qf_1\simeq\qf_2\iff \exists S\in C_0^+, \qf_1\vert_S=\qf_2\vert_S$. It is simply written $\qf:C_0^+\to C_x$;
 this does not depend on $A$.

\par The group  $\widetilde G$ permutes the marked chambers; for $g\in\widetilde G$ and $\qf$ as above, $g.\qf=\qf$ if, and only if, $g$ fixes (pointwise) $C_x$.
In particular the isotropy group in $\widetilde G$ of $\widetilde C_0^+=Id:C_0^+\to C_0^+\subset\A\subset\SHI$ is $K_I\subset G$.

\par A local chamber of type $0$, $C_x\in\SHC_0^+$ determines a unique marked chamber $\widetilde C_x^0:C_0^+\to C_x$ (said normalized) which is the restriction of some $\qf\in Isom^W_\R(\A,A)$ (\cf \ref{1.13}).
 These normalized marked chambers are permuted transitively by $G$.

 \par 2) A marked chamber is said {\it of type $0$} if it is in the orbit under $\widetilde G$ of any of  those $\widetilde C_x^0$.
 So the set $\widetilde\SHC_0^+$ of marked chambers of type $0$ is $\widetilde G/K_I$.

\par By hypothesis $\widetilde G$ may be written $G.\widetilde \QO$, where  $\widetilde \QO=\qn^{-1}(\QO)\subset\widetilde N$ stabilizes $C_0^+$ (considered as in $\SHI$) and induces $\QO$ on it.
So $\widetilde\SHC_0^+=\{\widetilde C_x=\widetilde C_x^0\circ\qo^{-1} \mid C_x\in \SHC_0^+,\qo\in\QO\}$.

 \subsection{$\widetilde W-$distance} \label{7.4}

\par 1) Let $\widetilde C_x:C_0^+\to C_x$, $\widetilde C_y:C_0^+\to C_y$  be in $\widetilde\SHC_0^+$ with $x{\leq}y$.
There is an apartment $A$ containing $C_x$ and $C_y$ so $\widetilde C_x,\widetilde C_y$ may be extended to $\qf,\psi\in Isom(\A,A)$.
 We "identify" $(\A,C_0^+)$ with $(A,C_x)$ via $\qf$.
 Then $\qf^{-1}(y){\geq}0$ and, as $\widetilde C_x,\widetilde C_y$ are in a same orbit of $\widetilde G$, there is $\widetilde {\mathbf w}\in \widetilde  W^+$ such that $\psi=\qf\circ \widetilde {\mathbf w}$.
 This $\widetilde {\mathbf w}$ does not depend on the choice of $A$ by \ref{1.12}.c.

\par We define the {\it $\widetilde W-$distance} between the marked chambers $\widetilde C_x$ and $\widetilde C_y$ as this unique element: $d^W(\widetilde C_x,\widetilde C_y)=\widetilde {\mathbf w}\in \widetilde  W^+$. So we get a $\widetilde G-$invariant map
$$
d^W:\widetilde\SHC_0^+\times_{\leq}\widetilde\SHC_0^+=\{ (\widetilde C_x,\widetilde C_y)\in \widetilde\SHC_0^+\times\widetilde\SHC_0^+\mid x{\leq}y\}\to \widetilde  W^+.
$$

\par 2) For $(C_x,C_y)\in \SHC_0^+\times_{\leq}\SHC_0^+$, we have $d^W(\widetilde C_x^0,\widetilde C_y^0)=d^W(C_x,C_y)$ and, more generally, for
$\qo_x,\qo_y\in\QO$, we have $(\widetilde C_x^0\circ\qo_x^{-1},\widetilde C_y^0\circ\qo_y^{-1})\in \widetilde\SHC_0^+\times_{\leq}\widetilde\SHC_0^+$ and $d^W(\widetilde C_x^0\circ\qo_x^{-1},\widetilde C_y^0\circ\qo_y^{-1})=\qo_x. d^W(C_x,C_y).\qo_y^{-1}\in\widetilde W^+$.
 For $(\widetilde C_x,\widetilde C_y)\in \widetilde\SHC_0^+\times_{\leq}\widetilde\SHC_0^+$ and $\qo_x,\qo_y\in\QO$, we have also  $d^W(\widetilde C_x\circ\qo_x^{-1},\widetilde C_y\circ\qo_y^{-1})=\qo_x. d^W(\widetilde C_x,\widetilde C_y).\qo_y^{-1}\in\widetilde W^+$.

 \medskip
 \par We deduce from this some interesting consequences:

\par 3) If $\widetilde C_x,\widetilde C_y,\widetilde C_z$, with $x{\leq}y{\leq}z$, are in a same apartment, we have the Chasles relation: $d^W(\widetilde C_x,\widetilde C_z)=d^W(\widetilde C_x,\widetilde C_y).d^W(\widetilde C_y,\widetilde C_z)$.

\par 4) For  $(\widetilde C_x,\widetilde C_y)\in \widetilde\SHC_0^+\times_{\leq}\widetilde\SHC_0^+$, if $\widetilde C_x$ (resp. $\widetilde C_y$) is normalized, then $d^W(\widetilde C_x,\widetilde C_y)\in W^+$ if, and only if, $\widetilde C_y$ (resp. $\widetilde C_x$) is normalized.

\par 5) For  $(\widetilde C_x,\widetilde C_y)\in \widetilde\SHC_0^+\times_{\leq}\widetilde\SHC_0^+$, then $d^W(\widetilde C_x,\widetilde C_y)=\qo\in\QO\iff \widetilde C_y=\widetilde C_x\circ\qo$; in particular $\widetilde C_y$ is uniquely determined by $\widetilde C_x$ and $\qo$, moreover $C_y=C_x$.

\par 6) If $(C_x,C_y)\in \SHC_0^+\times_{\leq}\SHC_0^+$ and $d^W(C_x,C_y)=r_i\in W^v$ (resp. $=\ql\in Y^+$) and $\qo\in\QO$, then $d^W(\widetilde C_x^0\circ\qo^{-1},\widetilde C_y^0\circ\qo^{-1})=\qo.r_i.\qo^{-1}=r_{\qo(i)}$ (resp. $=\qo(\ql)\in Y^+$), where we consider the action of $\QO$ on $I$ (resp. $Y$).

\par 7) When $\widetilde C_x=\widetilde C_0^+$ and $\widetilde C_y=g.\widetilde C_0^+$ (with $g\in \widetilde G^+$), then $d^W(\widetilde C_x,\widetilde C_y)$ is the only $\widetilde {\mathbf w}\in \widetilde W^+$ such that $g\in K_I.\widetilde {\mathbf w}.K_I$.
There is a Bruhat decomposition $\widetilde G^+=\bigsqcup_{\widetilde {\mathbf w}\in \widetilde W^+}\,K_I.\widetilde {\mathbf w}.K_I$.

\par The $\widetilde W-$distance classifies the orbits of $K_I$ on $\{ \widetilde C_y\in\widetilde \SHC_0^+\mid y{\geq}0\}$, hence also the orbits of $\widetilde G$ on $\widetilde\SHC_0^+\times_{\leq}\widetilde\SHC_0^+$.

 \subsection{The extended Iwahori-Hecke algebra} \label{7.5}

\par 1) We define this extended algebra for $\widetilde G$ as we did in {{\S}} \ref{s2} for $G$:

\par  To each $\widetilde {\mathbf w}\in \widetilde W^+$, we associate a function $T_{\widetilde {\mathbf w}}:\widetilde {\mathscr C}_0^+\times_\leq\widetilde {\mathscr C}_0^+\to R$ defined by
$$
T_{\widetilde {\mathbf w}}(\widetilde C,\widetilde C') =
\left\{
\begin{array}{l}
1 \quad\hbox{ if } d^W(\widetilde C,\widetilde C') = \widetilde {\mathbf w},\\
0 \quad\hbox{ otherwise.}
\end{array}
\right.
$$
And we  consider the following free $R-$module of functions $\widetilde {\mathscr C}_0^+\times_\leq\widetilde {\mathscr C}_0^+\to R$:
$$^I\widetilde \shh_R^\SHI = \{\qf = \sum_{\widetilde {\mathbf w}\in \widetilde W^+} a_{\widetilde {\mathbf w}}T_{\widetilde {\mathbf w}}\mid a_{\widetilde {\mathbf w}}\in R,\ a_{\widetilde {\mathbf w}} = 0 \hbox{ except for a finite number}\},
$$ We endow this $R-$module with the convolution product:

$$
(\qf*\psi)(\widetilde C_x,\widetilde C_y)=\sum_{\widetilde C_z}\,\qf(\widetilde C_x,\widetilde C_z)\psi(\widetilde C_z,\widetilde C_y).
$$ where $\widetilde C_z\in \widetilde {\mathscr C}^+_0$ is such that $x\leq z \leq y$.
This product is associative and $R-$bilinear.  We prove below that it is well defined.

\par As in {{\S}} \ref{s2}, we see easily that $^I\widetilde \shh_R^\SHI $ is the natural convolution algebra of the functions $\widetilde G^+\to R$, bi-invariant under $K_I$ and with finite support.

\par 2) For $\qo\in\QO$, $\widetilde {\mathbf w}\in \widetilde W^+$, the products $T_\qo*T_{\widetilde {\mathbf w}}$ and $T_{\widetilde {\mathbf w}}*T_\qo$ are well defined: actually $T_\qo*T_{\widetilde {\mathbf w}}=T_{\qo.\widetilde {\mathbf w}}$ and $T_{\widetilde {\mathbf w}}*T_\qo=T_{\widetilde {\mathbf w}.\qo}$, see \ref{7.4}.3 and \ref{7.4}.5.

\par 3) As the formula for $\qf*\psi$ is clearly $\widetilde G-$invariant, we may fix $\widetilde C_x$ normalized to calculate $\qf*\psi$.
 From \ref{7.4}.4, we deduce that, when $\mathbf w,\mathbf v\in W^+$, $T_{\mathbf w}*T_{\mathbf v}$ may be computed using only normalized marked chambers. So it is well defined and the same as in $^I \shh_R^\SHI $.

\par From 2) we deduce now that the convolution product is well defined in $^I\widetilde \shh_R^\SHI $:

\begin{prop*} For any ring $R$, $^I\widetilde \shh_R^\SHI $ is an algebra; it contains $^I\shh_R^\SHI $ as a subalgebra.
\end{prop*}

 \begin{defi*}The algebra $^I\widetilde \shh_R^\SHI $ is the {\it extended Iwahori-Hecke algebra} associated to $\SHI$ and $\widetilde G$ with coefficients in $R$.
  \end{defi*}

 \subsection{Relations} \label{7.6}

 \par 1) From \ref{7.5} we see that $^I\widetilde \shh_R^\SHI $ contains the algebra $R[\QO]=\oplus_{\qo\in\QO}\,R.T_\qo$ of the group $\QO$. Moreover, as an $R-$module, $^I\widetilde \shh_R^\SHI $ is a tensor product, $^I\widetilde \shh_R^\SHI =R[\QO]\otimes_R {^I \shh_R^\SHI }$: we identify $T_{\qo. {\mathbf w}}=T_\qo*T_{ {\mathbf w}}$ and $T_\qo\otimes T_{ {\mathbf w}}$ for $\qo\in\QO$ and $\mathbf w\in W^+$.

 \par The multiplication in this tensor product is semi-direct:

 \par\qquad\qquad $(T_\qo\otimes T_{ {\mathbf w}}).(T_{\qo'}\otimes T_{ {\mathbf v}})=T_\qo* T_{ {\mathbf w}}*T_{\qo'}* T_{ {\mathbf v}}= T_{\qo. {\mathbf w}.\qo'}* T_{ {\mathbf v}}$

 \par\qquad\qquad $=T_{\qo.\qo'. {\mathbf w'}}* T_{ {\mathbf v}}=T_{\qo.\qo'}*T_{ {\mathbf w'}}* T_{ {\mathbf v}}=T_{\qo.\qo'}\otimes (T_{ {\mathbf w'}}*T_{\mathbf v})$

 \par where $\mathbf w'={\qo'}^{-1}.\mathbf w.\qo'=:{\qo'}^{-1}(\mathbf w)\in W^+$.

 \parni  In particular, we get the following relations among some elements:

 \par 2) For $\qo\in\QO$, $\mathbf w\in W^+$, $T_\qo*T_{\mathbf w}*T_\qo^{-1}=T_{\qo(\mathbf w)}$,

  \par\qquad if moreover $\mathbf w=r_i\in W^v$, $\qo(r_i)=r_{\qo(i)}$ hence $T_\qo*T_{i}*T_\qo^{-1}=T_{\qo(i)}$,

 \par\qquad if now $\mathbf w=\ql\in Y^+$, $T_\qo*T_{\ql}*T_\qo^{-1}=T_{\qo(\ql)}$, with $\qo(\ql)\in Y^+$.

\par 3) From \ref{4.5}.1 and 2) above, it is clear that $T_\qo*X^{\ql}*T_\qo^{-1}=X^{\qo(\ql)}$ if $\qo\in\QO$ and $\ql\in Y^+$ (as $\QO$ stabilizes $Y^{++}=Y\cap C^v_f$).

\par 4) As the action of $\QO$ on $\A$ is induced by automorphisms of $\SHI$, we have $q_i=q_{\qo(i)}$ and $q_i'=q_{\qo(i)}'$ for $\qo\in\QO$ and $i\in I$.
 We may also choose the homomorphism $\qd^{1/2}:Y\to R^*$ of \ref{4.7} invariant by $\QO$ (for $R$ great enough).
  So, for $\qo\in\QO$, $w,r_i\in W^v$ and $\ql\in Y$, we have:

 \par $T_\qo*H_w*T_\qo^{-1}=H_{\qo(w)}$\quad,\qquad $T_\qo*H_i*T_\qo^{-1}=H_{\qo(i)}$\qquad and \qquad$T_\qo*Z^{\ql}*T_\qo^{-1}=Z^{\qo(\ql)}$.

 \subsection{The extended Bernstein-Lusztig-Hecke algebra} \label{7.7}

Notations of 7.1 are still in use.
But we no longer assume the existence of a group $\widetilde G$ or $G$.
 The group $W=W^v\ltimes Y\lhd\widetilde W$ satisfies $\widetilde W=\QO\ltimes W$ and the conditions of {{\S}} \ref{s5}.

 \par We consider the ring $\widetilde R=\Z[{(\widetilde \qs_{i}}^{\pm 1}, ({\widetilde \qs}'_{i})^{\pm 1})_{ i\in I} ]$, where the indeterminates $\widetilde \qs_{i}, {\widetilde \qs'_{i}}$ satisfy the same relations as $\qs_i,\qs'_i$ in \ref{5.0} and the following additional relation (see \ref{7.6}.4 above):

 \par\qquad If $\qo(i)=j$ for some $\qo\in\QO$, then $\widetilde \qs_i=\widetilde \qs_j$ and $\widetilde \qs'_i=\widetilde \qs'_j$.

 \par We denote by $^{BL}\widetilde\shh_{\widetilde R}$ the free $\widetilde R-$module with basis $(T_\qo Z^\ql H_w)_{\qo\in\QO,\ql\in Y,w\in W^v}$ and write $H_w=T_1Z^0H_w$, $H_i=T_1Z^0H_i$, $Z^\ql=T_1Z^\ql H_1$ and $T_\qo=T_\qo Z^0H_1$.

 \begin{prop*} There exists a unique multiplication $*$ on $^{BL}\widetilde \shh_{\widetilde R} $ which makes it an associative, unitary $\widetilde R-$algebra with unity $H_1=T_1=Z^0$ and satisfies the conditions (1), (2), (3), (4) of Theorem \ref{5.1} plus the following:

 \parni (5) For $\qo,\qo'\in\QO$, $i\in I$ and $\ql\in Y$, $T_\qo*T_{\qo'}=T_{\qo.\qo'}$,  $T_\qo*T_{i}*T_\qo^{-1}=T_{\qo(i)}$, $T_\qo*T_{\ql}*T_\qo^{-1}=T_{\qo(\ql)}$.
 \end{prop*}

 \begin{proof} As $\widetilde R-$modules, $^{BL}\widetilde \shh_{\widetilde R} =\widetilde R[\QO]\otimes {^{BL} \shh_{\widetilde R}}$, where the homomorphism $R_1\to\widetilde R$ is given by $\qs_i\mapsto\widetilde \qs_i,\qs_i'\mapsto\widetilde \qs'_i$.
 Now the multiplication is classical on $\widetilde R[\QO]$, given by \ref{5.1} on $^{BL} \shh_{\widetilde R}$ and semi-direct for general elements.
 \end{proof}

 \begin{defi*} This $\widetilde R-$algebra $^{BL}\widetilde \shh_{\widetilde R} $ depends only on $\A$, $Y$ and $\QO$ (\ie on $\A$ and $\widetilde W$). We call it  the {\it extended Bernstein-Lusztig-Hecke algebra} associated to $\A$ and $\widetilde W$ with coefficients in $\widetilde R$.
 \end{defi*}

 \par As in \ref{5.3}, we may identify, up to an extension of scalars, a subalgebra $^{BL}\widetilde \shh_{\widetilde R} ^+$ of $^{BL}\widetilde \shh_{\widetilde R} $ with the extended Iwahori-Hecke algebra $^{I}\widetilde \shh_{R} ^\SHI$.

  \subsection{The affine case} \label{7.8}

  \par 1) We suppose now $(\A^v,W^v)$ affine.
  So there is a smallest positive imaginary root $\qd=\sum a_i\qa_i\in\QD^+_{im}\subset Q^+$ satisfying $\qd(\qa_i^\vee)=0, \forall i\in I$ and a canonical central element $\g c=\sum a_i^\vee\qa_i^\vee\in Q_+^\vee$ satisfying $\qa_i(\g c)=0, \forall i\in I$.
  In particular $\qd$ and $\g c$ are fixed by $W^v$ and $\widetilde W^v$.

\par As $\qd\in Q^+$, it takes integral values on $Y$.
For $n\in\Z$, we define $Y^n=\{\ql\in Y\mid\qd(\ql)=n\}$ which is stable under $W^v$ and $\widetilde W^v$.
We have $Y=\bigsqcup_{n\in\Z}\,Y^n$ and $Y^+=(\bigsqcup_{n>0}\,Y^n)\bigsqcup Y^0_c$, with $Y^0_c=Y^0\cap Y^+=Y\cap\Q\g c$.
We write $\ql_c=(1/m)\g c$ a generator of $Y^0_c$ (with $m\in\Z_{>0}$).
 As $\qd(Q^\vee)=0$, we have $\qd(\ql)=\qd(\qm)$ whenever $\qm{\leq}_{Q^\vee}\ql$ or $\qm{\leq}_{Q^\vee_\R}\ql$ in $Y$.

 \par 2) Considering \ref{PrFinite1} and \ref{4.5}.2, we get the following gradations of algebras (for a suitable $R$):
$$
{^{I} \shh_{ R}^{\SHI }}=\bigoplus_{n{\geq}0}\,{^{I} \shh_{ R}^{\SHI n}},
$$ where ${^{I} \shh_{ R}^{\SHI n}}$ has for $R-$basis the $T_\ql*T_w$ (resp. $X^\ql*T_w$, $Z^\ql*H_w$) for $\ql\in Y^n$ (in $Y^0_c$ if $n=0$) and $w\in W^v$.

$$
{^{I}\widetilde  \shh_{ R}^{\SHI }}=\bigoplus_{n{\geq}0}\,{^{I}\widetilde  \shh_{ R}^{\SHI n}},
$$ where ${^{I} \widetilde \shh_{ R}^{\SHI n}}$ has for $R-$basis the $T_\qo*T_\ql*T_w$ (resp. $T_\qo*X^\ql*T_w$, $T_\qo*Z^\ql*H_w$) for $\qo\in\QO$, $\ql\in Y^n$ ($Y^0_c$ if $n=0$) and $w\in W^v$.

$$
{^{BL} \shh}_{ R_1}=\bigoplus_{n\in\Z}\,{^{BL} \shh}_{ R_1}^{ n},
$$ where ${^{BL} \shh}_{ R_1}^{ n}$ has for $R_1-$basis the $Z^\ql H_w$ for $\ql\in Y^n$  and $w\in W^v$.

$$
{^{BL}\widetilde \shh}_{\widetilde R}=\bigoplus_{n\in\Z}\,{^{BL}\widetilde \shh}_{\widetilde R}^{n},
$$ where ${^{BL}\widetilde \shh}_{\widetilde R}^{n}$ has for $\widetilde R-$basis the $T_\qo Z^\ql H_w$ for $\qo\in\QO$, $\ql\in Y^n$  and $w\in W^v$.

  \par These gradations are compatible with the identifications explained in \ref{5.3} or \ref{7.7}.

  \par 3) For any $\widetilde C_x\in\widetilde\SHC_0^+$ and any $\ql\in\ Y_c^0=\Z\ql_c$, there is a unique $\widetilde C_y\in\widetilde\SHC_0^+$ with $d^W(\widetilde C_x,\widetilde C_y)=\ql$:
  the translation by $\ql$ in $\A$ stabilizes all enclosed sets and extends to $\SHI$ as a translation in any apartment.
  From this we see that $T_\ql*T_\qm=T_{\ql+\qm}=T_\qm*T_\ql$ (for $\qm\in Y^+$), $T_\ql*X^\qm=X^{\ql+\qm}=X^\qm*T_\ql$ (for $\qm\in Y$) and $T_\ql*T_w=T_{\ql .w}=T_{w.\ql}=T_w*T_\ql$ (for $w\in W^v$).
  Such a $T_\ql$ is central and invertible in ${^{I} \shh_{ R}^{\SHI }}$, ${^{I}\widetilde  \shh_{ R}^{\SHI }}$, ${^{BL} \shh}_{ R_1}$ or ${^{BL}\widetilde \shh}_{\widetilde R}$.

  \par Actually ${^{I} \shh_{ R}^{\SHI 0}}$  is the tensor product $R[Y^0_c]\otimes_R\shh_R(W^v)$ with a direct multiplication (factor by factor) and ${^{I}\widetilde  \shh_{ R}^{\SHI 0}}=R[Y^0_c]\otimes_R(R[\QO]\otimes_R\shh_R(W^v))$ with a semi-direct multiplication.

   \subsection{The double affine Hecke algebra} \label{7.9}

\par The subalgebra ${^{BL}\widetilde \shh}_{\widetilde R}^{0}$ is well known as the Cherednik's double affine Hecke algebra (DAHA).
More precisely in \cite{Che92} and \cite{Che95}, Cherednik considers an untwisted affine root system, as in \cite[Ch. 7]{K90}; but, as he works with roots instead of coroots, we write $\QF^\vee$ this system.
He considers the case where $\widetilde W^v$ is the full extended Weyl group ($\widetilde W^v=W_0^v\ltimes P_0^\vee$ with the notations of \ref{7.2}) \ie $\QO\simeq P_0^\vee/Q_0^\vee$ acts on the extended Dynkin diagram, simply transitively on its ``special'' vertices.
His choice for $Y^0$ is $Y^0=\Z.(1/m).\g c\oplus P_0^\vee\subset P^\vee$ (and $Y=Y^0\oplus\Z d$ {\it e.g.}), where $m\in\Z_{{\geq}1}$ is suitably chosen.
He then defines the DAHA as an algebra over a field of rational functions $\C({{\underline\qd}},(q_\qn)_{\qn\in\qn_R})$ with generators $(T_i)_{i\in I}$, $(X_\qb)_{\qb\in P_0^\vee}$ and some relations.
It is easy to see that this DAHA is, up to scalar changes, a ring of quotients of our ${^{BL}\widetilde \shh}_{\widetilde R}^{0}$ (for $\A,\widetilde W$ as described above):
actually ${\underline\qd}$ stands for our $Z^{\ql_c}$.
Here is a partial dictionary to translate from \cite{Che92} and \cite{Che95} to our article: roots $\leftrightarrow$ coroots, $X_\qb\mapsto Z^\qb$, $T_i\mapsto H_i$, $q_i\mapsto \qs_i$, $\QP\mapsto\QO$, $\qp_r\mapsto T_\qo$, ${\underline\qd}\mapsto T_{\ql_c}$ and $\underline\QD=\underline\qd^m\mapsto T_{\g c}$.

\par In  \cite{Che92} there is another presentation of the same DAHA using the Bernstein presentation of $\shh_R(W^v)$.
This is also the point of view of \cite{Ma03}, where the framework is more general.


\medskip

\noindent Universit\'e de Lyon, Institut Camille Jordan (UMR 5208)\\
Universit\'e Jean Monnet, Saint-Etienne, F-42023, France

E-mail: Stephane.Gaussent@univ-st-etienne.fr
\medskip
\par
\noindent Universit\'e de Lorraine, Institut \'Elie Cartan de Lorraine, UMR 7502, and \\
CNRS, Institut \'Elie Cartan de Lorraine, UMR 7502,
\\Vand\oe uvre l\`es Nancy, F-54506, France

E-mail: Nicole.Panse@univ-lorraine.fr ; Guy.Rousseau@univ-lorraine.fr

\bigskip
\noindent {\bf Acknowledgement:} The second author acknowledges support of the ANR grants ANR-2010-BLAN-110-02 and ANR-13-BS01-0001-01.

\end{document}